\documentclass[a4paper,11pt,3p]{article}
\usepackage[margin=1in]{geometry}
\usepackage{amsfonts,amsmath,amssymb,mathrsfs,amsthm,amscd,enumitem,multicol,multienum}
\usepackage{graphics,multirow}
\usepackage{algorithm}
\usepackage{tabularx}
\usepackage{chemfig}
\usepackage[version=4]{mhchem}
\usepackage{bm}
\usepackage{latexsym}
\usepackage[T1]{fontenc}
\usepackage[utf8]{inputenc}
\usepackage[colorlinks=true, linkcolor=black, anchorcolor=black, citecolor=black, filecolor=black, urlcolor=black]{hyperref}

\newtheorem{thm}{Theorem}[section]
\newtheorem{lemma}[thm]{Lemma}

\newtheorem{defin}[thm]{Definition}
\newtheorem{rem}[thm]{Remark}

\numberwithin{equation}{section}
\renewcommand{\epsilon}{\varepsilon}

\begin{document}
	\begin{center}
		{\large Limiting analysis of a crystal dissolution and precipitation model coupled with the unsteady stokes equations in the context of porous media flow }

		\bigskip
		
		\bigskip
	\end{center}
	\begin{minipage}[h]{0.45\columnwidth}
		\begin{center}
			Nibedita Ghosh\\
			{\small Department of Mathematics, \\IIT Kharagpur}\\
			{\small WB 721302, India}\\
			{\small e-mail: nghosh.iitkgp@gmail.com}
		\end{center}
	\end{minipage}\hspace{1cm}
	\begin{minipage}[h]{0.45\columnwidth}
		\begin{center}
			Hari Shankar Mahato\\
			{\small Department of Mathematics, \\IIT Kharagpur}\\
			{\small WB 721302, India}\\
			{\small e-mail: hsmahato@maths.iitkgp.ac.in}
		\end{center}
	\end{minipage}

	\bigskip
	\hrule
	
	\bigskip
	
	\textbf{Abstract.}  We study the diffusion-reaction-advection model for mobile chemical species together with the dissolution and precipitation of immobile species in a porous medium at the micro-scale. This leads to a system of semilinear parabolic partial differential equations in the pore space coupled with a nonlinear ordinary differential equation at the grain boundary of the solid matrices. The fluid flow within the pore space is given by unsteady Stokes equation. The novelty of this work is to do the iterative limit analysis of the system by tackling the nonlinear terms, monotone multi-valued dissolution rate term, space-dependent non-identical diffusion coefficients and nonlinear precipitation (reaction) term. We also establish the existence of a unique positive global weak solution for the coupled system.  In addition to that, for upscaling we introduce a modified version of the extension operator. Finally, we conclude the paper by showing that the upscaled model admits a unique solution.
	
	\bigskip
	{\textbf{Keywords:} }  reactive transport in porous media, diffusion-reaction-advection systems, unsteady Stokes equations, crystal dissolution and precipitation, existence of solution, periodic homogenization. \\

	{\textbf{AMS subject classifications: }} 35K57, 35K55, 35B40, 35B27, 35K91, 76M50, 76S05, 47J35
	\bigskip
	\hrule
	
	\section{Introduction}
	Transport through porous media is encountered in several engineering and biological applications, e.g., \cite{bear1991introduction, levenspiel1998chemical, logan2001transport, missen1999introduction, rubin1983transport, saaf1997study}. Usually, the solute transport in the pore is modelled by diffusion (\textit{Fick's law}), dispersion and advection whereas the interaction amongst the solutes is given by \textit{mass action kinetics/law}. On the other hand, the activities of minerals (crystals) on the interfaces are given by dissolution and precipitation. Recently, several works have been done on dissolution and precipitation of minerals, e.g., \cite{knabner1986free, van1996crystal, knabner1995analysis, willis1987transport, van2007crystal, van2009crystal,showalter1997microstructure,peter2008different, peter2009multiscale} and references therein. In this paper, we shall consider a reversible reaction of two mobile species $A$ and $B$ and an immobile species $C$ connected via 
	\begin{equation}\label{eq2.5}
		A+B\rightleftharpoons C.
	\end{equation}
	We note that \eqref{eq2.5} occurs on the interface of the solid parts in a porous medium. \eqref{eq2.5} type reactions are very common in problems on concrete carbonation, sulfate attack in sewer pipes, understanding of the dynamics of hematopoietic stem cells (HSCs), leaching of saline soil etc.
	In all the aforementioned references, the rate of precipitation is modelled by mass action law and the rate of dissolution is taken to be constant in such a way that when there is no mineral on the interface then the dissolution rate is zero and if there is mineral present then it is positive. In this work,
	we use nonlinear \textit{``Langmuir kinetics"} to model the surface reaction (precipitation) phenomena and the dissolution process is described by a discontinuous multivalued term. The concept of introducing the multivalued dissolution rate function has been explored in \cite{knabner1995analysis,van1996crystal,van2004crystal,van2009crystal}. The idea of using the \textit{``Langmuir kinetics"} comes from \cite{allaire2016upscaling,cardone2019homogenization}. In these papers, the authors have considered Langmuir kinetics for single chemical species to describe the flux boundary condition on the interface between the solids and the pore space.   In our case, we apply it for two different mobile chemical species $A$ and $B$ and an immobile species $C$ on the interface. 
	
	
	All in all, we obtained a system of semilinear parabolic (diffusion-reaction) equations coupled with a nonlinear ordinary differential equations (precipitation-dissolution). A multispecies diffusion-reaction model with same diffusion coefficients is proposed in 
	\cite{krautle2011existence,mahato2013homogenization}, where the authors established the existence of a global weak solution 
	by incorporating a suitable Lyapunov functional and Schaefer's fixed point theorem to obtain an $L^\infty $ a-priori estimates.
	The existence theory in \cite{fischer2015global} is obtained by considering a renormalized solution to circumvent the question of $L^\infty-$boundness. 
	The model presented in \cite{bothe2017global} dealt with the adsorption reaction at the surface of the solid parts and then the existence of a global solution is proved. However, such a model with smooth rates is not suited to represent the case of precipitation and dissolution of minerals in real world situations. We consider the multi-valued discontinuous rate function to describe the dissolution process, cf. \cite{van1996crystal,van2004crystal}. The proof of the existence of a solution for the case of one single kinetic reaction with one mobile and one immobile can be found in \cite{knabner1986free}. For a multispecies diffusion-reaction-dissolution system with identical diffusion coefficients, an upscaled model is proposed in \cite{mahato2015existence}. The difficulties to prove the existence of a global in time solution for a system of diffusion-reaction equations with different diffusion coefficients have been elaborated in the survey paper \cite{pierre2010global}. Several partial results in this direction with certain restrictions on source term can be found in \cite{bothe2017global1, bothe2010quasi, bothe2015global, hoffmann2017existence}. A model describing processes of diffusion, convection and nonlinear reaction in a periodic array of cells can be found in \cite{hornung1994reactive}. In that article, the convergence of the nonlinear terms is achieved by utilizing their monotonicity. Homogenization of models of chemical reaction flows in a domain with periodically distributed reactive solid grains was studied by  Conca et al. \cite{conca2004homogenization}. They considered a stationary diffusion-reaction model with nonlinear, fast growing but monotone kinetics on the surface of solid grains and a model of diffusion-reaction processes both inside and outside of grains.

	Moreover, as we performed the homogenization of the micro model via asymptotic expansion, we at first removed the multivaluedness of the dissolution
	term by Lipschitz regularization. We noticed that if we pass the regularization parameter $\delta\rightarrow 0$ first to this term then it will become multi-valued and discontinuous again. Therefore, we are unable to apply Taylor series expansion and so we can not pass the homogenization limit as $\varepsilon\to 0$. This indicates that we have to pass the homogenization limit first, otherwise, we will have a multivalued discontinuous term whose expansion is not possible. This motivates us to do the iterative limit process and check whether the final outcome in both cases coincides with each other or not. 
	
	
	
	The structure of the paper is as follows: we start off in section 2 with the introduction of the periodic setting of the domain and the microscale model equations. In section 3, we collect few mathematical tools required to analyze the model from section 2. Next, we obtain the a-priori estimates of the solution needed to pass the limit in section 3. Further, we employ Rothe's method to show the existence of a unique global in time weak solution for the regularized system. Section 4 and section 5 contains the proof of the two main theorems of the paper. It include iterative limit process. Since we have two limits to pass one is the regularization parameter $\delta$ and another one is the homogenization limit $\epsilon$. So which one we should pass first and what will be the final outcome if we reverse our limit process these questions are also addressed in this manuscript.

	\section{Setting of the Problem}  
	We consider a (bounded) porous medium $\Omega\subset \mathbb{R}^{n_{|n\ge 2}}$ which consists of a pore space $\Omega^p$ and the union of solid parts $\Omega^s$ in such a way that $\Omega:=\Omega^{p}\cup\Omega^{s}$ where $\Bar{\Omega}^{s}\cap \Omega^{p}=\phi$. The exterior boundary of the domain is denoted by $\partial\Omega$ and $\Gamma^*$ represents the union of boundaries of solid parts. 
	\begin{figure}[h!]
		\begin{center}
			\includegraphics[width=10cm, height=7cm]{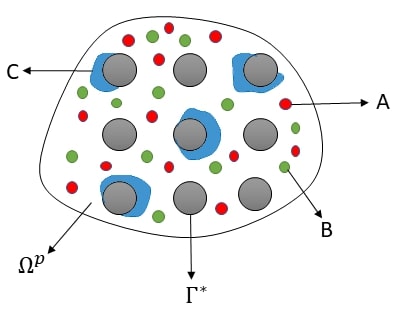}
			\caption{Mobile species in $\Omega^p$ with crystal dissolution and precipitation on $\Gamma^*$}
		\end{center}
	\end{figure} 
	Let $Y := (0,1)^n\subset \mathbb{R}^{n}$ be the unit periodicity cell which is composed of a solid part $Y^{s}$ with boundary $\Gamma$ and a pore part $Y^{p}$ such that $Y=Y^{s}\cup Y^{p}$, $\bar{Y}^{s}\subset Y$ and $\bar{Y}^{s}\cap\bar{Y}^{p}=\Gamma$. For each multi-index $k\in \mathbb{Z}^{n}$, we define the shifted sets as $Y_{k}:=Y+k$ and $Y^{\lambda}_{k}:=Y^{\lambda}+k$ for $\lambda \in \{ p,s\}$ and $\Gamma_{k}:=\Gamma+k$. Let $\Omega\subset \cup_{k\in \mathbb{Z}^n}Y_k$, which defines the periodicity assumption on $\Omega$. Let $0<\varepsilon\ll 1$ be the scaling parameter such that $\Omega$ is covered by the finite union of translated version of $\varepsilon Y_{k}$ cells, i.e. $\Omega \subset \underset{k\in \mathbb{Z}^{n}}{\cup} \varepsilon Y_{k}$, where $\varepsilon Y_{k}\subset \Omega$ for $k\in \mathbb{Z}^{n}$.
	We also define $\Omega^{p}_{\varepsilon}:= \text{int}\cup_{k\in \mathbb{Z}^{n}}\left\{\varepsilon Y^{p}_{k}:\varepsilon Y^{p}_{k}\subset \Omega\right\}$- pore space, $\Omega^{s}_{\varepsilon}:= \cup_{k\in \mathbb{Z}^{n}}\left\{\varepsilon Y^{s}_{k}:\varepsilon Y^{s}_{k}\subset \Omega\right\}$ - union of solid surfaces, $\Gamma_{\varepsilon}^*:= \cup_{k\in \mathbb{Z}^{n}}\left\{\varepsilon \Gamma_{k}:\varepsilon \Gamma_{k}\subset \Omega\right\}$ - union of interfaces between pore space and solid surfaces, $\partial \Omega _{\varepsilon}^{p}:=\partial \Omega \cup \Gamma^*_{\varepsilon}$ - union of outer boundary and interface $\Gamma_{\varepsilon}^*$, cf. FIG. 1. Let $S:=[0,T)$ be the time interval for $T>0$. We also denote the volume elements in $Y$ and $\Omega$ as $dy, dx$ and the surface elements on $\Gamma$ and $\Gamma_\epsilon^*$ as $d\sigma_y$ and $d\sigma_x$. The characteristic (indicator) function, $\chi_{\varepsilon}(x)=\chi(\frac{x}{\varepsilon})$, of $\Omega_{\varepsilon}^{p}$ in $\Omega$ is defined by $\chi_\epsilon(x)=1 \textnormal{ for } x\in \Omega_\epsilon^p$ and $=0 \text{ for } x\in \Omega\setminus\Omega_{\varepsilon}^{p}$.
	
	Let the pore space $\Omega_\epsilon^p$ is filled by some fluid with a unknown fluid velocity $\vec{\textbf{q}_\epsilon}=\textbf{q}_\epsilon(t, x)$, where $(t, x)\in S\times\Omega_\epsilon^p$. In our setting, the flow geometry and the fluid viscosity $ (\mu>0)$ are not affected by the surface reaction. We impose a no slip boundary condition on the grain boundary. We describe the fluid flow by evolutionary Stokes equations to make the situation more realistic where $p_\epsilon$ represents the fluid pressure. Now along with that two mobile species $A$ (type I) and $B$ (type II) are present in $\Omega_{\varepsilon}^{p}$ and one immobile species $C$ is present on $\Gamma_{\varepsilon}^*$. The species $A$, $B$ and $C$ are connected via \eqref{eq2.5}.
	In our situation, there is no reaction happening amongst the mobile species in $\Omega^p_\epsilon$. We decompose the outer boundary as $\partial\Omega=\partial\Omega_{in}\cup\partial\Omega_{out}$, where on $\partial\Omega_{in}$ and $\partial\Omega_{out}$ we prescribe the inflow and outflow boundary conditions for the type-I and type-II mobile spices, respectively. 
	As by (\ref{eq2.5}) $A$ and $B$ are supplied by the dissolution process on $\Gamma_\epsilon^*$, therefore, the flux condition for $A$ and $B$ on $\Gamma_\epsilon^*$ is equal to the rate of change of immobile species on $\Gamma_\epsilon^*$ which means we have an additional boundary condition due to activities on $\Gamma_\epsilon^*$. According to the relation (\ref{eq2.5}) one molecule of each $A$ and $B$ will give one molecule of $C$ which will be modelled by the surface reaction rate term coming from \textit{Langmuir kinetics}. On other hand, one molecule of $C$ will dissolve to give one molecule of each $A$ and $B$. The modeling of dissolution process is adopted from \cite{knabner1995analysis,van1996crystal,van2004crystal}. In case of dissolution at the surface of the solids, it is assumed to be constant if the mineral is present. If the mineral is absent, dissolution rate can not be stronger than precipitation in order to maintain the non-negativity of the surface concentration. This leads to a multivalued dissolution term $R_{d}(w_{\varepsilon})\in k_{d}\psi(w_{\varepsilon})$, where
	\begin{align}\label{1.15}
		\psi (c)&\;\;=\;\;
		\left\{
		\begin{aligned}
			&\{0\} && \textnormal{if }c< 0,\\
			&[0,1]&& \textnormal{if }c=0,\\
			&\{1\} && \textnormal{if }c>0.
		\end{aligned}\right.
	\end{align}
	Let the concentrations of $A$, $B$ and $C$ be given by $u_\varepsilon$, $v_\varepsilon$ and $w_\varepsilon$, respectively. Then, the unsteady Stokes equations and the mass-balance equations for $A, B$ and $C$ are given by
	\begin{subequations}
		\begin{align}
			\frac{\partial \textbf{q}_\epsilon}{\partial {t}}-\epsilon^2\mu\Delta \textbf{q}_\epsilon&=-\nabla p_\epsilon \;\text{ in } S\times\Omega^p_\epsilon,\label{eqn:v1}\\
			\nabla.\textbf{q}_\epsilon&=0 \;\text{ in } S\times\Omega^p_\epsilon, \label{eqn:v2}\\
			\textbf{q}_\epsilon&=0 \; \text{ on } S\times\partial\Omega_\epsilon^p,  \label{eqn:v3}\\
			\textbf{q}_\epsilon(0,x)&=\textbf{q}_0(x) \; \text{ in } \Omega_\epsilon^p, \label{eqn:v4}
		\end{align}
	\end{subequations}\vspace{-1.0cm}
	\begin{subequations}
		\begin{align}
			\frac{\partial u_\epsilon}{\partial {t}} -\nabla.({\bar{D}^\epsilon_1}\nabla u_\epsilon-\textbf{q}_\epsilon u_\epsilon)&=0\; \text{ in } S\times \Omega_\epsilon^p,\label{eqn:M11}\\
			-({\bar{D}^\epsilon_1}\nabla u_\epsilon-\textbf{q}_\epsilon u_\epsilon).\vec{n}&=d(t,x)\; \text{ on } S\times \partial\Omega_{in},\label{eqn:M12}\\
			-({\bar{D}^\epsilon_1}\nabla u_\epsilon-\textbf{q}_\epsilon u_\epsilon).\vec{n}&=e(t,x)\; \text{ on } S\times \partial\Omega_{out},\label{eqn:M13}\\
			-({\bar{D}^\epsilon_1}\nabla u_\epsilon-\textbf{q}_\epsilon u_\epsilon).\vec{n}&=\epsilon{\frac{\partial w_\epsilon}{\partial {t}}}\; \text{ on } S\times \Gamma^*_\epsilon,\label{eqn:M14}\\
			u_\epsilon(0,x)&=u_0(x)\; \text{ in }\Omega^p_\epsilon,\label{eqn:M15}
		\end{align}
	\end{subequations}\vspace{-0.95cm}
	\begin{subequations}
		\begin{align}
			\frac{\partial v_\epsilon}{\partial {t}} - \nabla.({\bar{D}^\epsilon_2}\nabla v_\epsilon-\textbf{q}_\epsilon v_\epsilon)&=0 \; \text{ in } S\times \Omega_\epsilon^p,\label{eqn:M21}	\\
			- \nabla.({\bar{D}^\epsilon_2}\nabla v_\epsilon-\textbf{q}_\epsilon v_\epsilon).\vec{n}&=g(t,x) \;\text{ on } S\times \partial\Omega_{in},\label{eqn:M22}\\
			- \nabla.({\bar{D}^\epsilon_2}\nabla v_\epsilon-\textbf{q}_\epsilon v_\epsilon).\vec{n}&=h(t,x) \;\text{ on } S\times \partial\Omega_{out},\label{eqn:M23}\\
			- \nabla.({\bar{D}^\epsilon_2}\nabla v_\epsilon-\textbf{q}_\epsilon v_\epsilon).\vec{n}&=\epsilon{\frac{\partial w_\epsilon}{\partial {t}}} \;\text{ on } S\times \Gamma^*_\epsilon,\label{eqn:M24}\\
			v_\epsilon(0,x)&=v_0(x)\; \text{ in }\Omega^p_\epsilon,\label{eqn:M25}
		\end{align}
	\end{subequations}	\vspace{-0.95cm}
	\begin{subequations}
		\begin{align}
			\frac{\partial w_\epsilon}{\partial  {t}}&=k_d(R(u_\epsilon,v_\epsilon)-z_\epsilon) \; \text{ on } S\times  \Gamma^*_\epsilon, \label{eqn:MN1}\\
			z_\epsilon&\in\psi(w_\epsilon) \; \text{ on }S\times \Gamma^*_\epsilon,\label{eqn:MN2}\\
			w_\epsilon(0,x)&=w_0(x)\; \text{ in }\Gamma_\epsilon^*,	\label{eqn:MN3}
		\end{align}
	\end{subequations}
	where $R : \mathbb{R}^2 \rightarrow [0, \infty)$ is defined by
	\begin{align*}
		R(u_\epsilon, v_\epsilon)&\;\;=\;\;
		\left\{
		\begin{aligned}
			&k\frac{k_1u_\epsilon k_2 v_\epsilon}{(1+k_1u_\epsilon+k_2v_\epsilon)^2}&& \textnormal{if } (u_\epsilon, v_\epsilon)\in [0,\infty)^2,\\
			&0&&\textnormal{otherwise }
		\end{aligned}\right.
	\end{align*} 
	and $k=\frac{k_f}{k_d}$. $k_f$ and $k_d$ are the forward and backward reaction rate constants, and $k_1$ and $k_2$ are the \textit{Langmuir parameters} for the mobile species $A$ and $B$, respectively. Here we considered \textit{Langmuir kinetics} for the surface reaction and the rate of dissolution $R_d=k_d\psi(w_\epsilon).$ We note that the diffusion coefficients are nonconstant and non-periodic, i.e. $\bar{D}^\epsilon_i=diag(D_i(x,\frac{x}{\epsilon}), D_i(x,\frac{x}{\epsilon}),. . ., D_i(x,\frac{x}{\epsilon}))$, $i\in\{1,2\}$ it depends on the space variables. We denote the problem $\eqref{eqn:v1} - (\ref{eqn:MN3})$ by $(\mathcal{P}_{\epsilon})$. For $(\mathcal{P}_{\epsilon})$, we would look for the existence of a unique global weak solution and then we upscale the model from micro scale to macro scale. These two tasks are being formulated into two major theorems of this manuscript in section 3. 
	\section{Analysis of the microscopic model ($\mathcal{P}_\epsilon$)}
	\subsection{Function Space Setup}
	Let $\theta \in [0,1]$ and $p>n+2$ be such that $\frac{1}{p}+\frac{1}{q}=1$. Note that $p$ as the superscript in $\Omega^{p}_{\varepsilon}$ is used to signify the pore space and should not be confused with the exponents of the function spaces defined here. Assume that $\Xi\in \{\Omega, \Omega^{p}_{\varepsilon}\}$, then as usual $L^{p}(\Xi)$, $H^{1,p}(\Xi)$, $C^{\theta}(\bar{\Xi})$, $(\cdot,\cdot)_{\theta,p}$ and $[\cdot,\cdot]_{\theta}$ are the Lebesgue, Sobolev, H\"older, real- and complex-interpolation spaces respectively endowed with their standard norms, for details see \cite{evans2010partial}. We denote $\|f\|^p_{\Xi}=\int_{\Xi}|f|^p\,dx $ and $\|f\|^p_{(\Xi)^T}=\int_{S\times\Xi}|f|^p\,dx\,dt $. Further, $C_0^\infty(S\times\Omega ; C_{per}^\infty(Y))$ is the space of all $Y-$periodic continuously differentiable functions in $t, x, y$ with compact support inside $\Omega$. In particular, $C^{\gamma}_{per}(Y)$ denotes the set of all \textit{Y$-$periodic $\gamma-$times} continuously differentiable functions in $y$ for $\gamma\in \mathbb{N}$. For a Banach space $X$, $X^{*}$ denotes its dual and the duality pairing is denoted by $\langle.\; ,\; .\rangle_{X^{*}\times X}$. The symbols $\hookrightarrow$, $\hookrightarrow\hookrightarrow$ and $\underset{\hookrightarrow}{d}$ denote the continuous, compact and dense embeddings, respectively. Further, we use the symbols $\rightarrow, \rightharpoonup, \overset{2}{\rightharpoonup}$ to represent the strong, weak and two-scale convergence of a sequence respectively. We define $L^{p}(\Xi)\hookrightarrow H^{1,q}(\Xi)^{*}$ as 
	\begin{align*}
		\langle f, v\rangle_{H^{1,q}(\Xi)^{*}\times H^{1,q}(\Xi)}=\langle f,v\rangle_{L^{p}(\Xi)\times L^{q}(\Xi)}:=\int_{\Xi}f v\,dx \textnormal{ for }f\in L^{p}(\Xi),\; v\in H^{1,q}(\Xi).
	\end{align*}
	By section 2, we note that $\lim_{\epsilon \to 0}\varepsilon |\Gamma_{\varepsilon}^*|=|\Gamma|\frac{|\Omega|}{|Y|}$. Since the surface area of $\Gamma_{\varepsilon}^*$ increases proportionally to $\frac{1}{\varepsilon}$, i.e. $|\Gamma_{\varepsilon}^*|\to \infty$ as $\varepsilon\to 0$, we introduce the $L^{p}(\Gamma_{\varepsilon}^*)-L^{q}(\Gamma_{\varepsilon}^*)$ duality as
	\begin{align*}
		\langle \zeta_{1},\zeta_{2}\rangle:=\varepsilon\int_{\Gamma_{\varepsilon}^*}\zeta_{1}\zeta_{2}\,d\sigma_{x}\text{ for }\zeta_{1}\in L^{p}(\Gamma_{\varepsilon}^*), \; \zeta_{2}\in L^{q}(\Gamma_{\varepsilon}^*)\label{2.2}
	\end{align*}
	and the space $L^{p}(S \times \Gamma_{\varepsilon}^*)$ is furnished with the norm
	\begin{align}
		\|\zeta\|^{p}_{(\Gamma_{\varepsilon}^*)^T}&\;\;:=\;\;
		\left\{
		\begin{aligned}
			&\varepsilon\int_{S\times \Gamma_{\varepsilon}^*}|\zeta(t,x)|^{p}\,d\sigma_{x} \,dt&& \textnormal{for } 1\le p <\infty,\\
			&\underset{(t,x) \in S\times \Gamma_{\varepsilon}^*}{\text{ess sup}}{|\zeta(t,x)|}&&\textnormal{for } p=\infty.
		\end{aligned}\right.
	\end{align}     
	We denote $\textbf{L}^2(\Omega_\epsilon^p)= (L^2(\Omega_\epsilon^p))^n,  \textbf{H}^{1,2}(\Omega_\epsilon^p) = (H^{1,2}(\Omega_\epsilon^p))^n$, likewise $\textbf{H}_0^{1,2}(\Omega_\epsilon^p) = (H_0^{1,2}(\Omega_\epsilon^p))^n$. The space of divergence free vector fields is denoted by $\textbf{H}^{1, 2}_{div}(\Omega_\epsilon^p)=\{\zeta \in \textbf{H}^{1, 2}_0(\Omega_\epsilon^p) : \nabla.\zeta=0 \}$. The corresponding dual space is $\textbf{H}^{-1, 2}_{div}(\Omega_\epsilon^p)$. The space $L^2_0(\Omega_\epsilon^p)=\{\zeta\in L^2(\Omega_\epsilon^p) : \int_{\Omega_\epsilon^p}\zeta dx=0 \}$.

	We take $p=2$. The \textit{Sobolev-Bochner spaces} are given by: $\mathcal{Q}_\epsilon :=\{\textbf{q}_\epsilon\in L^2(S; \textbf{H}^{1, 2}_{div}(\Omega_\epsilon^p) ) : \frac{\partial q_\epsilon}{\partial t}\in L^2(S ; \textbf{H}^{-1, 2}_{div}(\Omega_\epsilon^p))\}, \mathcal{U}_\epsilon :=\{u_\epsilon\in L^2(S ; H^{1,2}(\Omega_\epsilon^p)):\frac{\partial u_\epsilon}{\partial t}\in L^2(S ; H^{1,2}(\Omega_\epsilon^p)^*)\}:= H^{1,2}(S ; H^{1,2}(\Omega_\epsilon^p)^*) \cap L^2(S ; H^{1,2}(\Omega_\epsilon^p)), \mathcal{V}_\epsilon :=\{v_\epsilon\in L^2(S ; H^{1,2}(\Omega_\epsilon^p)) : \frac{\partial v_\epsilon}{\partial t}\in L^{2}(S ; H^{1,2}(\Omega_\epsilon^p)^*)\}:= H^{1,2}(S ; H^{1,2}(\Omega_\epsilon^p)^*) \cap  L^2(S ; H^{1,2}(\Omega_\epsilon^p)), \mathcal{W}_\epsilon :=\{w_\epsilon\in L^2(S ; L^2(\Gamma^*_\epsilon)) : \frac{\partial w_\epsilon}{\partial t}\in L^2(S ; L^2(\Gamma^*_\epsilon)) \}:= H^{1,2}(S ; H^{1,2}(\Gamma^*_\epsilon)), 	\mathcal{Z}_\epsilon :=\{z_\epsilon\in L^\infty(S\times\Gamma^*_\epsilon)\,: 0\leq z_\epsilon \leq1\}, $
	where $\frac{\partial}{\partial t}$ is the distributional time derivative and their respective norms can be defined as 
	\begin{align} 
		&\|u\|_{\mathcal{U}_\epsilon}:=\|u\|_{L^{2}(S; H^{1,2}(\Omega_\epsilon^p))}+\left\lVert\frac{\partial u}{\partial t}\right\rVert_{L^{2}(S; H^{1,2}(\Omega_\epsilon^p)^{*})}.\nonumber
	\end{align}
	\subsection{Weak formulation and assumptions on data}
	The vector-valued function space $\mathcal{Q}_\epsilon\times\mathcal{U}_\epsilon\times\mathcal{V}_\epsilon\times\mathcal{W}_\epsilon\times\mathcal{Z}_\epsilon$ is the solution space. We say that  $(\textbf{\textbf{q}}_\epsilon, u_\epsilon, v_\epsilon, w_\epsilon, z_\epsilon)\in \mathcal{Q}_\epsilon\times\mathcal{U}_\epsilon\times\mathcal{V}_\epsilon\times\mathcal{W}_\epsilon\times\mathcal{Z}_\epsilon$ is a weak solution of the problem $(\mathcal{P}_\epsilon)$ if $(\textbf{\textbf{q}}_\epsilon(0), u_\epsilon(0), v_\epsilon(0), w_\epsilon(0))=(\textbf{q}_0, u_0, v_0, w_0)\in \textbf{L}^2(\Omega_\epsilon^p)\times L^2(\Omega_\epsilon^p)\times L^2(\Omega_\epsilon^p)\times L^2(\Gamma_\epsilon^*)$ and 
	\begin{subequations}
		\begin{align} 
			&\langle \frac{\partial \textbf{q}_\epsilon}{\partial t}, \psi\rangle_{(\Omega_\epsilon^p)^T}+\epsilon^2\mu\langle \nabla \textbf{q}_\epsilon, \nabla \psi\rangle_{(\Omega_\epsilon^p)^T}=0,\label{eqn:vw}\\
			&\langle\frac{\partial u_\epsilon}{\partial t},\phi\rangle_{(\Omega_\epsilon^p)^T}+\langle \bar{D}^\epsilon_1\nabla u_\epsilon-\textbf{q}_\epsilon u_\epsilon,\nabla\phi\rangle_{(\Omega_\epsilon^p)^T}=-\langle\frac{\partial w_\epsilon}{\partial t},\phi\rangle_{(\Gamma^*_\epsilon)^T}-\langle d,\phi\rangle_{(\partial\Omega_{in})^T}-\langle e,\phi\rangle_{(\partial\Omega_{out})^T},\label{eqn:M1WF}\\
			&\langle\frac{\partial v_\epsilon}{\partial t},\theta\rangle_{(\Omega_\epsilon^p)^T}+\langle \bar{D}^\epsilon_2\nabla v_\epsilon-\textbf{q}_\epsilon v_\epsilon,\nabla\theta\rangle_{(\Omega_\epsilon^p)^T}=-\langle\frac{\partial w_\epsilon}{\partial t},\theta\rangle_{(\Gamma^*_\epsilon)^T}-\langle g,\theta\rangle_{(\partial\Omega_{in})^T}-\langle h,\theta\rangle_{(\partial\Omega_{out})^T}, \label{eqn:M2WF}\\
			&\langle\frac{\partial w_\epsilon}{\partial t},\eta\rangle_{(\Gamma^*_\epsilon)^T}=k_d\langle R(u_\epsilon, v_\epsilon)-z_\epsilon, \eta\rangle_{(\Gamma^*_\epsilon)^T},\; z_\epsilon\in\psi(w_\epsilon)\text{ a.e. on }(\Gamma^*_\epsilon)^T, \label{eqn:M12WF}
		\end{align}
		for all $\psi\in L^2(S; \textbf{H}^{1,2}_{div}(\Omega_\epsilon^p))$ and $(\phi,\theta,\eta)\in L^2(S;H^{1,2}(\Omega_\epsilon^p))\times L^2(S;H^{1,2}(\Omega_\epsilon^p))\times L^2(S; L^2(\Gamma^*_\epsilon))$,
	\end{subequations}
	where 
	\begin{align*}
		&\langle u,v \rangle_{(\Omega_\epsilon^p)^T}=\int_{0}^{T}\int_{\Omega_\epsilon^p}uvdxdt\text{ and }\langle u,v \rangle_{(\Gamma^*_\epsilon)^T}=\epsilon\int_{0}^{T}\int_{\Gamma^*_\epsilon}uv d\sigma_xdt,\\
		&\|u\|^2_{(\Omega_\epsilon^p)^T}=\int_{0}^{T}\int_{\Omega_\epsilon^p}u^2dxdt\text{ and }\|u\|^2_{(\Gamma^*_\epsilon)^T}=\epsilon\int_{0}^{T}\int_{\Gamma^*_\epsilon}u^2 d\sigma_xdt.
	\end{align*}
	Moreover, for each weak solution $(\textbf{q}_\epsilon, u_\epsilon, v_\epsilon, w_\epsilon)$ we associate a pressure $p_\epsilon :=\partial_t P_\epsilon, P_\epsilon\in L^\infty (S; L^2_0(\Omega_\epsilon^p))$ which satisfies \eqref{eqn:v1} in the distributional sense.\\ 
	For a $\beta \in \mathbb{R}$, we use $\beta_{+} := max\{ \beta, 0 \}\ge 0 \; \text{ and } \beta_{-} := max\{ -\beta, 0 \}\ge 0,$ so it holds that $\beta=\beta_{+}-\beta_{-}$. We will make following assumptions for the sake of analysis:
	
	\begin{itemize}
		\item[\bf A1.] $u_0(x), v_0(x), w_0(x) \geq 0.$\qquad {\bf A2.}  $R(u_\epsilon, v_\epsilon) = 0$  for all  $u_\epsilon\leq 0, v_\epsilon\leq0.$
		\item[\bf A3.] $d, g \in L^2(S\times\partial\Omega_{in})$, $e, h \in L^2(S\times\partial\Omega_{out})$.\qquad {\bf A4.} $\textbf{q}_0\in \textbf{L}^2(\Omega), u_0, v_0\in H^{1,2}(\Omega)$ and $w_0\in L^\infty(\Omega)$. 
		\item [\bf A5.] $p_\epsilon\in L^2(S; H^{1,2}(\Omega_\epsilon^p))$ such that $\sup_{\epsilon>0} \|p_\epsilon\|_{L^2(S; H^{1,2}(\Omega_\epsilon^p))}<\infty$.
		\item[\bf A6.] $\bar{D}^\epsilon_i=diag(D_i(x,\frac{x}{\epsilon}), D_i(x,\frac{x}{\epsilon}),. . ., D_i(x,\frac{x}{\epsilon}))$, $i\in\{1,2\}$ such that
		\begin{itemize}
			\item for any $\zeta\in \mathbb{R}^n$, $(D_i(x,y)\zeta,\zeta)\ge \alpha|\zeta|^2$ where $\alpha>0$ is a constant independent of $\epsilon$.
			\item there exists a constant $M>0$ such that $\|\bar{D}^\epsilon_i(x,y)\|_{L^\infty(\Omega)^{n\times n}}\le M$ for every $y\in Y$
			\item[] and 
		\end{itemize} 
		\begin{align}\label{eqn:DC}
			\lim_{\epsilon\rightarrow 0}\int_{\Omega}D_i(x,\frac{x}{\epsilon})dx=\int_{\Omega}\int_{Y} D_i(x,y) dxdy \text{ for } i=1,2.
		\end{align}
	\end{itemize}
	Here the \textit{Langmuir} reaction rate term $R:\mathbb{R}^2\rightarrow[0,\infty)$ is locally Lipschitz in $\mathbb{R}^2$ such that 
	\begin{align}\label{eqn:LL}
		|R(u_\epsilon^1, v_\epsilon)-R(u^2_\epsilon, v_\epsilon)|\leq L_R|u_\epsilon^1-u^2_\epsilon|,  	
	\end{align}
	where $L_R>0$ is a constant. $L_R=\sup L_R(u_\epsilon, v_\epsilon)$ where $L_R(u_\epsilon,v_\epsilon)=kk_1k_2|v_\epsilon||(1+k_2v_\epsilon)^2+k_1^2u_\epsilon^1 u_\epsilon^2|$.
	
	\subsection{\textbf{A regularized problem}}
	As we can see in the system of equations $\eqref{eqn:v1} - \eqref{eqn:MN3}$, the multivaluedness of the dissolution rate term in \eqref{eqn:MN1} creates the main difficulty in showing the well-posedness of ($\mathcal{P}_\varepsilon$). To overcome that we introduce a regularization parameter $\delta>0$ and replace $\psi(w_\epsilon)$ by $\psi_\delta(w_\epsilon)$, where
	\begin{equation}\label{eqn:MDR}
		\psi_\delta(w_\epsilon)=\begin{cases}
			\text{0} \quad\text{\hspace{0.25cm}if $w_\epsilon < 0$},\\
			\frac{w_\epsilon}{\delta}\quad\text{if  $w_\epsilon\in(0,\delta$)},\\
			\text{1} \quad\text{\hspace{0.25cm}if $w_\epsilon > \delta$}.
		\end{cases}   
	\end{equation}
	The \textit{regularized problem}  ($\mathcal{P}_{\epsilon\delta}$) is given by:
			\textit{find $(\textbf{q}_\epsilon, u_{\epsilon\delta}, v_{\epsilon\delta}, w_{\epsilon\delta})\in\mathcal{Q}_\epsilon\times \mathcal{U}_\epsilon\times\mathcal{V}_\epsilon\times\mathcal{W}_\epsilon$ such that $(\textbf{q}_\epsilon(0), u_{\epsilon\delta}(0), v_{\epsilon\delta}(0), w_{\epsilon\delta}(0))=(\textbf{q}_0, u_0, v_0, w_0)$ and 
				\begin{subequations}
					\begin{align} 
						&\langle \frac{\partial \textbf{q}_\epsilon}{\partial t}, \psi\rangle_{(\Omega_\epsilon^p)^T}+\epsilon^2\mu\langle \nabla \textbf{q}_\epsilon, \nabla \psi\rangle_{(\Omega_\epsilon^p)^T}=0, \label{eqn:vrw}\\
						&\langle\frac{\partial u_{\epsilon\delta}}{\partial t},\phi\rangle_{(\Omega_\epsilon^p)^T}+\langle \bar{D}^\epsilon_1\nabla u_{\epsilon\delta}-\textbf{q}_\epsilon u_{\epsilon\delta},\nabla\phi\rangle_{(\Omega_\epsilon^p)^T}=-\langle\frac{\partial w_{\epsilon\delta}}{\partial t},\phi\rangle_{(\Gamma^*_\epsilon)^T}-\langle d,\phi\rangle_{(\partial\Omega_{in})^T}-\langle e,\phi\rangle_{(\partial\Omega_{out})^T}\label{eqn:M1RWF},\\
						&\langle\frac{\partial v_{\epsilon\delta}}{\partial t},\theta\rangle_{(\Omega_\epsilon^p)^T}+\langle \bar{D}^\epsilon_2\nabla v_{\epsilon\delta}-\textbf{q}_\epsilon v_{\epsilon\delta},\nabla\theta\rangle_{(\Omega_\epsilon^p)^T}=-\langle\frac{\partial w_{\epsilon\delta}}{\partial t},\theta\rangle_{(\Gamma^*_\epsilon)^T}-\langle g,\theta\rangle_{(\partial\Omega_{in})^T}-\langle h,\theta\rangle_{(\partial\Omega_{out})^T}  \label{eqn:M2RWF},\\
						&\langle\frac{\partial w_{\epsilon\delta}}{\partial t},\eta\rangle_{(\Gamma^*_\epsilon)^T}=k_d\langle R(u_{\epsilon\delta}, v_{\epsilon\delta})-\psi_\delta(w_{\epsilon\delta}), \eta\rangle_{(\Gamma^*_\epsilon)^T} \label{eqn:M12RWF},
					\end{align}
					for all $\psi\in L^2(S; \textbf{H}^{1,2}_{div}(\Omega_\epsilon^p))$ and $(\phi, \theta, \eta )  \in L^2(S;H^{1,2}(\Omega_\epsilon^p))\times L^2(S;H^{1,2}(\Omega_\epsilon^p)) \times L^2(S; L^2(\Gamma^*_\epsilon))$. As the fluid velocity $\textbf{q}_\epsilon$ does not depend on the concentrations of the chemical species.
			\end{subequations}}
	
	\vspace{0.1cm}
	\begin{rem}
		In the regularized problem, we have two limits to pass to zero, one is the regularization parameter $\delta$ and another one is the homogenization limit $\epsilon$.  Since it is now an iterative limit problem, which one to pass first is a matter of interest. We find out the iterative limit equations and show that the final macroscopic equations remains the same, so which one we are passing first does not matter in this context. This is somewhat now stated in the two main theorems of this paper which are given below.
	\end{rem}
	\begin{thm}\label{thm:1}
		There exists a unique positive global weak solution $(\textbf{q}_\epsilon, u_\epsilon,v_\epsilon,w_\epsilon,z_\epsilon)\in\mathcal{Q}_\epsilon\times\mathcal{U}_\epsilon\times\mathcal{V}_\epsilon\times\mathcal{W}_\epsilon\times\mathcal{Z}_\epsilon$ of the problem ($\mathcal{P}_\epsilon$) satisfying the a-priori estimate:
		\begin{align}\label{eqn:UB}
			&\|\textbf{q}_\epsilon\|_{L^\infty(S; \textbf{L}^2(\Omega_\epsilon^p))}+\sqrt{\mu}\epsilon \|\nabla\textbf{q}_\epsilon\|_{(L^2(S\times\Omega_\epsilon^p))^{n\times n}}+\left\lVert\frac{\partial \textbf{q}_\epsilon}{\partial t}\right\rVert_{L^2(S; \textbf{H}^{-1,2}_{div}(\Omega_\epsilon^p))}+\|u_{\epsilon}\|_{(\Omega_\epsilon^p)^T}+\|\nabla u_{\epsilon}\|_{(\Omega_\epsilon^p)^T} \notag\\
			&+\left\lVert \frac{\partial u_{\epsilon} }{\partial t}\right\rVert_{L^{2}(S; H^{1,2}(\Omega_\epsilon^p)^{*})}+\|v_{\epsilon}\|_{(\Omega_\epsilon^p)^T}+\|\nabla v_{\epsilon}\|_{(\Omega_\epsilon^p)^T}+\left\lVert \frac{\partial v_{\epsilon} }{\partial t}\right\rVert_{L^{2}(S; H^{1,2}(\Omega_\epsilon^p)^{*})}+\|w_{\epsilon}\|_{(\Gamma_\epsilon^*)^T}+\left\lVert\frac{\partial w_{\epsilon}}{\partial t}\right\rVert_{(\Gamma_\epsilon^*)^T}\le C,
		\end{align}
		where $C$ is a generic constant independent of $\epsilon$. Moreover, 	under the assumptions $(\textbf{A1}.)-(\textbf{A6}.)$, there exist $(\textbf{q}, u,v,w,z)\in [ (L^2(S\times\Omega; H^{1,2}_{per}(Y))^n\times          H^{1,2}(S ; H^{1,2}(\Omega)^*) \cap L^2(S ; H^{1,2}(\Omega))]^2\times H^{1,2}(S\times\Omega\times\Gamma)\times L^\infty(S\times\Omega\times\Gamma)$ such that $(\textbf{q}, u,v,w,z)$ is the unique solution of the (homogenized) problem
		\begin{align}\label{eqn:qs}
			\begin{split}
				\frac{\partial \textbf{q}}{\partial t}-\mu \Delta_y \textbf{q}&=-\nabla_yp_1(t, x, y)-\nabla_x p(t, x)\quad \text{in}\quad S\times\Omega\times Y^p,\\
				\nabla_y.\textbf{q}&=0 \quad \text{in} \quad S\times\Omega\times Y^p,\\
				\nabla_x.\bar{\textbf{q}}&=0 \quad \text{in}\quad S\times\Omega,\\
				\textbf{q}(t, x, y)&=0 \quad\text{in} \quad S\times\Omega\times\Gamma,\\
				\bar{\textbf{q}}.\vec{n}&=0 \quad\text{in}\quad S\times\partial\Omega,\\
				\textbf{q}(0,x) &=\textbf{q}_0(x)\quad\text{in}\quad\Omega,
			\end{split}			
		\end{align} \vspace{-0.6cm}
		\begin{align}\label{eqn:M1S}
			\begin{split}
				\frac{\partial u}{\partial {t}} - \nabla.(A\nabla u-(\bar{\textbf{q}}-\tilde{\textbf{q}}_0)u)+P(t,x)&=0 \quad \text{in}\quad S\times \Omega,	\\
				-(A\nabla u-(\bar{\textbf{q}}-\tilde{\textbf{q}}_0)u).\vec{n}&=\frac{d}{|Y^p|}\quad\text{on}\quad S\times \partial\Omega_{in},\\
				-(A\nabla u-(\bar{\textbf{q}}-\tilde{\textbf{q}}_0)u).\vec{n}&=\frac{e}{|Y^p|}\quad\text{on}\quad S\times \partial\Omega_{out},\\
				u(0,x)&=u_0(x)\quad\text{in}\quad\Omega, 
			\end{split}
		\end{align} \vspace{-0.5cm}
		\begin{align}\label{eqn:M2S}
			\begin{split}
				\frac{\partial v}{\partial {t}} - \nabla.(B\nabla v-(\bar{\textbf{q}}-\tilde{\textbf{q}}_0)v)+P(t,x)&=0 \quad \text{in}\quad S\times \Omega,	\\
				-(B\nabla v-(\bar{\textbf{q}}-\tilde{\textbf{q}}_0)v).\vec{n}&=\frac{g}{|Y^p|}\quad\text{on}\quad S\times \partial\Omega_{in},\\
				-(B\nabla v-(\bar{\textbf{q}}-\tilde{\textbf{q}}_0)v).\vec{n}&=\frac{h}{|Y^p|}\quad\text{on}\quad S\times \partial\Omega_{out},\\
				v(0,x)&=v_0(x)\quad\text{in}\quad\Omega,
			\end{split}
		\end{align}\vspace{-0.5cm}
		\begin{align}\label{eqn:M12S}
			\begin{split}
				\frac{\partial w}{\partial  {t}}&=k_d(R(u,v)-z)\quad\text{in}\quad S\times\Omega\times \Gamma, \\ 
				z&\in\psi(w)\quad\text{in}\quad S\times\Omega\times \Gamma,\\
				w(0,x)&=w_0(x)\quad\text{on}\quad \Omega\times\Gamma,	
			\end{split} 
		\end{align}
		where 
		\begin{align*}
			\tilde{\textbf{q}}_0=\frac{1}{|Y^p|}\int_{Y^p}D_1(x,y)\nabla_y k_0 dy,\;\bar{\xi}=\frac{1}{|Y^p|}\int_{Y^p}\xi (t, x, y)dy,\;	P(t,x)=\int_{\Gamma}\frac{1}{|Y^p|}\frac{\partial w}{\partial t}d\sigma_y
		\end{align*}
		and the coefficient matrices $A(t,x)=(a_{ij})_{1\le i,j\le n}$ and $B(t,x)=(b_{ij})_{1\le i,j\le n}$ are defined by
		\begin{align*}
			a_{ij}(t,x)=\int_{Y^p}\frac{D_1(x,y)}{|Y^p|}\left(\delta_{ij}+\sum_{i,j=1}^{n}\frac{\partial k_j}{\partial y_i}\right)dy, \quad b_{ij}(t,x)=\int_{Y^p}\frac{D_2(x,y)}{|Y^p|}\left(\delta_{ij}+\sum_{i,j=1}^{n}\frac{\partial k_j}{\partial y_i}\right)dy.
		\end{align*}
		Furthermore, $k_j(t,x,y)\in L^\infty(S\times\Omega ; H^{1,2}_{per}(Y))$, for $j=1,2,\cdots,n$ are the solutions of the cell problems 
		\begin{align}\label{eqn:cp}
			\begin{cases}
				-(div)_y(D_i(x,y)(\nabla_y k_j+e_j+\nabla_y k_0-\textbf{q}))=0 \text{ for all }y\in Y^p,\\
				-D_i(x,y)(\nabla_y k_j+e_j+\nabla_y k_0-\textbf{q}).\vec{n}=0 \text{ on } \Gamma,\\
				y \mapsto k_j(y) \text{ is } Y-\text{periodic},
			\end{cases}
		\end{align}
		for $i=1,2$ and for almost every $x\in\Omega$.
	\end{thm}
	\begin{thm}\label{thm:2}(Alternative limit process)
		Let $(\textbf{q}_\epsilon), (u_{\epsilon\delta}), (v_{\epsilon\delta}), (w_{\epsilon\delta})$ be the sequence of solutions to the system $\eqref{eqn:vrw}-\eqref{eqn:M12RWF}$. The sequences $(\textbf{q}_\epsilon), (u_{\epsilon\delta}), (v_{\epsilon\delta}) $ and $(w_{\epsilon\delta})$ converge in the sense of Lemma \ref{thm:TS2} and Lemma \ref{thm:TS3} such that
		\begin{align*}
			&\textbf{q}_\epsilon \overset{2}{\rightharpoonup} \textbf{q}, \quad \epsilon \nabla_x \textbf{q}_\epsilon \overset{2}{\rightharpoonup} \nabla_y \textbf{q}, \quad u_{\epsilon\delta} \overset{2}{\rightharpoonup} u_\delta,  \quad  \nabla 	u_{\epsilon\delta} \overset{2}{\rightharpoonup} \nabla u_\delta+\nabla_y u_{\delta1},\\
			& v_{\epsilon\delta} \overset{2}{\rightharpoonup} v_\delta, \quad \nabla 	v_{\epsilon\delta} \overset{2}{\rightharpoonup} \nabla v_\delta+\nabla_y v_{\delta1}, \quad w_{\epsilon\delta} \overset{2}{\rightharpoonup} w_\delta,
		\end{align*}
		where $u_{\delta1},v_{\delta1}\in L^2(S\times\Omega ; H^{1,2}_{per}(Y)/\mathbb{R}))$ are explicitly given by
		\begin{align}\label{eqn:NE7}
			u_{\delta1}=\sum_{j=1}^{n} \frac{\partial u_\delta}{\partial x_j}k_j(t,x,y)+k_0(t, x, y)u_\delta+p(x), \; v_{\delta1}=\sum_{j=1}^{n} \frac{\partial v_\delta}{\partial x_j}k_j(t,x,y)+k_0(t, x, y)v_\delta+q(x)
		\end{align}
		and $k_j\in L^\infty(S\times\Omega ; H^{1,2}_{per}(Y))$, for $j=1,2,\cdots,n$ are the solutions of the cell problems $\eqref{eqn:cp}$. The two-scale limit $(\textbf{q}(t, x, y), u_\delta(t,x), v_\delta(t,x)$, $w_\delta(t,x,y))\in (L^2(S\times\Omega; H^{1,2}_{per}(Y))^n\times[H^{1,2}(S ; H^{1,2}(\Omega)^*) \cap L^2(S ; H^{1,2}(\Omega))]^2 \times H^{1,2}(S\times\Omega\times\Gamma)$ is the solution of the macroscopic system of equations 
		\begin{align}\label{eqn:qs1}
			\begin{split}
				\frac{\partial \textbf{q}}{\partial t}-\mu \Delta_y \textbf{q}&=-\nabla_yp_1(t, x, y)-\nabla_x p(t, x)\quad \text{in}\quad S\times\Omega\times Y^p,\\
				\nabla_y.\textbf{q}&=0 \quad \text{in} \quad S\times\Omega\times Y^p,\\
				\nabla_x.\bar{\textbf{q}}&=0 \quad \text{in}\quad S\times\Omega,\\
				\textbf{q}(t, x, y)&=0 \quad\text{in} \quad S\times\Omega\times\Gamma,\\
				\bar{\textbf{q}}.\vec{n}&=0 \quad\text{in}\quad S\times\partial\Omega,\\
				\textbf{q}(0,x) &=\textbf{q}_0(x)\quad\text{in}\quad\Omega,
			\end{split}			
		\end{align} 
		\begin{align}\label{eqn:M1S2}
			\begin{split}
				\frac{\partial u_\delta}{\partial {t}} - \nabla.(A\nabla u_\delta-(\bar{\textbf{q}}-\tilde{\textbf{q}}_0)u_\delta)+P_\delta(t,x)&=0 \; \text{ in } S\times \Omega,	\\
				-(A\nabla u_\delta-(\bar{\textbf{q}}-\tilde{\textbf{q}}_0)u_\delta).\vec{n}&=\frac{d}{|Y^p|}\; \text{ on } S\times \partial\Omega_{in},\\
				-(A\nabla u_\delta-(\bar{\textbf{q}}-\tilde{\textbf{q}}_0)u_\delta).\vec{n}&=\frac{e}{|Y^p|}\; \text{ on } S\times \partial\Omega_{out},\\
				u_\delta(0,x)&=u_0(x)\;\text{ in }\Omega, 
			\end{split}
		\end{align} \vspace{-0.65cm}
		\begin{align}\label{eqn:M2S2}
			\begin{split}
				\frac{\partial v_\delta}{\partial {t}} - \nabla.(B\nabla v_\delta-(\bar{\textbf{q}}-\tilde{\textbf{q}}_0)v_\delta)+P_\delta(t,x)&=0 \; \text{ in } S\times \Omega,	\\
				-(B\nabla v_\delta-(\bar{\textbf{q}}-\tilde{\textbf{q}}_0)v_\delta).\vec{n}&=\frac{g}{|Y^p|} \;\text{ on } S\times \partial\Omega_{in},\\
				-(B\nabla v_\delta-(\bar{\textbf{q}}-\tilde{\textbf{q}}_0)v_\delta).\vec{n}&=\frac{h}{|Y^p|}\; \text{ on } S\times \partial\Omega_{out},\\
				v_\delta(0,x)&=v_0(x)\; \text{ in }\Omega,
			\end{split}
		\end{align}\vspace{-0.65cm}
		\begin{align}\label{eqn:M12S2}
			\begin{split}
				\frac{\partial w_\delta}{\partial  {t}}&=k_d(R(u_\delta,v_\delta)-\psi_\delta(w_\delta)) \; \text{ in } S\times\Omega\times \Gamma, \\ 
				w_\delta(0,x)&=w_0(x)\; \text{ on } \Omega\times\Gamma,	
			\end{split} 
		\end{align}
		which satisfies the estimate
		\begin{align*}
			\|u_\delta\|_{L^\infty(S; L^2(\Omega))}&+\|u_\delta\|_{(\Omega)^T}+\|\nabla u_\delta\|_{(\Omega)^T}+\left\lVert \frac{\partial u_\delta}{\partial t}\right\rVert_{L^2(S;H^{1,2}(\Omega)^*)}\notag\\
			+\|v_\delta\|&_{L^\infty(S; L^2(\Omega))}+\|v_\delta\|_{(\Omega)^T}+\|\nabla v_\delta\|_{(\Omega)^T}
			+\left\lVert \frac{\partial v_\delta}{\partial t}\right\rVert_{L^2(S;H^{1,2}(\Omega)^*)}\notag\\
			&+\|w_\delta\|_{(\Omega\times\Gamma)^T}+\left\lVert\frac{\partial w_\delta}{\partial t}\right\rVert_{(\Omega\times\Gamma)^T}+\|P_\delta\|_{(\Omega)^T}\le C,
		\end{align*}
		where $C$ is the generic constant independent of $\delta$ and $P_\delta(t,x)=\int_{\Gamma}\frac{1}{|Y^p|}\frac{\partial w_\delta}{\partial t}\;d\sigma_y$, with the elements of the elliptic and bounded homogenized matrix $A $ and $ B$ are given by
		\begin{align*}
			a_{ij}(t,x)=\int_{Y^p}\frac{D_1(x,y)}{|Y^p|}\left(\delta_{ij}+\sum_{i,j=1}^{n}\frac{\partial k_j}{\partial y_i}\right)dy, \quad b_{ij}(t,x)=\int_{Y^p}\frac{D_2(x,y)}{|Y^p|}\left(\delta_{ij}+\sum_{i,j=1}^{n}\frac{\partial k_j}{\partial y_i}\right)dy.
		\end{align*}
		Furthermore,  the sequences $(\textbf{q}, u_\delta, v_\delta, w_\delta, z_\delta )$ converge to a unique weak solution $(\textbf{q}, u, v, w, z)\in(L^2(S\times\Omega; H^{1,2}_{per}(Y))^n\times [H^{1,2}(S ; $ $H^{1,2}(\Omega)^*) \cap L^2(S ; H^{1,2}(\Omega))]^2\times H^{1,2}(S\times\Omega\times\Gamma)\times L^\infty(S\times\Omega\times\Gamma)$ of the problem $\eqref{eqn:qs}-\eqref{eqn:M12S}$.
	\end{thm}
	
	\subsection{A-priori estimates of the problem $(\mathcal{P}_{\epsilon\delta})$}
	\begin{lemma}\label{lemma:M1LB}
		The fluid velocity $\textbf{q}_\epsilon$ and concentrations $u_{\epsilon\delta}$, $v_{\epsilon\delta}$ and $w_{\epsilon\delta}$ satisfy 
		
		\begin{enumerate}
			\item[\textrm{i.}] $\|\textbf{q}_\epsilon\|_{L^\infty(S; \textbf{L}^2(\Omega_\epsilon^p))}\le M_q, \sqrt{\mu}\epsilon \|\nabla\textbf{q}_\epsilon\|_{(L^2(S\times\Omega_\epsilon^p))^{n\times n}}\le \bar{M}_q, \|\textbf{q}_\epsilon\|_{L^\infty(S; \textbf{L}^4(\Omega_\epsilon^p))}\le C$ and $ \left\lVert\frac{\partial \textbf{q}_\epsilon}{\partial t}\right\rVert_{L^2(S; \textbf{H}^{-1,2}_{div}(\Omega_\epsilon^p))}\le M_q^*$ a.e. in $S\times\Omega_\epsilon^p$.
			
			\item[\textrm{ii.}] $\|u_{\epsilon\delta}\|_{(\Omega^p_\epsilon)^T}\leq M_u $ and $\|v_{\epsilon\delta}\|_{(\Omega^p_\epsilon)^T}\leq M_v $ a.e. in $S\times\Omega^p_\epsilon,$ 
			
			\item[\textrm{iii.}] $\|\nabla u_{\epsilon\delta}\|_{(\Omega^p_\epsilon)^T}\leq \bar{M}_u$ and $\|\nabla v_{\epsilon\delta}\|_{(\Omega^p_\epsilon)^T}\le \bar{M}_v$ a.e. in $S\times\Omega^p_\epsilon,$
			
			\item[\textrm{iv.}] $\|u_{\epsilon\delta}\|_{L^\infty(S; L^4(\Omega_\epsilon^p))}\le C, \|v_{\epsilon\delta}\|_{L^\infty(S; L^4(\Omega_\epsilon^p))}\le C$ \text{ a.e. in } $S\times\Omega_\epsilon^p$.

			\item[\text{v.}] $\left\lVert\frac{\partial u_{\epsilon\delta}}{\partial t}\right\rVert_{L^2(S;H^{1,2}(\Omega^p_\epsilon)^*)}\leq M_u^*$ and $\left\lVert\frac{\partial v_{\epsilon\delta}}{\partial t}\right\rVert_{L^2(S;H^{1,2}(\Omega^p_\epsilon)^*)}\leq M_v^*$,
			
			\item[\text{vi.}] $\|w_{\epsilon\delta}\|_{(\Gamma^*_\epsilon)^T}\leq M_w$ and $\left\lVert\frac{\partial w_{\epsilon
					\delta}}{\partial t}\right\rVert_{(\Gamma^*_\epsilon)^T}\le M_w^*$  a.e. on $S\times\Gamma^*_\epsilon,$
			
		\end{enumerate}
		where the constants $M_q$, $M_u$, $M_v$, $M_w$, $\bar{M}_q$, $\bar{M}_u$, $\bar{M}_v$, $C$, $M_q^*$, $M_u^*$, $M_v^*$ and $M_w^*$ are independent of $\epsilon$ and $\delta$.
		Moreover, there exists a pressure $p_\epsilon=\partial_tP_\epsilon$ which satisfy the equation \eqref{eqn:v1} in the distributional sense. The pressure $P_\epsilon\in L^\infty(S; L^2_0(\Omega_\epsilon^p))$ satisfies
		\begin{align}\label{eqn:pe}
			\sup_{t\in S}\|\nabla P_\epsilon(t)\|_{H^{-1, 2}(\Omega_\epsilon^p)^n}\le C \text{ for all } \epsilon>0,
		\end{align}
		where $C$ is the constant independent of $\epsilon$.
	\end{lemma}
	
	\begin{proof}
		 We insert $\psi=\chi_{(0,t)}\textbf{q}_\epsilon\in L^2(S;\textbf{H}^{1,2}_{div}(\Omega_\epsilon^p))$ in \eqref{eqn:vrw} and deduce that
		\begin{align*}
			\|\textbf{q}_\epsilon(t)\|_{\textbf{L}^2(\Omega_\epsilon^p)}+2\epsilon^2\mu\|\nabla \textbf{q}_\epsilon\|_{(L^2(S\times\Omega_\epsilon^p))^{n\times n}}\le 	\|\textbf{q}_0\|_{\textbf{L}^2(\Omega_\epsilon^p)}.
		\end{align*}
		So, we have $\|\textbf{q}_\epsilon\|_{L^\infty(S; \textbf{L}^2(\Omega_\epsilon^p))}\le M_q$ and $\sqrt{\mu}\epsilon \|\nabla\textbf{q}_\epsilon\|_{(L^2(S\times\Omega_\epsilon^p))^{n\times n}}\le \bar{M}_q$. 
		We obtain $\|\textbf{q}_{\epsilon}\|_{L^4(Y)}\le C\|\textbf{q}_{\epsilon\delta}\|_{H^{1, 2}(Y)}$ by Gagliardo$-$Nirenberg$-$Sobolev inequality, where the constant $C$ depends on $n$ and $Y$. We now use a simple scaling argument and derive
		\begin{align}\label{eqn:4e}
			\|\textbf{q}_{\epsilon}\|_{L^4(\Omega_\epsilon^p)}\le C\|\textbf{q}_{\epsilon\delta}\|_{H^{1, 2}(\Omega_\epsilon^p)}\le C. 
		\end{align}
		From \eqref{eqn:vrw}, we see that
		\begin{align*}
			|\langle \partial_t \textbf{q}_\epsilon, \psi\rangle_{\Omega_\epsilon^p}|\le \epsilon^2\mu \|\nabla\textbf{q}_\epsilon\|_{\Omega_\epsilon^p}\|\nabla\psi\|_{\Omega_\epsilon^p}.
		\end{align*}
		Consequently,
		\begin{align*}
			\sup_{\|\psi\|_{\textbf{H}^{1, 2}_{div}(\Omega_\epsilon^p)}\le 1}|\langle \partial_t \textbf{q}_\epsilon, \psi\rangle_{\Omega_\epsilon^p}|\le \mu \|\nabla\textbf{q}_\epsilon\|_{\Omega_\epsilon^p} \text{  as   } 0<\epsilon<<1.
		\end{align*}
		Integration over the time gives,
		\begin{align*}
			\|\partial_t\textbf{q}_\epsilon\|_{L^2(S; \textbf{H}^{-1,2}_{div}(\Omega_\epsilon^p))}\le M_q^*.
		\end{align*}
		We choose $\phi = \chi_{(0,t)}u_{\epsilon\delta} \in L^2(S;H^{1,2}(\Omega_\epsilon^p))$ in the weak formulation $\eqref{eqn:M1RWF}$ and after simplification obtain 
		\begin{align}
			{\left\lVert u_{\epsilon\delta}(t)\right\rVert}^2_{\Omega^p_\epsilon}+2\alpha\left\lVert\nabla u_{\epsilon\delta} \right\rVert^2_{(\Omega^p_\epsilon)^t}  \leq {\left\lVert u_0\right\rVert}^2_{\Omega^p_\epsilon}+2\left|k_d \langle R(u_{\epsilon\delta},v_{\epsilon\delta})-\psi_\delta(w_{\epsilon\delta}), u_{\epsilon\delta} \rangle_{(\Gamma^*_\epsilon)^t}\right|\notag\\
			+2\left|\langle d, u_{\epsilon\delta}\rangle_{(\partial\Omega_{in})^t}\right|+2\left|\langle e, u_{\epsilon\delta}\rangle_{(\partial\Omega_{out})^t}\right|.\label{eqn: LM1}
		\end{align}
		Since
		\begin{align}\label{eqn:v}
			\int_{0}^{t}\int_{\Omega_\epsilon^p} (\textbf{q}_\epsilon.\nabla u_{\epsilon\delta}) u_{\epsilon\delta}dxdt&=\frac{1}{2}\int_{0}^{t}\int_{\Omega_\epsilon^p} \textbf{q}_\epsilon.\nabla u_{\epsilon\delta}^2dxdt\notag\\
			&=\frac{1}{2}\int_{0}^{t}\int_{\Omega_\epsilon^p} \textbf{q}_\epsilon.\vec{n} u_{\epsilon\delta}^2dxdt-\frac{1}{2}\int_{0}^{t}\int_{\Omega_\epsilon^p} \nabla.\textbf{q}_\epsilon u_{\epsilon\delta}^2dxdt=0.
		\end{align}
		We use Lemma \ref{thm:T1} to deal with the outer boundary term in $\eqref{eqn: LM1}$ and derive the inequality
		\begin{align*}
			2\left|\langle d, u_{\epsilon\delta}\rangle_{(\partial\Omega_{in})^t}\right|+2\left|\langle e, u_{\epsilon\delta}\rangle_{(\partial\Omega_{out})^t}\right|\le\frac{1}{2\gamma}\left(\|d\|^2_{(\partial\Omega_{in})^t}+\|e\|^2_{(\partial\Omega_{out})^t}\right)+2C\gamma\|u_{\epsilon\delta}\|^2_{L^2(S; H^{1,2}(\Omega_\epsilon^p))},
		\end{align*}
		where $\gamma$ is the constant comes from \textit{Young's }inequality. We now estimate the reaction rate term. Since $\left(\frac{k_1 u_{\epsilon\delta} + k_2 v_{\epsilon\delta}}{2}\right)^2\geq k_1 u_{\epsilon\delta} k_2 v_{\epsilon\delta}$ and $\left(\frac{k_1 u_{\epsilon\delta} + k_2 v_{\epsilon\delta}}{1+ k_1 u_{\epsilon\delta}+ k_2 v_{\epsilon\delta}}\right)^2\leq1$, therefore
		\begin{align}\label{eqn:RI}
			R(u_{\epsilon\delta}, v_{\epsilon\delta})=k\frac{k_1u_{\epsilon\delta} k_2 v_{\epsilon\delta}}{(1+k_1u_{\epsilon\delta}+k_2v_{\epsilon\delta})^2}\le \frac{k}{4}.
		\end{align}
		This in combination with the trace inequality \eqref{eqn:TI1} gives that
		\begin{align*}
			|2k_d \langle R(u_{\epsilon\delta},v_{\epsilon\delta})-\psi_\delta(w_{\epsilon\delta}), u_{\epsilon\delta}
			\rangle_{(\Gamma^*_\epsilon)^t}|\le C_1+2C\gamma_1{\|u_{\epsilon\delta}\|}^2_{(\Omega^p_\epsilon)^t}+2C\gamma_1{\epsilon}^2{\|\nabla u_{\epsilon\delta}\|}^2_{(\Omega^p_\epsilon)^t},
		\end{align*}
		where the \textit{Young's }inequality constant $\gamma_1$, the \textit{trace} constant $C$ and $C_1\left(:=\frac{k_d^2T}{2\gamma_1}\left(1+\frac{k}{4}\right)^2\frac{|\Gamma||\Omega|}{|Y|}\right)$ are independent of $\epsilon$ and $\delta$. As $0<\epsilon<<1$ so from \eqref{eqn: LM1} we get 
		\begin{align} 
			{\left\lVert u_{\epsilon\delta}(t) \right\rVert}^2_{\Omega^p_\epsilon}+(2\alpha-2C\gamma-2C\gamma_1)\left\lVert\nabla u_{\epsilon\delta} \right\rVert^2_{(\Omega^p_\epsilon)^t}\le C_2+(2C\gamma+2C\gamma_1){\|u_{\epsilon\delta} \|}^2_{(\Omega^p_\epsilon)^t},\label{eqn:II} 
		\end{align}
		\text{where }$C_2=C_1+{\left\lVert u_0 \right\rVert}^2_{\Omega^p_\epsilon}+\frac{1}{2\gamma}\left(\|d\|^2_{(\partial\Omega_{in})^t}+\|e\|^2_{(\partial\Omega_{out})^t}\right)$. Now, for the choice of $\gamma=\frac{\alpha}{4C}=\gamma_1$ and as an application of Gronwall's inequality, we have
		\begin{align}
			{\left\lVert u_{\epsilon\delta}(t) \right\rVert}^2_{\Omega^p_\epsilon}\le C_2+\alpha{\|u_{\epsilon\delta}\|}^2_{(\Omega^p_\epsilon)^t}\Rightarrow  {\left\lVert u_{\epsilon\delta}(t) \right\rVert}^2_{\Omega^p_\epsilon}\leq C_2 e^{\alpha t}, \text{ for a.e. }t\in S.\notag
		\end{align}
		Integration w.r.t. time gives, $\|u_{\epsilon\delta}\|_{(\Omega_\epsilon^p)^T}\le M_u$.
		Similarly, for $v_{\epsilon\delta}$ we will get $ {\left\lVert v_{\epsilon\delta} \right\rVert}_{(\Omega^p_\epsilon)^T}\leq M_v.$
		Again, we deduce from \eqref{eqn:II} that
		\begin{align*}
			\alpha\left\lVert\nabla u_{\epsilon\delta}\right\rVert^2_{(\Omega^p_\epsilon)^t}\leq C_2+\alpha{\| u_{\epsilon\delta}\|}^2_{(\Omega^p_\epsilon)^t}\implies {\left\lVert\nabla u_{\epsilon\delta}\right\rVert}_{(\Omega^p_\epsilon)^T}\le \bar{M}_u.
		\end{align*}
		Similarly, for $v_{\epsilon\delta}$ we get $\left\lVert\nabla v_\epsilon\right\rVert_{(\Omega^p_\epsilon)^T}\le \bar{M}_v.$
		Following the same line of the proof of \eqref{eqn:4e} we establish
		\begin{align*}
			\|u_{\epsilon\delta}\|_{L^\infty(S; L^4(\Omega_\epsilon^p))}\le C \text{  and  } \|v_{\epsilon\delta}\|_{L^\infty(S; L^4(\Omega_\epsilon^p))}\le C. 
		\end{align*}
		Furthermore, 
		\begin{align*}
			\left|\langle \frac{\partial u_{\epsilon\delta}}{\partial t},\phi\rangle_{\Omega_\epsilon^p}\right|&\le|\langle \bar{D}^\epsilon_1\nabla u_{\epsilon\delta},\nabla\phi\rangle_{\Omega_\epsilon^p}|+|\langle \textbf{q}_\epsilon u_{\epsilon\delta}, \nabla\phi\rangle_{\Omega_\epsilon^p}|+|k_d\langle R(u_{\epsilon\delta},v_{\epsilon\delta})-\psi_\delta(w_{\epsilon\delta}),\phi\rangle_{\Gamma^*_\epsilon}|+|\langle d,\phi\rangle_{\partial \Omega_{in}}|+|\langle e,\phi\rangle_{\partial \Omega_{out}}|\\
			&\le\left\{M\|\nabla u_{\epsilon\delta}\|_{\Omega_\epsilon^p}+\|\textbf{q}_\epsilon u_{\epsilon\delta}\|_{\Omega_\epsilon^p}+Ck_d\sqrt{\frac{|\Gamma||\Omega|}{|Y|}}+C\left(\|d\|_{\partial\Omega_{in}}+\|e\|_{\partial\Omega_{out}}\right)\right\} \|\phi\|_{H^{1,2}(\Omega_\epsilon^p)}.
		\end{align*}
		
		Now applying $i.$ and $iv.$ we get
		\begin{align}\label{eqn:qu}
			\|\textbf{q}_\epsilon u_{\epsilon\delta}\|_{\Omega_\epsilon^p}\le \|u_{\epsilon\delta}\|_{L^4(\Omega_\epsilon^p)}\|\textbf{q}_{\epsilon\delta}\|_{L^4(\Omega_\epsilon^p)} \le C.
		\end{align} 
		This implies $\left\lVert \frac{\partial u_{\epsilon\delta}}{\partial t} \right\rVert_{L^2(S;H^{1,2}(\Omega_\epsilon^p)^{*})}\le M_u^*$. Similarly, for $v_{\epsilon\delta}$ we establish $\left\lVert \frac{\partial v_{\epsilon\delta}}{\partial t} \right\rVert_{L^2(S;H^{1,2}(\Omega_\epsilon^p)^{*})}\le M_v^*.$
		Next, we put $\eta=\chi_{(0,t)}w_{\epsilon\delta}  \in $ $L^2(S\times\Gamma^*_\epsilon)$ in \eqref{eqn:M12RWF} and integrate w.r.t. $t$ to get
		\begin{align}\label{eqn:N2}
			{\left\lVert w_{\epsilon\delta}(t)\right\rVert}^2_{\Gamma^*_\epsilon}-{\left\lVert w_0\right\rVert}^2_{\Gamma^*_\epsilon}=2 k_d\langle R(u_{\epsilon\delta}, v_{\epsilon\delta})-\psi_\delta(w_{\epsilon\delta}), w_{\epsilon\delta}\rangle_{(\Gamma^*_\epsilon)^t}.
		\end{align}
		We note that
		\begin{align*}
			|k_d\langle R(u_{\epsilon\delta}, v_{\epsilon\delta})-\psi_\delta(w_{\epsilon\delta}), w_{\epsilon\delta} \rangle_{(\Gamma^*_\epsilon)^t}|\leq  \frac{k_d^2T}{2}(1+\frac{k}{4})^2\sup_{\epsilon>0}\epsilon|\Gamma^*_\epsilon|+\frac{1}{2}{\left\lVert w_{\epsilon\delta}\right\rVert}^2_{(\Gamma^*_\epsilon)^t},
		\end{align*}
		then \eqref{eqn:N2} becomes
		\begin{align*}
			&{\left\lVert w_{\epsilon\delta}(t)\right\rVert}^2_{\Gamma^*_\epsilon}
			\leq C_3+{\left\lVert w_{\epsilon\delta}\right\rVert}^2_{(\Gamma^*_\epsilon)^t}, \text{ where }
			C_3={\left\lVert w_0 \right\rVert}^2_{\Gamma^*_\epsilon}+k_d^2T(1+\frac{k}{4})^2\frac{|\Gamma||\Omega|}{|Y|}
		\end{align*}
		which by Gronwall's inequality yields ${\left\lVert w_{\epsilon\delta}\right\rVert}_{(\Gamma^*_\epsilon)^T}\leq M_w.$
		Again, taking $\eta=\chi_{(0,t)}\frac{\partial w_{\epsilon\delta}}{\partial t}$ in \eqref{eqn:M12RWF}, we obtain $  \left\lVert\frac{\partial w_{\epsilon\delta}}{\partial t}\right\rVert^2_{(\Gamma^*_\epsilon)^T}\le k_d^2T\left(1+\frac{k}{4}\right)^2\frac{|\Gamma||\Omega|}{|Y|}\implies \left\lVert\frac{\partial w_{\epsilon\delta}}{\partial t}\right\rVert _{(\Gamma^*_\epsilon)^T}\le M_w^*.$\\
		We now use Proposition III.1.1 of \cite{temam1977navier} and the identity \eqref{eqn:vrw} implies that there exists $p_\epsilon=\partial_tP_\epsilon\in W^{-1,\infty}(S; L^2_0(\Omega_\epsilon^p))$ such that
		\begin{align}\label{eqn:p1}
			-\int_{\Omega_\epsilon^p} P_\epsilon(t)\nabla.\psi dx=-\int_{\Omega_\epsilon^p} (\textbf{q}_\epsilon(t)-\textbf{q}_0).\psi dx-\epsilon^2\mu\int_{0}^{t}\int_{\Omega_\epsilon^p}\nabla\textbf{q}_\epsilon(s):\nabla\psi dxds, \text{ for all } \psi\in\textbf{H}^1_0(\Omega_\epsilon^p).
		\end{align}
		Therefore, we have
		\begin{align*}
			\langle \nabla P_\epsilon(t), \psi\rangle_{\Omega_\epsilon^p}\le (\|\textbf{q}_\epsilon(t)\|_{\Omega_\epsilon^p}+\|\textbf{q}_0\|_{\Omega_\epsilon^p})\|\psi\|_{\Omega_\epsilon^p}+\mu\int_{0}^{t}\|\nabla\textbf{q}_\epsilon\|_{\Omega_\epsilon^p}\|\nabla\psi\|_{\Omega_\epsilon^p}dt,
		\end{align*}
		as $0<\epsilon<<1$. Applying $i.$ we can write
		\begin{align*}
			\langle \nabla P_\epsilon(t), \psi\rangle_{\Omega_\epsilon^p}\le C\|\psi\|_{\textbf{H}^1_0(\Omega_\epsilon^p)}.
		\end{align*}
		This yields \eqref{eqn:pe}.
	\end{proof}
	\begin{rem}cf. \cite{bavnas2017homogenization}
		We multiply the equation \eqref{eqn:p1} by a test function $\partial_t\zeta$ where $\zeta\in C_0^\infty(S)$ and integration w.r.t time leads to
		\begin{align}\label{eqn:p2}
			-\int_{0}^{t}\int_{\Omega_\epsilon^p}P_\epsilon(t)\nabla.\psi dx\partial_t\zeta dt=\int_{0}^{t}\langle \partial_t \textbf{q}_\epsilon, \psi\rangle_{\Omega_\epsilon^p}\zeta(t)dt+\epsilon^2\mu\int_{0}^{t}\int_{\Omega_\epsilon^p}\nabla\textbf{q}_\epsilon:\nabla\psi dx\zeta(t)dt.
		\end{align}
		The above formulation implies that $P_\epsilon(t)=\int_{0}^{t}p_\epsilon(s)ds$ satisfies the equation \eqref{eqn:v1} in the distributional sense. Moreover, the formulation \eqref{eqn:p2} is equivalent to \eqref{eqn:vw} for $\psi\in \textbf{H}^{1,2}_{div}(\Omega_\epsilon^p)$. We will use the formulation \eqref{eqn:p2} to derive the two-scale limit of \eqref{eqn:v1} due to the limited time regularity of the pressure $p_\epsilon\in W^{-1,\infty}(S; L^2_0(\Omega_\epsilon^p)).$
	\end{rem}
	\subsection{Existence of solution of the regularized problem}
	We now show the existence of solution by applying \textit{Rothe's method} and \textit{Galerkin's method}. Since the chemistry does not affect the fluid flow, therefore we can treat the unsteady Stokes system $\eqref{eqn:v1}-\eqref{eqn:v4}$ independently from the transport of the mobile and immobile species. We follow the idea of \cite{evans2010partial, temam1977navier} and establish the existence of a unique weak solution $\textbf{q}_\epsilon\in \mathcal{Q}_\epsilon$ by applying Galerkin's approximation.

	Next for arbitrary $(u_{\epsilon\delta}, v_{\epsilon\delta})\in \mathcal{U}_\epsilon\times\mathcal{V}_\epsilon$, we consider \eqref{eqn:M12RWF} with $w_{\epsilon\delta}(0, x)=w_0(x)$. Since $R(u_{\epsilon\delta}, v_{\epsilon\delta})$ is constant in $w_{\epsilon\delta}$ and $\psi_\delta(w_{\epsilon\delta})$ is Lipschitz w.r.t $w_{\epsilon\delta}$ so by \textit{Picard-Lindelof theorem} there exists a unique local solution $w_{\epsilon\delta} \in C^1(0,T_1(x))$ of the problem \eqref{eqn:M12RWF}, where $T_1(x)\le T$. Now, upon partial integration of the strong form of \eqref{eqn:M12RWF} and using \eqref{eqn:RI}, we have
	\begin{align*}
		|w_{\epsilon\delta}(t,x)|&\le |w_0(x)|+k_d\int_{0}^{t}(1+\frac{k}{4})dt\le\|w_0\|_{L^\infty(\Omega)}+k_d(1+\frac{k}{4})T,\; \text{ for a.e. } t \text{ and } x.
	\end{align*}
	
	Therefore, for a.e. $t\in S$, the solution of the equation for immobile species exists globally. Now, we define two billinear forms on $\Omega_\epsilon^p$ such that $b(u_{\epsilon\delta},\phi)=\langle \bar{D}^\epsilon_1\nabla u_{\epsilon\delta}-\textbf{q}_\epsilon u_{\epsilon\delta},\nabla\phi\rangle_{\Omega_\epsilon^p}$ and $b(v_{\epsilon\delta}, \theta)= \langle \bar{D}^\epsilon_2\nabla v_{\epsilon\delta}-\textbf{q}_\epsilon v_{\epsilon\delta},\nabla\theta\rangle_{\Omega_\epsilon^p}$. Then, \eqref{eqn:M2RWF} can be rewritten as
	\begin{align*}
		\langle\frac{\partial v_{\epsilon\delta}}{\partial t},\theta\rangle_{\Omega_\epsilon^p}+b(v_{\epsilon\delta}, \theta)+\langle \frac{\partial w_{\epsilon\delta}}{\partial t},\theta\rangle_{\Gamma^*_\epsilon}+\langle g,\theta\rangle_{\partial\Omega_{in}}+\langle h,\theta\rangle_{\partial\Omega_{out}}=0,\\
		\text{ for all } \theta \in H^{1,2}(\Omega_\epsilon^p)\;\text{ and a.e. $t$ in } [0,T].
	\end{align*}
	In other words, for arbitary $u_{\epsilon\delta}\in\mathcal{U}_\epsilon$, we have to find $v_{\epsilon\delta} : [0,T] \rightarrow H^{1,2}(\Omega_\epsilon^p)$ such that
	\begin{subequations}
		\begin{align}
			\langle\frac{\partial v_{\epsilon\delta}}{\partial t},\theta\rangle_{\Omega_\epsilon^p}+b(v_{\epsilon\delta},\theta)&=\langle f(v_{\epsilon\delta}),\theta \rangle_{\Gamma^*_\epsilon}-\langle g,\theta\rangle_{\partial\Omega_{in}}-\langle h,\theta\rangle_{\partial\Omega_{out}}, \label{eqn:R}\\
			&\qquad\qquad\text{ for all } \theta \in H^{1,2}(\Omega^p_\epsilon) \; \text{ and a.e. $t\in S$  } , \notag\\
			v_{\epsilon\delta}|_{t=0}&=v_0,\label{eqn:Rb}
		\end{align}
	\end{subequations}
	where $f(v_{\epsilon\delta})=k_d(\psi_\delta(w_{\epsilon\delta})-R(u_{\epsilon\delta}, v_{\epsilon\delta}))$.
	\subsubsection{Rothe's method}(cf. \cite{neuss1992homogenization})
	Let $\{0=t_0<t_1<\cdots<t_{n-1}<t_n=T\}$ be a partition of the time interval $[0,T]$ with step size $h=(t_i-t_{i-1})=\frac{T}{n}$, where $n\in \mathbb{N}$. We employ time discretization to the equation \eqref{eqn:R} which leads to
	\begin{align}
		\langle \frac{v_i-v_{i-1}}{h}, \theta \rangle_{\Omega_\epsilon^p}+b(v_i,\theta)=\langle f(v_{i-1}),\theta \rangle_{\Gamma_\epsilon^*}-\langle g_{i-1},\theta\rangle_{\partial\Omega_{in}}-\langle h_{i-1},\theta\rangle_{\partial\Omega_{out}}  ,  \label{eqn:R2}
	\end{align}
	for all $\theta \in H^{1,2}(\Omega_\epsilon^p)$ and for all $i=1,2,\cdots,n$, where $f(v_{i-1})=k_d(\psi_\delta(w_{\epsilon\delta})-R(u_{\epsilon\delta}, v_{i-1}))$. Now, we introduce the linear operator $\mathcal{T}_h : H^{1,2}(\Omega_\epsilon^p)\rightarrow L^2(\Omega_\epsilon^p) $ such that 
	\begin{align*}
		\langle \mathcal{T}_hv, \theta \rangle_{\Omega_\epsilon^p}=\frac{1}{2}\langle v,\theta \rangle_{\Omega_\epsilon^p}+b(v,\theta)
	\end{align*}
	and the linear form
	\begin{align*}
		\langle l_{i-1},\theta \rangle_{\Omega_\epsilon^p}=\langle f(v_{i-1}), \theta \rangle_{\Gamma^*_\epsilon}-\langle g_{i-1},\theta\rangle_{\partial\Omega_{in}}-\langle h_{i-1},\theta\rangle_{\partial\Omega_{out}}+\frac{1}{2}\langle v_{i-1}, \theta \rangle_{\Omega_\epsilon^p}.
	\end{align*}
	Then, our new equation looks like
	\begin{align*}
		\langle \mathcal{T}_hv, \theta \rangle_{\Omega_\epsilon^p} = \langle l_{i-1}, \theta \rangle_{\Omega_\epsilon^p},  \;\text{ for all } \theta \in H^{1,2}(\Omega_\epsilon^p),
	\end{align*}
	where
	\begin{align*}
		\langle \mathcal{T}_hv, v \rangle_{\Omega_\epsilon^p}\ge\frac{1}{2}\|v\|^2_{\Omega_\epsilon^p}+\alpha\|\nabla v\|^2_{\Omega_\epsilon^p}\geq C_4\|v\|^2_{H^{1,2}(\Omega_\epsilon^p)}, \text{ where } C_4=min\{\alpha,\frac{1}{2}\}.
	\end{align*}
	as $\langle \textbf{q}v, \nabla v\rangle_{\Omega_\epsilon^o}=0$ by \eqref{eqn:v}. Again,
	\begin{align*}
		\langle \mathcal{T}_hv, v \rangle_{\Omega_\epsilon^p}\le \frac{1}{2}\|v\|^2_{\Omega_\epsilon^p}+M\|\nabla v\|^2_{\Omega_\epsilon^p}\leq C_5\|v\|^2_{H^{1,2}(\Omega_\epsilon^p)}, \text{ where }C_5=max\{M,\frac{1}{2}\}.
	\end{align*}  
	Hence, $\langle \mathcal{T}_h ., . \rangle_{\Omega_\epsilon^p}$ is an elliptic and bounded bilinear form and since $l_{i-1}$ is a bounded functional on $H^{1,2}(\Omega_\epsilon^p)$ so by Lax-Milgram lemma there exists a unique $v_i \in H^{1,2}(\Omega_\epsilon^p)$ satisfying \eqref{eqn:R2}. Next, we define Rothe functions $v_n : [0,T] \rightarrow H^{1,2}(\Omega_\epsilon^p)$ by 
	\begin{align*}
		v_n(t)=v_i\left(\frac{t-t_{i-1}}{h}\right)-v_{i-1}\left(\frac{t-t_i}{h}\right),
	\end{align*}
	and the step function $\overline{v}_n: [0,T] \rightarrow H^{1,2}(\Omega_\epsilon^p)$ in such a way that $\overline{v}_n(t)=v_i$, for all $t\in (t_{i-1},t_i]$ and $\overline{v}_n(0)=v_0$. To show that this Rothe functions converges to a solution of the continuous equation \eqref{eqn:R} we have to obtain a-priori estimates for $v_n$.
	
	\begin{lemma}\label{lemma: 1}
		The difference $(v_i-v_{i-1})$ satisfies the inequality
		\begin{align*}
			\left\lVert\frac{v_i-v_{i-1}}{h}\right\rVert^2_{\Omega_\epsilon^p}+\frac{\lambda}{h^2}\|\nabla(v_i-v_{i-1})\|^2_{\Omega_\epsilon^p}\leq C, 	\text{ for all } i=1,2,\cdots,n,
			\text{  where  }\lambda=\begin{cases}
				\frac{1}{2} \quad\text{\hspace{0.25cm}if $i=1$}\\
				2\quad\text{otherwise}.
			\end{cases}
		\end{align*}		
	\end{lemma}
	
	\begin{proof}
		The proof is mere calculation so we give it in the appendix.
	\end{proof}

	\begin{lemma}
		For $v_i$, ${v}_n$ and $\bar{v}_n$, we have the following bounds
		
		$(a) \|v_i\|_{H^{1,2}(\Omega_\epsilon^p)}\leq C$ and $\left\lVert\frac{v_i-v_{i-1}}{h}\right\rVert_ {\Omega_\epsilon^p}\leq C$ for all $i=1,2,\cdots,n$.
		
		(b) $\|\bar{v}_n(t)\|_ {H^{1,2}(\Omega_\epsilon^p)}\leq C$ and $\left\lVert\frac{dv_n(t)}{dt}\right\rVert_{\Omega_\epsilon^p}\leq C$ for a.e. $t\in [0,T]$.
	\end{lemma}
	\begin{proof}
		We can estimate from $v_i=v_0+\sum_{j=1}^{i}(v_j-v_{j-1})$ that
		\begin{align*}
			&\|v_i\|_{\Omega_\epsilon^p}\leq \|v_0\|_{\Omega_\epsilon^p}+h\sum_{j=1}^{i}\left\lVert\frac{v_j-v_{j-1}}{h}\right\rVert_{\Omega_\epsilon^p}\le C.
		\end{align*}
		We choose the test function $\theta=v_i$ in \eqref{eqn:R2} and obtain
		\begin{align*}
			\langle \frac{v_i-v_{i-1}}{h}, v_i\rangle_{\Omega_\epsilon^p}+b(v_i,v_i)=\langle f(v_{i-1}), v_i\rangle_{\Gamma^*_\epsilon}-\langle g_{i-1},v_i\rangle_{\partial\Omega_{in}}-\langle h_{i-1},v_i\rangle_{\partial\Omega_{out}}.
		\end{align*}
		Relying on (\textbf{A6}) and the trace inequalities, we have
		\begin{align*}
			(\alpha-C \gamma_1-C\gamma_2)\|\nabla v_i\|^2_{\Omega_\epsilon^p}\leq C_{11} +(C\gamma_1+C\gamma_2+\frac{1}{2})\|v_i\|^2_{\Omega_\epsilon^p},
		\end{align*}
		where $C_{11}=C_6+\frac{\sigma_i}{2}+\frac{1}{4\gamma_2}(\|g_{i-1}\|^2_{\partial\Omega_{in}}+\|h_{i-1}\|^2_{\partial\Omega_{out}})$. Hence, for the choice of $\gamma_1=\frac{\alpha}{4C}=\gamma_2$, we can conclude
		\begin{align*}
			\|\nabla v_i\|_{\Omega_\epsilon^p}\leq C \text{  and  } \|v_i\|_{H^{1,2}(\Omega_\epsilon^p)}=\|v_i\|_{\Omega_\epsilon^p}+\|\nabla v_i\|_{\Omega_\epsilon^p}\leq C.
		\end{align*} 
		Next,
		\begin{align*}
			\left\lVert\frac{v_i-v_{i-1}}{h}\right\rVert_{\Omega_\epsilon^p}\leq \sigma_i\leq C \; \text{ follows from Lemma \ref{lemma: 1}. }
		\end{align*}
		Now, for the Rothe's step function, we get $\bar{v}_n(t)=v_i$ for all $t\in (t_{i-1}, t_i]$. Consequently, $\|\bar{v}_n(t)\|_{H^{1,2}(\Omega_\epsilon^p)}=\|v_i\|_{H^{1,2}(\Omega_\epsilon^p)}\leq C$ and $\left\lVert\frac{dv_n}{dt}\right\rVert_{\Omega_\epsilon^p}=\left\lVert\frac{v_i-v_{i-1}}{h}\right\rVert_{\Omega_\epsilon^p}\leq C.$
	\end{proof}
	\begin{lemma}
		The problem $\eqref{eqn:R} - \eqref{eqn:Rb}$ has atmost one solution.
	\end{lemma}
	\begin{proof}
		We arrive at a situation of choosing an appropriate Banach space as a range of the Rothe's function on which we can apply a version of \textit{Arcela-Ascoli theorem.} We take the banach space $B=L^2(\Omega_\epsilon^p)$ and $[0,T]$ as the compact set then
		\begin{align*}
			&\|v_n(t)-v_n(s)\|_{\Omega_\epsilon^p}\leq\int_{s}^{t}\left\lVert\frac{dv_n(\sigma)}{d\sigma}\right\rVert_{\Omega_\epsilon^p}d\sigma\leq C|t-s|\text{ and } \|v_n(t)\|_{H^{1,2}(\Omega_\epsilon^p)}\leq C.
		\end{align*}
		So, $\{v_n\}$ is equicontinuous and therefore by \textit{Arcela-Ascoli theorem} there exists $v\in L^2(\Omega_\epsilon^p) $ such that upto a subsequence $v_n\rightarrow v$ strongly in $C([0,T] ; L^2(\Omega_\epsilon^p))$. We also have
		\begin{align*}
			\int_{\Omega_\epsilon^p}|v_n(t)|^2dx+\int_{\Omega_\epsilon^p}|\nabla v_n(t)|^2dx\leq C \Longrightarrow  \int_{\Omega_\epsilon^p}\|\nabla v(t)\|^2dx\leq C, \text{ as }n\to \infty.
		\end{align*}
		Hence, $v(t)\in H^{1,2}(\Omega_\epsilon^p)$. Further, we need to show, $\bar{v}_n(t)\rightharpoonup v(t)$ in $H^{1,2}(\Omega_\epsilon^p)$ for all $t\in[0,T]$. Since $\|\bar{v}_n(t)\|_{H^{1,2}(\Omega_\epsilon^p)}\leq C$, therefore, upto a subsequence $\bar{v}_n(t)\rightharpoonup\bar{v}(t)$ in $H^{1,2}(\Omega_\epsilon^p)$ for all $t\in[0,T]$. Now, for $t\in (t_{i-1}, t_i]$ 
		\begin{align*}
			\|v_n(t)-\bar{v}_n(t)\|^2_{\Omega_\epsilon^p}&=\int_{\Omega_\epsilon^p}|v_{i-1}+\frac{v_i-v_{i-1}}{h}(t-t_{i-1})-v_i|^2dx\\
			&\leq (t-t_i)^2\int_{\Omega_\epsilon^p}{\left|\frac{v_i-v_{i-1}}{h}\right|}^2dx\leq h^2C\rightarrow 0 \; \text{  as   }h\rightarrow 0
		\end{align*}
		and so $\bar{v}(t)\equiv v(t)$ for all $t\in[0,T]$. Our next target is to establish $\frac{dv_n}{dt}\rightharpoonup\frac{dv}{dt}$ in $L^2([0,T]; L^2(\Omega_\epsilon^p))$. As $\left\lVert\frac{dv_n}{dt}\right\rVert_{\Omega_\epsilon^p}\leq C$ for a.e. $t\in [0,T]$, therefore, there exists a subsequence of $\frac{dv_n}{dt}$ which converges weakly to some $\eta$ in $L^2([0,T]; L^2(\Omega_\epsilon^p))$. Claim: $\eta=\frac{dv}{dt}$ in the sense of distribution
		\begin{align*}
			&\langle v,\frac{\partial \phi}{\partial t}\rangle_{\Omega_\epsilon^p}=\lim_{n\rightarrow \infty}\langle v_n,\frac{\partial \phi}{\partial t}\rangle_{\Omega_\epsilon^p}=-\lim_{n\rightarrow \infty}\langle \frac{\partial v_n}{\partial t}, \phi\rangle_{\Omega_\epsilon^p}=-\langle \eta, \phi\rangle_{\Omega_\epsilon^p} \; \text{ for all smooth }\phi\\
			&\implies \eta=\frac{dv}{dt} \; \text{ in the sense of distribution. }
		\end{align*}
		Next, we consider the behavior of $v_n$ on $(\Gamma^*_\epsilon)^t$. For that, we require the $L^2-$bounds $\|v_n(t)\|_{\Omega_\epsilon^p}=\left\lVert v_{i}(\frac{t-t_i}{h})-v_{i-1}(\frac{t-t_i}{h})\right\rVert_{\Omega_\epsilon^p}\leq C$ and $\|\nabla v_n(t)\|_{\Omega_\epsilon^p}\leq \|\nabla v_{i-1}\|_{\Omega_\epsilon^p}+\|\nabla v_{i}\|_{\Omega_\epsilon^p}\leq C$.
		The equation \eqref{eqn:R2} can be written in terms of Rothe's function as
		\begin{align*}
			\langle \frac{dv_n(t)} {dt}, \theta\rangle_{\Omega_\epsilon^p}+b(\bar{v}_n, \theta)= \langle f(\bar{v}_{n-1}), \theta\rangle_{\Gamma^*_\epsilon}-\langle g_{i-1},\theta\rangle_{\partial\Omega_{in}}-\langle h_{i-1},\theta\rangle_{\partial\Omega_{out}}.
		\end{align*}
		Since $\|f(\bar{v}_{n-1})\|^2_{\Gamma^*_\epsilon}\le k_d^2(1+\frac{k}{4})^2\frac{|\Gamma||\Omega|}{|Y|}=C$, therefore, we arrive at
		\begin{align*}
			&\left|\langle \frac{dv_n(t)} {dt}, \theta\rangle_{\Omega_\epsilon^p}\right|\leq (M\|\nabla v_i\|_{\Omega_\epsilon^p}+\|\textbf{q}_iv_i\|_{\Omega_\epsilon^p}+C(1+\|g_{i-1}\|_{\partial\Omega_{in}}+\|h_{i-1}\|_{\partial\Omega_{out}}))\|\theta\|_{H^{1,2}(\Omega_\epsilon^p)}\\
			\implies& \left\lVert\frac{dv_n} {dt}\right\rVert_{L^2(S ; H^{1,2}(\Omega_\epsilon^p)^*)}\leq C.
		\end{align*}
		Finally, we get the following estimate
		\begin{align}
			\|v_n(t)\|_{\Omega_\epsilon^p}+\|\nabla v_n\|_{(\Omega_\epsilon^p)^t}+\left\lVert\frac{dv_n} {dt}\right\rVert_{L^2(S ; H^{1,2}(\Omega_\epsilon^p)^*)}\leq C. \notag
		\end{align}
		Hence, $v_n\rightarrow v$ strongly in $C([0,T];H^s(\Omega_\epsilon^p)^*)\cap L^2((0,T); H^s(\Omega_\epsilon^p))$ for $s\in(0,1)$. In particular, $v_n\rightarrow v$ strongly in $L^2(\Gamma^*_\epsilon)^t$. Since $f$ is Lipschitz, $f(v_n)\rightarrow f(v)$ strongly in $L^2(\Gamma^*_\epsilon)^t$ and pointwise in $(\Gamma^*_\epsilon)^t$. We pass the limit as $n\rightarrow\infty$ in 
		\begin{align*}
			\int_{0}^{t}\langle\frac{dv_n} {dt}, \theta\rangle_{\Omega_\epsilon^p}dt+\int_{0}^{t}b(\bar{v}_n, \theta)dt=\int_{0}^{t} \langle f(\bar{v}_{n-1}), \theta\rangle_{\Gamma^*_\epsilon}dt-\int_{0}^{t}\langle g_{i-1},\theta\rangle_{\partial\Omega_{in}}dt-\int_{0}^{t}\langle h_{i-1},\theta\rangle_{\partial\Omega_{out}}dt, 
		\end{align*} 
		and get $v$ as a solution of
		\begin{align*}
			\int_{0}^{t}\langle \frac{dv(t)} {dt}, \theta\rangle_{\Omega_\epsilon^p}dt+\int_{0}^{t}b(v, \theta)dt=\int_{0}^{t} \langle f(v), \theta\rangle_{\Gamma^*_\epsilon}dt-\int_{0}^{t}\langle g,\theta\rangle_{\partial\Omega_{in}}-\int_{0}^{t}\langle h,\theta\rangle_{\partial\Omega_{out}}, 
		\end{align*}
		for all $\theta\in H^{1,2}(\Omega_\epsilon^p)$. The proof to establish the existence of $u_{\epsilon\delta}\in \mathcal{U}_\epsilon$ for \eqref{eqn:M1RWF} follows the same line of arguments as we did for $v_{\epsilon\delta}$. Here we apply the time discretization to the equation \eqref{eqn:M1RWF} with $u_{\epsilon\delta}(0, x)=u_0(x)$ and use Rothe's method and proceed as above.
		
		\textbf{Uniqueness: }If possible, suppose $(\textbf{q}^1_\epsilon, u^1_{\epsilon\delta}, v^1_{\epsilon\delta}, w^1_{\epsilon\delta})$ and $(\textbf{q}^2_\epsilon, u^2_{\epsilon\delta}, v^2_{\epsilon\delta}, w^2_{\epsilon\delta})$ are two solutions of the regularized problem $\eqref{eqn:vrw}-\eqref{eqn:M12RWF}$. Let $\textbf{Q}_\epsilon=\textbf{q}^1_\epsilon-\textbf{q}^2_\epsilon\ge 0$,  $W_{\epsilon\delta}=w^1_{\epsilon\delta}-w^2_{\epsilon\delta}\ge 0$, $V_{\epsilon\delta}=v^1_{\epsilon\delta}-v^2_{\epsilon\delta}\ge 0$ and $U_{\epsilon\delta}=u^1_{\epsilon\delta}-u^2_{\epsilon\delta}\ge 0$.\\
		We choose the test function $\psi=\chi_{(0,t)}\textbf{Q}_\epsilon$ in \eqref{eqn:vrw} and write it for $\textbf{Q}_\epsilon$ to deduce
		\begin{align*}
			\|\textbf{Q}_\epsilon(t)\|^2_{\Omega_\epsilon^p}+\epsilon^2\mu\|\nabla\textbf{Q}_\epsilon\|_{(\Omega_\epsilon^p)^t}=0.
		\end{align*}
		Then application of the  Gronwall's inequality yields
		\begin{align*}
			\|\textbf{Q}_\epsilon(t)\|^2_{\Omega_\epsilon^p}=0\implies \textbf{q}^1_\epsilon(t)=\textbf{q}^1_\epsilon(t)\text{ for a.e. } t\in S.
		\end{align*}
		We write \eqref{eqn:M12RWF} for $W_{\epsilon\delta}$ and use $\eta=\chi_{(0,t)}W_{\epsilon\delta}$ to obtain
		\begin{align*}
			&\frac{\partial}{\partial t}\|W_{\epsilon\delta}(s)\|^2_{\Gamma_\epsilon^*}\leq 2k_dL_R\langle (U_{\epsilon\delta}+V_{\epsilon\delta}), W_{\epsilon\delta} \rangle_{\Gamma^*_\epsilon}.
		\end{align*}
		Applying young's inequality and integrating both sides w.r.t. $t$, we have
		\begin{align}
			\|W_{\epsilon\delta}(t)\|^2_{\Gamma^*_\epsilon}&\le k_d^2L_R^2\|U_{\epsilon\delta}+V_{\epsilon\delta}\|^2_{(\Gamma^*_\epsilon)^t}+\int_{0}^{t}\|W_{\epsilon\delta}(s)\|^2_{\Gamma^*_\epsilon}ds.\notag\\
			\Longrightarrow \|W_{\epsilon\delta}(t)\|^2_{\Gamma^*_\epsilon}&\le e^Tk_d^2L_R^2\|U_{\epsilon\delta}+V_{\epsilon\delta}\|^2_{(\Gamma^*_\epsilon)^t}\text{ by Gronwall's inequality. } \label{eqn:RE}
		\end{align}
		Now, writing \eqref{eqn:M1RWF} and \eqref{eqn:M2RWF} for $U_{\epsilon\delta}$ and $V_{\epsilon\delta}$ and adding them both, we get
		\begin{align*}
			\left\langle \frac{\partial}{\partial t}(U_{\epsilon\delta}+V_{\epsilon\delta}), \phi \right\rangle_{\Omega_\epsilon^p} + \langle \bar{D}_1^\epsilon\nabla U_{\epsilon\delta}+ \bar{D}_2^\epsilon\nabla V_{\epsilon\delta}-\textbf{q}_\epsilon(U_{\epsilon\delta}+V_{\epsilon\delta}), \nabla \phi\rangle_{\Omega_\epsilon^p}\le 2 \left|\left\langle \frac{\partial W_{\epsilon\delta}}{\partial t}, \phi\right\rangle_{\Gamma^*_\epsilon}\right|,
		\end{align*}
		since the uniqueness of $\textbf{q}_\epsilon$ is already proved. Now  we put $P(t)=U_{\epsilon\delta}(t)+V_{\epsilon\delta}(t)$ and $\phi=U_{\epsilon\delta}+V_{\epsilon\delta}$. Then relying on (\textbf{A6}) and \eqref{eqn:v} we derive
		\begin{align*}
			&\frac{\partial}{\partial t}\|P(t)\|^2_{\Omega_\epsilon^p}+2\alpha \|\nabla P(t)|\|^2_{\Omega_\epsilon^p}\le 4 \left|\left\langle \frac{\partial W_{\epsilon\delta}}{\partial t}, P(t)\right\rangle_{\Gamma^*_\epsilon}\right|\\
			\Rightarrow &\|P(t)\|^2_{\Omega_\epsilon^p}+2\alpha\|\nabla P\|^2_{(\Omega_\epsilon^p)^t}\le 4\left|\left\langle \frac{\partial W_{\epsilon\delta}}{\partial t}, P\right\rangle_{(\Gamma^*_\epsilon)^t}\right|, \text{ by integrating both sides}.
		\end{align*}
		Replacing $\eta$ by $P(t)$ in \eqref{eqn:M12RWF}, we get
		\begin{align*}
			\left|\left\langle \frac{\partial W_{\epsilon\delta}}{\partial t}, P\right\rangle_{(\Gamma^*_\epsilon)^t}\right|\le k_dL_R\| P\|^2_{(\Gamma^*_\epsilon)^t}.
		\end{align*}
		Now, applying trace inequality \eqref{eqn:TI2} and up on simplification, we can write
		\begin{align*}
			&\|P(t)\|^2_{\Omega_\epsilon^p}+2\alpha\|\nabla P\|^2_{(\Omega_\epsilon^p)^t}\le 4k_dL_RC^*\|P\|^2_{(\Omega_\epsilon^p)^t}+8k_dL_R\left[\frac{\bar{C}^2}{4\mu}\|P\|^2_{(\Omega_\epsilon^p)^t}+\mu\|\nabla P\|^2_{(\Omega_\epsilon^p)^t}\right]\\
			&\Rightarrow\|P(t)\|^2_{\Omega_\epsilon^p}\le (4k_dL_RC^*+\frac{2\bar{C}^2k_dL_R}{\mu})\|P\|^2_{(\Omega_\epsilon^p)^t},\text{ by putting $\mu=\frac{\alpha}{4k_dL_R}$}.
		\end{align*}
		Gronwall's inequality gives
		\begin{align*}
			\|P(t)\|^2_{\Omega_\epsilon^p}=0\implies u^1_{\epsilon\delta}(t)=u^2_{\epsilon\delta}(t) \text{ and } v^1_{\epsilon\delta}(t)=v^2_{\epsilon\delta}(t) \text{ for a.e. } t\in S.
		\end{align*}
		From \eqref{eqn:RE}, we conclude that
		\begin{align*}
			\|W_{\epsilon\delta}(t)\|^2_{\Gamma^*_\epsilon}=0\implies w^1_{\epsilon\delta}(t)=w^2_{\epsilon\delta}(t) \text{ for a.e. } t\in S.
		\end{align*}
		Hence, the solution is unique.
	\end{proof}
	\section{Proof of the Theorem \ref{thm:1}}
	\textbf{$\delta\rightarrow 0$ part}:
	We prove Theorem \ref{thm:1} by using Lemma \ref{lemma:M1LB}, which gives the necessary estimates to pass the regularization parameter to zero. Now, by standard compactness arguments, we can extract a subsequence $(u_{\epsilon\delta}, v_{\epsilon\delta}, w_{\epsilon\delta}, z_{\epsilon\delta})$ and pass to the limit as $\delta\rightarrow 0$ to get the limit function $(u_\epsilon, v_\epsilon, w_\epsilon, z_\epsilon)$ is indeed a solution of the problem $(\mathcal{P}_\epsilon)$. So, we can write
	\begin{itemize}
		\item[(i)] $u_{\epsilon\delta}\rightharpoonup u_\epsilon$ weakly in $L^2(S; H^{1,2}(\Omega_\epsilon^p))$,\quad (ii) $\frac{\partial u_{\epsilon\delta}}{\partial t} \rightharpoonup \frac{\partial u_\epsilon}{\partial t}$ weakly in $L^2(S ; H^{1,2}(\Omega_\epsilon^p)^*)$,
		\item[(iii)]	$v_{\epsilon\delta}\rightharpoonup v_\epsilon$ weakly in $L^2(S; H^{1,2}(\Omega_\epsilon^p))$,\quad (iv) $\frac{\partial v_{\epsilon\delta}}{\partial t} \rightharpoonup \frac{\partial v_\epsilon}{\partial t}$ weakly in $L^2(S; H^{1,2}(\Omega_\epsilon^p)^*)$,
		\item[(v)]$w_{\epsilon\delta}\rightharpoonup w_\epsilon$ weakly in $L^2(S\times\Gamma^*_\epsilon)$,\quad\quad\; (vi) $\frac{\partial w_{\epsilon\delta}}{\partial t} \rightharpoonup \frac{\partial w_\epsilon}{\partial t}$ weakly in $L^2(S\times{\Gamma^*_\epsilon})$,
		\item[(vii)]$z_{\epsilon\delta}\overset{*}{\rightharpoonup} z_\epsilon$ weakly* in $L^{\infty}(S\times{\Gamma^*_\epsilon})$.
	\end{itemize}
	
	By Corollary 4 and Lemma 9 of \cite{simon1986compact} we obtain, for $s\in (0, 1)$, 
	\begin{align*}
		&u_{\epsilon\delta}\rightarrow u_\epsilon\text{ strongly in } L^2((0,T); L^2(\Omega_\epsilon^p)) ,\\
		&u_{\epsilon\delta}\rightarrow u_\epsilon\text{ strongly in } L^2((0,T); H^{s,2}(\Omega_\epsilon^p))\cap C([0,T]; H^{-s,2}(\Omega_\epsilon^p)),\\
		&v_{\epsilon\delta}\rightarrow v_\epsilon\text{ strongly in } L^2((0,T); L^2(\Omega_\epsilon^p)) \; \text{ and }\\
		& v_{\epsilon\delta}\rightarrow v_\epsilon\text{ strongly in } L^2((0,T); H^{s,2}(\Omega_\epsilon^p))\cap C([0,T]; H^{-s,2}(\Omega_\epsilon^p)).
	\end{align*}
	Then, by trace theorem (cf. Satz 8.7 of \cite{wloka1982partielle}), we have
	\begin{align*}
		&u_{\epsilon\delta}\rightarrow u_\epsilon\text{ strongly in }L^2({\Gamma^*_\epsilon})^t, v_{\epsilon\delta}\rightarrow v_\epsilon\text{ strongly in }L^2({\Gamma^*_\epsilon})^t.
	\end{align*}
	Since $R$ is Lipschitz, $R(u_{\epsilon\delta}, v_{\epsilon\delta}) \rightarrow R(u_\epsilon, v_\epsilon)$ strongly in $L^2({\Gamma^*_\epsilon}^t)$ and pointwise a.e. in $(\Gamma^*_\epsilon)^t$. Now we follow the same arguments given in Theorem 2.21 of \cite{van2004crystal} to derive the system of equations $\eqref{eqn:M11}-\eqref{eqn:MN3}$.
	\begin{lemma}\label{lemma:E1}
		There exists an extension of the solution $(\textbf{q}_\epsilon, u_\epsilon,v_\epsilon)$, still denoted by the same symbol, of $(\mathcal{P}_\epsilon)$ into all of $S\times\Omega$ which satisfies
		\begin{align}\label{eqn:EPB}
			&\sup_{\epsilon>0}\left( \|\textbf{q}_\epsilon\|_{L^\infty(S; \textbf{L}^2(\Omega))}+\epsilon\sqrt{\mu}\|\nabla\textbf{q}_\epsilon\|_{L^2(S\times\Omega)^{n\times n}}+\|\partial_t \textbf{q}_\epsilon\|_{L^2(S; \textbf{H}^{-1,2}_{div}(\Omega))}+\|u_{\epsilon}\|_{(\Omega)^T}+\|\nabla u_{\epsilon}\|_{(\Omega)^T}\right.\notag\\
			&+\left\lVert \chi_\epsilon\frac{\partial u_{\epsilon} }{\partial t}\right\rVert_{L^{2}(S; H^{1,2}(\Omega)^{*})}+\|v_{\epsilon}\|_{(\Omega)^T}
			\left. +\|\nabla v_{\epsilon}\|_{(\Omega)^T}+\left\lVert \chi_\epsilon\frac{\partial v_{\epsilon} }{\partial t}\right\rVert_{L^2(S; H^{1,2}(\Omega)^*)}+\sup_{t\in S}\|P_\epsilon(t)\|_{L^2_0(\Omega)}\right) \le C<\infty, 
		\end{align}where $C$ is independent of $\epsilon$. 
	\end{lemma}
	\begin{proof}
		We use Lemma \ref{lemma:EL1} and Lemma \ref{lemma:M1LB} to get the inequalities
		\begin{align*}
			&\|u_{\epsilon}\|_{(\Omega)^T}+\|\nabla u_{\epsilon}\|_{(\Omega)^T}+\|v_{\epsilon}\|_{(\Omega)^T}+\|\nabla v_{\epsilon}\|_{(\Omega)^T}\le C \left \{ \|u_{\epsilon}\|_{L^2(S;H^{1,2}(\Omega^p_\epsilon))}+\|v_{\epsilon}\|_{L^2(S; H^{1,2}(\Omega^p_\epsilon))} \right\}=C
		\end{align*}
		and 
		\begin{align*}
			\|\textbf{q}_\epsilon\|_{L^\infty(S; \textbf{L}^2(\Omega))}+\epsilon\sqrt{\mu}\|\nabla\textbf{q}_\epsilon\|_{L^2(S\times\Omega)^{n\times n}}+\|\partial_t \textbf{q}_\epsilon\|_{L^2(S; \textbf{H}^{-1,2}_{div}(\Omega))}\le C.
		\end{align*}
		The variational formulation of \eqref{eqn:M11} is given by
		\begin{align*}
			&\int_{0}^{T}\langle \chi_\epsilon\frac{\partial u_\epsilon (t)}{\partial t}, \phi(t)\psi\rangle_{H^{1,2}(\Omega)^*\times H^{1,2}(\Omega)} dt+\int_{0}^{T}\int_{\Omega}\chi_\epsilon(x) \phi(t) [D_1(x,\frac{x}{\epsilon})\nabla u_\epsilon-\textbf{q}_\epsilon u_\epsilon] \nabla\psi dxdt\\
			&+\epsilon\int_{0}^{T}\int_{\Gamma_\epsilon^*}\frac{\partial w_\epsilon}{\partial t}\phi(t)\psi(x)d\sigma_xdt+\int_{0}^{T}\int_{\partial\Omega_{in}}d\phi\psi d\sigma_xdt+\int_{0}^{T}\int_{\partial\Omega_{out}}e\phi\psi d\sigma_xdt=0,
		\end{align*}
		for all $\phi\in H^{1,2}_0(S)$ and $\psi\in H^{1,2}(\Omega)$. Then, we can simplify it as
		\begin{align*}
			\left| \int_{0}^{T}\langle \chi_\epsilon\frac{\partial u_\epsilon (t)}{\partial t}, \phi(t)\psi\rangle_{H^{1,2}(\Omega)^*\times H^{1,2}(\Omega)} dt\right|\le C_{12}+(C+\frac{1}{2})\|\phi\|^2_{L^2(S)}\|\psi\|^2_{H^{1,2}(\Omega)},
		\end{align*}
		where $C_{12}=\frac{1}{2}\left(M\|\nabla u_\epsilon\|^2_{L^2(S\times\Omega)}+C+\left\lVert\frac{\partial w_\epsilon}{\partial  t}\right\rVert^2_{L^2(S\times\Gamma_\epsilon^*)}+\|d\|^2_{L^2(S\times\partial\Omega_{in})}+\|e\|^2_{L^2(S\times\partial\Omega_{out})}\right)$. We choose $\phi\in H^{1,2}_0(S)$, therefore $\|\phi\|_{L^2(S)}\le C_0\|\phi\|_{H^{1,2}_0(S)}\implies \left\lVert\frac{\phi}{C_0}\right\rVert_{L^2(S)}\le\|\phi\|_{H^{1,2}_0(S)}$, where $C_0$ is the embedding constant. Now, taking supremum on both sides of the above inequality, we get
		\begin{align*}
			&C_0\sup_{\begin{subarray} {c}\|\frac{\phi}{C_0}\|_{L^2(S)}\le 1 \\ \|\psi\|_{H^{1,2}(\Omega)\le 1}\end{subarray}} \left| \int_{0}^{T}\langle \chi_\epsilon\frac{\partial u_\epsilon (t)}{\partial t}, \frac{\phi(t)}{C_0}\psi\rangle_{H^{1,2}(\Omega)^*\times H^{1,2}(\Omega)} dt\right|\\ &\le C_{12} + C_0^2(C+\frac{1}{2})\sup_{\begin{subarray} {c}\|\frac{\phi}{C_0}\|_{L^2(S)}\le 1 \\ \|\psi\|_{H^{1,2}(\Omega)\le 1}\end{subarray}} \left\lVert\frac{\phi}{C_0}\right\rVert^2_{L^2(S)}\|\psi\|^2_{H^{1,2}(\Omega)}
			\implies \left\lVert \chi_\epsilon\frac{\partial u_\epsilon (t)}{\partial t}\right\rVert_{L^{2}(S; H^{1,2}(\Omega)^{*})}\le C.
		\end{align*}
		Proceeding in the same way with \eqref{eqn:M21}, we will have $\left\lVert \chi_\epsilon\frac{\partial v_\epsilon (t)}{\partial t}\right\rVert_{L^{2}(S; H^{1,2}(\Omega)^{*})}\le C$.
			We follow the same line of arguments as given in Lemma 5.1 of \cite{bavnas2017homogenization} and use \eqref{eqn:pe} to extend the pressure $P_\epsilon$ to the whole domain $\Omega$, i.e.,
			\begin{align*}
				\sup_{t\in S}\|P_\epsilon(t)\|_{L^2_0(\Omega)}\le C.
			\end{align*}
	\end{proof}
	\begin{lemma}\label{lemma:APB}
		The a-priori estimates \eqref{eqn:EPB} and \eqref{eqn:UB} lead to the following convergence results:
		\begin{itemize}
			\item[(i)] $\textbf{q}_\epsilon\overset{2}{\rightharpoonup} \textbf{q}$ in $L^2(S\times\Omega; H^{1,2}_{per}(Y))^n$,\; (ii) $\epsilon\nabla_x\textbf{q}_\epsilon \overset{2}{\rightharpoonup}\nabla_y \textbf{q}$ in $L^2(S\times\Omega; H^{1,2}_{per}(Y))^{n\times n}$,
			\item[(iii)] $u_{\epsilon}\rightharpoonup u$  in $L^2(S; H^{1,2}(\Omega))$,\quad (iv) $\frac{\partial u_{\epsilon}}{\partial t} \rightharpoonup \frac{\partial u}{\partial t}$  in $L^2(S ; H^{1,2}(\Omega)^*)$,
			\item[(v)]	$v_{\epsilon}\rightharpoonup v$  in $L^2(S; H^{1,2}(\Omega))$,\quad (vi) $\frac{\partial v_{\epsilon}}{\partial t} \rightharpoonup \frac{\partial v}{\partial t}$  in $L^2(S; H^{1,2}(\Omega)^*)$,
			\item[(vii)] $u_\epsilon\rightarrow u$ in $L^2(S\times\Omega)$, \quad\quad (viii)  $v_\epsilon\rightarrow v$ in $L^2(S\times\Omega)$, 
			\item[(ix)]  $u_{\epsilon}\rightarrow u$ in $L^2(S\times\Gamma_\epsilon^*)$, \quad\quad (x)  $v_{\epsilon}\rightarrow v$ in $L^2(S\times\Gamma_\epsilon^*)$,
			\item[(xi)] there exists $u_1\in L^2(S\times\Omega ; H^{1,2}_{per}(Y)/\mathbb{R})$ such that $u_\epsilon\overset{2}{\rightharpoonup}u$ and $\nabla u_\epsilon \overset{2}{\rightharpoonup} \nabla_x u + \nabla_y u_1$,
			\item[(xii)] there exists $v_1\in L^2(S\times\Omega ; H^{1,2}_{per}(Y)/\mathbb{R})$ such that $v_\epsilon\overset{2}{\rightharpoonup}v$ and $\nabla v_\epsilon \overset{2}{\rightharpoonup} \nabla_x v + \nabla_y v_1$,
			\item[(xiii)] $w_\epsilon\overset{2}{\rightharpoonup}w$ in $L^2(S\times\Omega\times\Gamma),$ \quad (xiv) $\frac{\partial w_{\epsilon}}{\partial t}\overset{2}{\rightharpoonup} \frac{\partial w}{\partial t}$ in $L^2(S\times\Omega\times\Gamma)$,
			\item[(xv)] $z_\epsilon\overset{2}{\rightharpoonup}z$ in $L^2(S\times\Omega\times\Gamma), $ \quad (xvi) $P_\epsilon\overset{2}{\rightharpoonup}P$ in $L^2(S\times\Omega\times Y)$.
		\end{itemize}	
	\end{lemma} 
	\begin{proof}
		$(i)-(vi)$ and $(xvi)$ follows directly from the a-priori estimate \eqref{eqn:EPB} and Lemma \ref{thm:TS2}. The convergence $(vii)$ and $(viii)$ is a consequence of Lemma \ref{lemma:1} and the estimate \eqref{eqn:EPB}. To establish $(ix)$ and $(x)$, we use the compact embedding $H^{\beta^{'},2}(\Omega)\hookrightarrow\hookrightarrow H^{\beta,2}(\Omega)$ for $\beta\in(\frac{1}{2},1)$ and $0<\beta<\beta^{'}\le 1$. If we denote
		\begin{align*}
			U=\{u\in L^2(S; H^{1,2}(\Omega))\text{  and  } \frac{\partial u}{\partial t}\in L^2(S\times\Omega)\},
		\end{align*}
		then, by Lions-Aubin lemma for fixed $\epsilon$, $U\hookrightarrow\hookrightarrow L^2(S;H^{\beta,2}(\Omega))$. Application of the trace inequality \eqref{eqn:TI3} gives
		\begin{align*}
			\|u_\varepsilon-u\|_{L^2(S\times\Gamma^*_\varepsilon)}
			&\le C_0\|u_\varepsilon-u\|_{L^2(S;H^{\beta,2}(\Omega^p_\varepsilon))}\le C\|u_\varepsilon-u\|_{L^2(S;H^{\beta,2}(\Omega))}\rightarrow0\text{ as }\varepsilon\rightarrow 0.
		\end{align*}
		Similarly, $V=\{v\in L^2(S; H^{1,2}(\Omega))\text{  and  } \frac{\partial v}{\partial t}\in L^2(S\times\Omega)\}\hookrightarrow\hookrightarrow L^2(S;H^{\beta,2}(\Omega))$, therefore $\|v_\varepsilon-v\|_{L^2(S\times\Gamma^*_\varepsilon)}\le C \|u_\varepsilon-u\|_{L^2(S;H^{\beta,2}(\Omega))}\rightarrow0\text{ as }\varepsilon\rightarrow 0.$ The results $(xi)-(xiv)$ follows from Lemma \ref{thm:TS2} and Lemma \ref{thm:TS3}. We get $(xv)$ by using Lemma \ref{thm:TS3} and the following inequality:
		\begin{align*}
			\epsilon\int_{0}^{T}\int_{\Gamma_\epsilon^*}|z_\epsilon|^2d\sigma_xdt \le  T \|z_\epsilon\|_{L^\infty(S\times\Gamma_\epsilon^*)}\sup_{\epsilon>0}\epsilon|\Gamma_\epsilon^*|\le T\|z_\epsilon\|_{L^\infty(S\times\Gamma_\epsilon^*)}\frac{|\Gamma||\Omega|}{|Y|}<\infty.
		\end{align*}
	\end{proof} 
	\begin{lemma}\label{lemma:R2S}
		\begin{enumerate}
			\item [\textrm{(i)}] The nonlinear terms $\textbf{q}_\epsilon u_\epsilon$ and $\textbf{q}_\epsilon v_\epsilon$ two scale converges to $\textbf{q} u$ and $\textbf{q} v$ in $L^2(S\times\Omega\times Y),$ respectively. 
			\item [\textrm{(ii)}] The reaction rate term $\{R(u_\epsilon,v_\epsilon)\}$ is strongly convergent to $R(u,v)$ in\\ $L^2(S\times\Gamma_\epsilon^*)$. From this we can deduce $R(u_\epsilon,v_\epsilon)\overset{2}{\rightharpoonup} R(u,v)$ in $L^2(S\times\Omega\times\Gamma)$.
			\item [\textrm{(iii)}] $\{\mathcal{T}_\epsilon^b(w_\epsilon)\}$ is strongly convergent to $w$ in $L^2(S\times\Omega\times\Gamma)$.
		\end{enumerate}
	\end{lemma}
	\begin{proof}
		$(i)$ We want to establish $\textbf{q}_\epsilon u_\epsilon \overset{2}{\rightharpoonup} \textbf{q} u$ in $L^2(S\times\Omega\times Y)$ that means by $(iii)$ of Lemma \ref{lemma:uo} we have to show $\mathcal{T}_\epsilon(\textbf{q}_\epsilon u_\epsilon)=\mathcal{T}_\epsilon(\textbf{q}_\epsilon)\mathcal{T}_\epsilon(u_\epsilon)\rightharpoonup \textbf{q} u$ in $L^2(S\times\Omega\times Y)$. Let $\phi\in C_0^\infty(S\times\Omega; C^\infty_{per}(Y))$, then we calculate
		\begin{align*}
			&\left\lvert\int_{0}^{T}\int_{\Omega\times Y}(\mathcal{T}_\epsilon(\textbf{q}_\epsilon)\mathcal{T}_\epsilon(u_\epsilon)-\textbf{q} u)\phi dxdydt\right\rvert\\
			&\le \|\phi\|_{L^\infty(S\times\Omega\times Y)}\|\mathcal{T}_\epsilon(\textbf{q}_\epsilon)\|_{L^2(S\times\Omega\times Y)}\|\mathcal{T}_\epsilon(u_\epsilon)-u\|_{L^2(S\times\Omega\times Y)}+\left\lvert\int_{0}^{t}\int_{\Omega\times Y}(\mathcal{T}_\epsilon(\textbf{q}_\epsilon)-q)\phi u dxdydt\right\rvert\\
			&\rightarrow 0 \text{  as   } \epsilon\rightarrow 0.
		\end{align*}	
		Here we use the norm preserving property of $\mathcal{T}_\epsilon(\textbf{q}_\epsilon)$, given in $(ii)$ of Lemma \ref{lemma:uo}. Then by Lemma \ref{lemma:APB}, we know the strong convergence of $u_\epsilon$ and two scale convergence of $\textbf{q}_\epsilon$. Therefore from $(iii)$ of Lemma \ref{lemma:uo}, we can conclude that $\mathcal{T}_\epsilon(u_\epsilon)\rightarrow u$ in $L^2(S\times\Omega\times Y)$ and $\mathcal{T}_\epsilon(\textbf{q}_\epsilon)\rightharpoonup\textbf{q}$ in $L^2(S\times\Omega\times Y)$. Hence, $\textbf{q}_\epsilon u_\epsilon\overset{2}{\rightharpoonup}\textbf{q}u$ in $L^2(S\times\Omega\times Y)$. Similarly, $\textbf{q}_\epsilon v_\epsilon\overset{2}{\rightharpoonup}\textbf{q}v$ in in $L^2(S\times\Omega\times Y)$.\\
		$(ii)$ The proof  follows from Lemma 6.3 of \cite{ghosh2022diffusion}.\\
		$(iii)$ We unfold the ODE \eqref{eqn:MN1} by using the boundary unfolding operator \eqref{eqn:BU}. Next changing the variable $x=\epsilon y+\epsilon k$(for $x\in\Gamma_\epsilon^*$) to the fixed domain $S\times\Omega\times\Gamma$, we have
		\begin{align}
			\mathcal{T}^b_\epsilon(\frac{\partial w_\epsilon}{\partial t})=\frac{\partial}{\partial t}(	\mathcal{T}^b_\epsilon(w_\epsilon))=k_d(\mathcal{T}^b_\epsilon(R(u_\epsilon, v_\epsilon))-\mathcal{T}^b_\epsilon(z_\epsilon)), \text{  where  }\label{eqn:OU}\\
			\mathcal{T}^b_\epsilon(z_\epsilon)\in \psi(\mathcal{T}^b_\epsilon(w_\epsilon))=\mathcal{T}^b_\epsilon\psi(w_\epsilon)=\begin{cases}
				\{0\} \text{  if  } \mathcal{T}^b_\epsilon w_\epsilon<0,\\
				[0, 1] \text{  if  } \mathcal{T}^b_\epsilon w_\epsilon=0,\\
				\{1\}  \text{  if  } \mathcal{T}^b_\epsilon w_\epsilon>0
			\end{cases} \label{eqn:MDU}
		\end{align}  
		The rest part of the proof follows from $(ii)$ of Lemma 6.3 of \cite{ghosh2022diffusion}.
	\end{proof}
	
	\textbf{Homogenization of ($\mathcal{P}_\epsilon$), i.e. $\epsilon\rightarrow 0$ part:  }
	We make use of the two-scale convergence techniques to obtain the macroscopic equations $\eqref{eqn:qs}-\eqref{eqn:M12S}$. 
	We first upscale the unsteady Stokes equations. We choose $\psi\in C^\infty_0(\Omega ; C^\infty_{per}(Y))^n, \zeta\in C^\infty_0(S)$ in \eqref{eqn:p2} and pass the limit $\epsilon\rightarrow 0$ using Lemma \ref{lemma:APB}
	to derive
	\begin{align}\label{eqn:p3}
		\int_{0}^{T}\int_{\Omega\times Y^p} P(t,x,y)\nabla_y.\psi \partial_t\zeta dxdydt=0.
	\end{align}
	Therefore from \eqref{eqn:p3} it implies that the two scale limit $P(t, x, y)$ does not depend on $y$, i.e., $P(t)\in L^2_0(\Omega)$. Next, we pick $\psi\in C^\infty_0(\Omega ; C^\infty_{per}(Y))^n$ such that $\nabla_y.\psi=0$ in \eqref{eqn:p2} and obtain by Lemma \ref{lemma:APB} that
	\begin{align*}
		-\int_{0}^{T}\int_{\Omega\times Y^p} \textbf{q}(t, x, y)\psi(x, y)\partial_t\zeta dxdydt&+\mu \int_{0}^{T}\int_{\Omega\times Y^p} \nabla_y\textbf{q}(t, x, y): \nabla_y \psi(x, y)\zeta(t)dxdydt\\
		&=-\int_{0}^{T}\int_{\Omega\times Y^p} P(t,x)\nabla_x.\psi(x,y)\partial_t\zeta dxdydt.
	\end{align*}
	Moreover, following the same line of the proof of \cite{bavnas2017homogenization}, we have there exists a pressure $P_1\in L^\infty(S; L^2_0(\Omega); L^2_{per}(Y^p))$ such that
	\begin{align}\label{eqn:p4}
		-\int_{0}^{T}\int_{\Omega\times Y^p} \textbf{q}(t, x, y)\psi(x, y)\partial_t\zeta dxdydt+\mu \int_{0}^{T}\int_{\Omega\times Y^p} \nabla_y\textbf{q}(t, x, y): \nabla_y \psi(x, y)\zeta(t)dxdydt\notag\\
		=-\int_{0}^{T}\int_{\Omega\times Y^p} P(t,x)\nabla_x.\psi(x,y)\partial_t\zeta dxdydt-\int_{0}^{T}\int_{\Omega\times Y^p} P_1(t,x, y)\nabla_y.\psi(x,y)\partial_t\zeta dxdydt,
	\end{align}
	for all $\psi\in C^\infty_0(\Omega ; C^\infty_{per}(Y))^n$ and $\zeta\in C^\infty_0(S)$. It follows from \eqref{eqn:p4} that the first equation of \eqref{eqn:qs} holds in the distributional sense with $p=\partial_tP, p_1=\partial_tP_1$. The remaining equations of \eqref{eqn:qs} follows from \cite{allaire1992homogenization1}.
	Next, we take the test functions $\Phi_i(t,x,\frac{x}{\epsilon})=\phi_i(t,x)+\epsilon\psi_i(t,x,\frac{x}{\epsilon})$, where $\phi_i\in C^\infty_0(S\times\Omega)$ and $\psi_i(t,x,\frac{x}{\epsilon})\in C^\infty_0(S\times\Omega ; C^\infty_{per}(Y))$, for $i=1,2$. Now, multiplying \eqref{eqn:M11} by $\Phi_1$, we have
	\begin{align*}
		&\underbrace{\int_{0}^{T}\int_{\Omega_\epsilon^p}\frac{\partial u_\epsilon}{\partial t}[\phi_1(t,x)+\epsilon\psi_1(t,x,\frac{x}{\epsilon})]dxdt}_{I_1}+\underbrace{\int_{0}^{T}\int_{\Omega_\epsilon^p} [D_1(x,\frac{x}{\epsilon})\nabla u_\epsilon-\textbf{q}_\epsilon u_\epsilon] [\nabla_x \phi_1+\epsilon\nabla_x \psi_1+\nabla_y\psi_1]dxdt}_{I_2}\\
		&+\underbrace{\epsilon\int_{0}^{T}\int_{\Gamma^*_\epsilon}\frac{\partial w_\epsilon}{\partial t}[\phi_1(t,x)+\epsilon\psi_1(t,x,\frac{x}{\epsilon})]d\sigma_xdt}_{I_3}=0.
	\end{align*}
	We pass to the two scale limit to each term separately. This gives
	\begin{align*}
		\lim_{\epsilon\rightarrow 0} I_1&=-\lim_{\epsilon\rightarrow 0}\int_{0}^{T}\int_{\Omega_\epsilon^p}u_\epsilon(t,x)[\frac{\partial \phi_1}{\partial t}+\epsilon\frac{\partial \psi_1}{\partial t}]dxdt\\
		&=-\lim_{\epsilon\rightarrow 0}\int_{0}^{T}\int_{\Omega}\chi(\frac{x}{\epsilon})u_\epsilon(t,x)[\frac{\partial \phi_1}{\partial t}+\epsilon\frac{\partial \psi_1}{\partial t}]dxdt=-\int_{0}^{T}\int_{\Omega}\int_{Y} \chi(y)u(t,x)\frac{\partial \phi_1}{\partial t}dxdydt,\\
		\lim_{\epsilon\rightarrow 0} I_2&=\lim_{\epsilon\rightarrow 0}\int_{0}^{T}\int_{\Omega}\chi(\frac{x}{\epsilon})[D_1(x,\frac{x}{\epsilon})\nabla u_\epsilon-\textbf{q}_\epsilon u_\epsilon] [\nabla_x \phi_1+\epsilon\nabla_x \psi_1+\nabla_y\psi_1]dxdt\\
		&=\int_{0}^{T}\int_{\Omega}\int_{Y} \chi(y)(D_1(x,y)(\nabla u(t,x)+\nabla_y u_1(t,x,y))-\textbf{q}u)(\nabla_x \phi_1+\nabla_y \psi_1)dxdydt,\\
		\lim_{\epsilon\rightarrow 0} I_3&=\lim_{\epsilon\rightarrow 0}\epsilon\int_{0}^{T}\int_{\Gamma_\epsilon^*}\frac{\partial w_\epsilon}{\partial t}[\phi_1(t,x)+\epsilon\psi_1(t,x,\frac{x}{\epsilon})]d\sigma_xdt=\int_{0}^{T}\int_{\Omega}\int_{\Gamma}\frac{\partial w}{\partial t}\phi_1 dxd\sigma_ydt.
	\end{align*}
	Combing all these, we get
	\begin{align}\label{eqn:HE1}
		-\int_{0}^{T}\int_{\Omega}u(t,x)\frac{\partial \phi_1}{\partial t}dxdt+\frac{1}{|Y^p|}\int_{0}^{T}\int_{\Omega}\int_{Y^p} (D_1(x, y)(\nabla u&+\nabla_y u_1)-\textbf{q}u)(\nabla \phi_1+\nabla_y \psi_1)dxdydt\notag\\
		&+\int_{0}^{T}\int_{\Omega}\int_{\Gamma}\frac{1}{|Y^p|}\frac{\partial w}{\partial t}\phi_1 dxd\sigma_ydt=0.
	\end{align} 
	We decompose the equation \eqref{eqn:HE1} to achieve the homogenized equation and the cell problem. Setting $\phi_1\equiv0$ yields
	\begin{align*}
		\int_{0}^{T}\int_{\Omega}\int_{Y^p} (D_1(x,y)(\nabla u(t,x)+\nabla_y u_1(t,x,y)-\textbf{q}u)\nabla_y \psi_1dxdydt=0.
	\end{align*} 
	Choosing $u_1(t,x,y)=\sum_{j=1}^{n}\frac{\partial u}{\partial x_j}(t,x)k_j(t,x,y)+k_0(t, x, y)u(t, x)+p(x)$, we have for each $j=1,2,...,n,\; k_j(t,x,y)\in L^\infty(S\times\Omega ; H^{1,2}_{per}(Y))$ is the $Y-$periodic solution of the cell problems
	\begin{align}\label{eqn:CE1}
		\begin{cases}
			-(div)_y(D_1(x,y)(\nabla_y k_j+e_j+\nabla_y k_0 -\textbf{q}))=0 \text{ for all } y\in Y^p,\\
			-D_1(x,y)(\nabla_y k_j+e_j+\nabla_y k_0 -\textbf{q}).\vec{n}=0 \text{ on } \Gamma,\\
			y \mapsto k_j(y) \text{ is } Y-\text{periodic}.
		\end{cases}
	\end{align}  
	Now, $\psi_1\equiv 0$ implies
	\begin{align*}
		-\int_{0}^{T}\int_{\Omega}u(t,x)\frac{\partial \phi_1}{\partial t}dxdt+\int_{0}^{T}\int_{\Omega}\int_{Y^p} \frac{D_1(x,y)}{|Y^p|}(\nabla u(t,x)+\nabla_y u_1(t,x,y))\nabla_x \phi_1dxdydt\notag\\
		-\frac{1}{|Y^p|} \int_{0}^{T}\int_{\Omega}\int_{Y^p} \textbf{q}u\nabla_x \phi_1dxdydt+\int_{0}^{T}\int_{\Omega} P(t,x)\phi_1 dxdt=0.
	\end{align*}
	With the choice of $u_1$, the above equation takes the form
	\begin{align*}
		\int_{0}^{T}\int_{\Omega}\frac{\partial u(t,x) }{\partial t}\phi_1dxdt+\int_{0}^{T}\int_{\Omega}\int_{Y^p} \frac{D_1(x,y)}{|Y^p|}\left(\delta_{ij}+\sum_{i,j=1}^{n}\frac{\partial k_j}{\partial y_i}\right)\frac{\partial u(t,x)}{\partial x_j}\nabla_x \phi_1dxdydt\notag\\
		-\frac{1}{|Y^p|} \int_{0}^{T}\int_{\Omega}\int_{Y^p}[\textbf{q}u-D_1(x, y)\nabla_y k_0 u]\nabla_x\phi_1 dxdydt+\int_{0}^{T}\int_{\Omega} P(t,x)\phi_1 dxdt=0.
	\end{align*}
	Upon further simplification, the above equation reduces to
	\begin{align*}
		\int_{0}^{T}\int_{\Omega}\left[\frac{\partial u(t,x) }{\partial t}-\nabla.(A\nabla u-(\bar{\textbf{q}}-\tilde{\textbf{q}}_0)u)+P(t,x)\right]\phi_1dxdt=0,\text{ for all }\phi_1\in C^\infty_0(S\times\Omega),
	\end{align*}
	where
	\begin{align*}
		&A=(a_{ij})_{1\le i,j \le n}=\int_{Y^p}\frac{D_1(x, y)}{|Y^p|}\left(\delta_{ij}+\sum_{i,j=1}^{n}\frac{\partial k_j}{\partial y_i}\right)dy, \quad P(t,x)=\int_{\Gamma}\frac{1}{|Y^p|}\frac{\partial w}{\partial t}\;d\sigma_y,\\
		&\bar{\textbf{q}}=\frac{1}{|Y^p|}\int_{Y^p} \textbf{q}(t, x, y)dy, \quad \tilde{\textbf{q}}_0=\frac{1}{|Y^p|}\int_{Y^p} D_1(x, y)\nabla_y k_0 dy.
	\end{align*}
	Hence, the homogenized equation for the mobile species $A$ is given by \eqref{eqn:M1S}. Similarly, testing \eqref{eqn:M21} by $\Phi_2$ implies
	\begin{align*}
		\frac{\partial v}{\partial {t}} - \nabla.(B\nabla v-(\bar{\textbf{q}}-\tilde{\textbf{q}}_0)v)+P(t,x)&=0 \; \text{ in } S\times \Omega,	\\
		-(B\nabla v-(\bar{\textbf{q}}-\tilde{\textbf{q}}_0)v).\vec{n}&=\frac{g}{|Y^p|} \text{ on } S\times \partial\Omega_{in},\\
		-(B\nabla v-(\bar{\textbf{q}}-\tilde{\textbf{q}}_0)v).\vec{n}&=\frac{h}{|Y^p|} \text{ on } S\times \partial\Omega_{out},\\
		v(0,x)&=v_0(x)\; \text{ in }\Omega,
	\end{align*}
	where
	\begin{align*}
		B=(b_{ij})_{1\le i,j \le n}=\int_{Y^p}\frac{D_2(x,y)}{|Y^p|}\left(\delta_{ij}+\sum_{i,j=1}^{n}\frac{\partial k_j}{\partial y_i}\right)dy, \quad P(t,x)=\int_{\Gamma}\frac{1}{|Y^p|}\frac{\partial w}{\partial t}\;d\sigma_y
	\end{align*}
	and $k_j(t,x,y)\in L^\infty(S\times\Omega ; H^{1,2}_{per}(Y))$ is the solution of the cell problems
	\begin{align}\label{eqn:CE2}
		\begin{cases}
			-(div)_y(D_2(x,y)(\nabla_y k_j+e_j+\nabla_y k_0 -\textbf{q}))=0 \text{ for all } y\in Y^p,\\
			-D_2(x,y)(\nabla_y k_j+e_j+\nabla_y k_0 -\textbf{q}).\vec{n}=0 \text{ on } \Gamma,\\
			y \mapsto k_j(y) \text{ is } Y-\text{periodic},
		\end{cases}
	\end{align} 
	for $j=1,2,...,n$. Finally, we pass the two-scale limit to the ode \eqref{eqn:MN1}. We choose the test function $\psi(t,x,\frac{x}{\epsilon})\in C^\infty_0(S\times\Omega ; C^\infty_{per}(Y))$ and obtain
	\begin{align*}
		\int_{0}^{T}\int_{\Gamma_\epsilon^*}\frac{\partial w_\epsilon}{\partial t}\psi(t,x,\frac{x}{\epsilon})d\sigma_xdt=k_d\int_{0}^{T}\int_{\Gamma_\epsilon^*}(R(u_\epsilon,v_\epsilon)-z_\epsilon)\psi(t,x,\frac{x}{\epsilon})d\sigma_xdt.
	\end{align*}
	Lemma \ref{lemma:APB} and Lemma \ref{lemma:R2S} leads to
	\begin{align*}
		\int_{0}^{T}\int_{\Omega}\int_{\Gamma} \frac{\partial w}{\partial t}\psi(t,x,y)dxd\sigma_ydt=k_d\int_{0}^{T}\int_{\Omega}\int_{\Gamma}(R(u,v)-z)\psi(t,x,y)dxd\sigma_ydt.
	\end{align*}
	Hence, the macroscopic equation for the ode is
	\begin{align*}
		\frac{\partial w}{\partial  {t}}&=k_d(R(u,v)-z) \; \text{ in } S\times\Omega\times \Gamma.
	\end{align*}
	Now, we need to characterize the two-scale limit of the multi-valued dissolution rate term. According to the Lemma \ref{lemma:R2S}, $\mathcal{T}^b_\epsilon(w_\epsilon)\rightarrow w$ in $L^2(S\times\Omega\times\Gamma)$. By corollary on page 53 of \cite{yosidafunctional}, there exists a subsequence $\{\mathcal{T}_\epsilon^b(w_{\epsilon})\}$(still denoted by the same symbol) pointwise convergent to $w$ a.e. in $S\times\Omega\times\Gamma$. i.e. 
	\begin{align*}
		\lim_{\epsilon\rightarrow 0} \mathcal{T}_\epsilon^b(w_\epsilon)=w \text{ a.e. in }S\times\Omega\times\Gamma.
	\end{align*}
	As $\mathcal{T}_\epsilon^b(w_\epsilon)\ge 0$, $w\ge0$ so we have to consider two different cases\\
	Case 1: Let $w>0$ pointwise a.e. in $S\times\Omega\times\Gamma$. 
	As $w$ is the pointwise limit of $\mathcal{T}_\epsilon^b(w_\epsilon)$, therefore for any $\eta>0$ there exists a $\delta>0$ such that $|\epsilon|<\delta\implies |\mathcal{T}_\epsilon^b(w_\epsilon)-w|<\eta$. Choosing $\eta=\frac{|w|}{2}=\frac{w}{2}$ gives $\mathcal{T}_\epsilon^b(w_\epsilon)>\frac{w}{2}$. By definition \eqref{eqn:MDU}, we have $\mathcal{T}_\epsilon^bz_\epsilon=1$. This means $\mathcal{T}_\epsilon^bz_\epsilon \rightharpoonup 1$. Then by Lemma \ref{thm:PU1}, $z_\epsilon\overset{2}{\rightharpoonup} 1$ in $L^2(S\times\Omega\times\Gamma)$, whereas Lemma \ref{lemma:APB} implies that $z_\epsilon\overset{2}{\rightharpoonup} z$ in $L^2(S\times\Omega\times\Gamma)$. Consequently, $z=1$.\\
	Case 2: Let $w=0$ pointwise a.e. in $S\times\Omega\times\Gamma$. 
	Since $\frac{\partial w_\epsilon}{\partial t}\in L^2(S\times\Gamma_\epsilon^*)$, for a test function $\phi\in C_0^\infty(S\times\Omega)$ we get
	\begin{align*}
		\int_{0}^{T}\int_{\Gamma_\epsilon^*}\frac{\partial w_\epsilon}{\partial t}\phi d\sigma_xdt=-\int_{0}^{T}\int_{\Gamma^*_\epsilon} w_\epsilon\frac{\partial \phi}{\partial t}d\sigma_xdt\rightarrow 0 \text{  as   } \epsilon\rightarrow 0 \text{ because } \lim_{\epsilon\rightarrow 0} w_\epsilon=0.
	\end{align*}
	As $\phi$ is arbitrary, it follows that $\frac{\partial w_\epsilon}{\partial t}\overset{2}{\rightharpoonup} 0$ which means $\frac{\partial}{\partial t}(\mathcal{T}_\epsilon^b(w_\epsilon))\rightharpoonup 0$ in $L^2(S\times\Omega\times\Gamma)$. Then,
	\begin{align*}
		\int_{0}^{T}\int_{\Omega}\int_{\Gamma} \frac{\partial}{\partial t}(\mathcal{T}_\epsilon^b(w_\epsilon))\phi dxd\sigma_ydt=k_d\int_{0}^{T}\int_{\Omega}\int_{\Gamma}(\mathcal{T}_\epsilon^b(R(u_\epsilon,v_\epsilon))-\mathcal{T}_\epsilon^bz_\epsilon)\phi dxd\sigma_ydt.
	\end{align*}
	Passing the limit $\epsilon\rightarrow 0$, we are led to
	\begin{align*}
		\int_{0}^{T}\int_{\Omega}\int_{\Gamma} 0\phi dxd\sigma_ydt=k_d\int_{0}^{T}\int_{\Omega}\int_{\Gamma}(R(u,v)-z)dxd\sigma_ydt,
	\end{align*} 
	i.e. $z=R(u,v)$ and $0\le z\le1$ when $w=0$ on $\Gamma$. Since $w_0\in L^2(\Omega)$, $\mathcal{T}_\epsilon^b(w_0)\rightarrow w_0$ in $L^2(S\times\Omega\times\Gamma)$. Hence, the two-scale limit equation for the ODE is given by \eqref{eqn:M12S}.
	\section{Proof the the Theorem \ref{thm:2}}
	\begin{lemma}
		There exists an extension of the solution $(\textbf{q}_\epsilon, u_{\epsilon\delta},v_{\epsilon\delta})$, still denoted by same symbol, of $(\mathcal{P}_{\epsilon\delta})$ into all of $S\times\Omega$ which satisfies
		\begin{align}\label{eqn:NE3}
			&\|\textbf{q}_\epsilon\|_{L^\infty(S; \textbf{L}^2(\Omega))}+\epsilon\sqrt{\mu}\|\nabla\textbf{q}_\epsilon\|_{L^2(S\times\Omega)^{n\times n}}+\|\partial_t \textbf{q}_\epsilon\|_{L^2(S; \textbf{H}^{-1,2}_{div}(\Omega))}+\|u_{\epsilon\delta}\|_{L^2(S;H^{1,2}(\Omega))}\notag\\
			+&\left\lVert \chi_\epsilon \partial_t u_\epsilon\right\rVert_{L^{2}(S; H^{1,2}(\Omega)^{*})}+\|v_{\epsilon\delta}\|_{L^2(S;H^{1,2}(\Omega))}+\left\lVert \chi_\epsilon \partial_t v_\epsilon\right\rVert_{L^{2}(S; H^{1,2}(\Omega)^{*})}+\sup_{t\in S}\|P_\epsilon(t)\|_{L^2_0(\Omega)}\le C,
		\end{align}
		where the constant $C>0$ is independent of $\epsilon$ and $\delta$.
	\end{lemma}
	\begin{proof}
		The proof follows the same lines of the proof of Lemma \ref{lemma:E1}.
	\end{proof}
	The derivation of the homogenized system $\eqref{eqn:qs1}$ of the unsteady Stokes system $\eqref{eqn:v1}-\eqref{eqn:v4}$ comes from repeating the same calculations as stated in $\eqref{eqn:p3}-\eqref{eqn:p4}$ as there is no regularization parameter $\delta>0$ involved. Therefore now we mainly focus on deducing the strong form of the macromodel $\eqref{eqn:M1S2}-\eqref{eqn:M12S2}$. 
	\begin{lemma}\label{lemma:APB2}
		Let $(u_{\epsilon\delta})_{\epsilon>0},(v_{\epsilon\delta})_{\epsilon>0}$ and $(w_{\epsilon\delta})_{\epsilon>0}$ satisfy the hypotheses of the Lemma \ref{lemma:M1LB} and the estimate \eqref{eqn:NE3}. Then, the following convergence holds:
		\begin{itemize}
			\item[(i)] $u_{\epsilon\delta}\rightharpoonup u_\delta$  in $L^2(S; H^{1,2}(\Omega))$,\quad (ii) $\frac{\partial u_{\epsilon\delta}}{\partial t} \rightharpoonup \frac{\partial u_\delta}{\partial t}$  in $L^2(S ; H^{1,2}(\Omega)^*)$,
			\item[(iii)]	$v_{\epsilon\delta}\rightharpoonup v_\delta$  in $L^2(S; H^{1,2}(\Omega))$,\quad (iv) $\frac{\partial v_{\epsilon\delta}}{\partial t} \rightharpoonup \frac{\partial v_\delta}{\partial t}$  in $L^2(S; H^{1,2}(\Omega)^*)$,
			\item[(v)]  $u_{\epsilon\delta}\rightarrow u_\delta$ in $L^2(S\times\Gamma_\epsilon^*)$, \quad\quad (vi)  $v_{\epsilon\delta}\rightarrow v_\delta$ in $L^2(S\times\Gamma_\epsilon^*)$,
			\item[(vii)] There exists  $u_{\delta1}\in L^2(S\times\Omega ; H^{1,2}_{per}(Y)/\mathbb{R})$ such that $u_{\epsilon\delta}\overset{2}{\rightharpoonup}u_\delta$ and $\nabla u_{\epsilon\delta} \overset{2}{\rightharpoonup} \nabla_x u_\delta + \nabla_y u_{\delta1}$,
			\item[(viii)] There exists $v_{\delta1}\in L^2(S\times\Omega ; H^{1,2}_{per}(Y)/\mathbb{R})$ such that $v_{\epsilon\delta}\overset{2}{\rightharpoonup}v_\delta$ and $\nabla v_{\epsilon\delta} \overset{2}{\rightharpoonup} \nabla_x v_\delta + \nabla_y v_{\delta1}$,
			\item[(ix)] $w_{\epsilon\delta}\overset{2}{\rightharpoonup}w_\delta$ in $L^2(S\times\Omega\times\Gamma),$ \quad (x) $\frac{\partial w_{\epsilon\delta}}{\partial t}\overset{2}{\rightharpoonup} \frac{\partial w_\delta}{\partial t}$ in $L^2(S\times\Omega\times\Gamma)$.
		\end{itemize}
	\end{lemma}
	\begin{proof}
		Following the same line of arguments as in the proof of Lemma \ref{lemma:APB} we can obtain the results immediately.
	\end{proof}
	\begin{lemma}\label{lemma:R2S1}
		\begin{enumerate}
			\item [$(i)$] The reaction rate term $R(u_{\epsilon\delta},v_{\epsilon\delta})\rightarrow R(u_\delta,v_\delta)$ strongly in $L^2(S\times\Gamma_\epsilon^*)$. This strong convergence further implies that $R(u_{\epsilon\delta},v_{\epsilon\delta})\overset{2}{\rightharpoonup} R(u_\delta,v_\delta)$ in $L^2(S\times\Omega\times\Gamma)$.
			\item[$(ii)$] $\mathcal{T}_\epsilon^b(w_{\epsilon\delta})\rightarrow w_\delta$ strongly in $L^2(S\times\Omega\times\Gamma)$.
		\end{enumerate}
	\end{lemma}
	\begin{proof}
		$(i)$ We proceed similarly as in Lemma \ref{lemma:R2S} to establish the result. \\
		$(ii)$ We get the unfolded version of the ODE as $\eqref{eqn:OU}-\eqref{eqn:MDU}$ by applying the boundary unfolding operator \eqref{eqn:BU}. Now for all $\mu,m\in\mathbb{N}$ with $\mu>m$ arbitrary the difference $(\mathcal{T}^b_{\epsilon_\mu}w_{\epsilon_{\mu}\delta}-\mathcal{T}^b_{\epsilon_m}w_{\epsilon_{m}\delta})$ satisfies
		\begin{align*}
			\int_{0}^{t}\frac{\partial}{\partial t}\|\mathcal{T}^b_{\epsilon_\mu}w_{\epsilon_{\mu}\delta}-&\mathcal{T}^b_{\epsilon_m}w_{\epsilon_{m}\delta}\|^2_{\Omega\times\Gamma}\\
			=2k_d\int_{0}^{t}\int_{\Omega\times\Gamma}&(\mathcal{T}^b_{\epsilon_\mu}R(u_{\epsilon_{\mu}\delta},v_{\epsilon_{\mu}\delta})-\mathcal{T}^b_{\epsilon_{m}\delta}R(u_{\epsilon_{m}\delta},v_{\epsilon_{m}\delta}))(\mathcal{T}^b_{\epsilon_\mu}w_{\epsilon_{\mu}\delta}-\mathcal{T}^b_{\epsilon_m}w_{\epsilon_{m}\delta})dxd\sigma_ydt\\
			-2k_d\int_{0}^{t}\int_{\Omega\times\Gamma}&(\psi_\delta(\mathcal{T}^b_{\epsilon_\mu}w_{\epsilon_{\mu}\delta})-\psi_\delta(\mathcal{T}^b_{\epsilon_m}w_{\epsilon_{m}\delta}))(\mathcal{T}^b_{\epsilon_\mu}w_{\epsilon_{\mu}\delta}-\mathcal{T}^b_{\epsilon_m}w_{\epsilon_{m}\delta})dxd\sigma_ydt.
		\end{align*}
		By the definition \eqref{eqn:MDU}, the function $\psi_\delta$ is Lipschitz that means
		\begin{align*}
			\|\mathcal{T}^b_{\epsilon_\mu}w_{\epsilon_{\mu}\delta}&(t)-\mathcal{T}^b_{\epsilon_m}w_{\epsilon_{m}\delta}(t)\|^2_{\Omega\times\Gamma}\le \|\mathcal{T}^b_{\epsilon_\mu}w_{\epsilon_{\mu}\delta}(0)-\mathcal{T}^b_{\epsilon_m}w_{\epsilon_{m}\delta}(0)\|^2_{\Omega\times\Gamma}\\
			&+k_d^2\int_{0}^{t}\|\mathcal{T}^b_{\epsilon_\mu}R(u_{\epsilon_{\mu}\delta},v_{\epsilon_{\mu}\delta})-\mathcal{T}^b_{\epsilon_{m}\delta}R(u_{\epsilon_{m}\delta},v_{\epsilon_{m}\delta})\|^2_{\Omega\times\Gamma}dt\\
			&+(1+2k_dL_{Lip})\int_{0}^{t}\|\mathcal{T}^b_{\epsilon_\mu}w_{\epsilon_{\mu}\delta}(\theta)-\mathcal{T}^b_{\epsilon_m}w_{\epsilon_{m}\delta}(\theta)\|^2_{\Omega\times\Gamma}d\theta.
		\end{align*} 
		By following the same steps as in Lemma \ref{lemma:R2S}, we can derive
		\begin{align}\label{eqn:NE6}
			\|\mathcal{T}^b_{\epsilon_\mu}w_{\epsilon_{\mu}\delta}(t)-\mathcal{T}^b_{\epsilon_m}w_{\epsilon_{m}\delta}(t)\|^2_{\Omega\times\Gamma}&\le C[({\epsilon_\mu}^n+{\epsilon_m}^n)+k_d^2T(\epsilon_\mu+\epsilon_m)]e^{(1+2k_dL_{Lip})T}\notag\\
			&\rightarrow 0 \text{  as   } \mu,m\rightarrow \infty.
		\end{align}
		Hence, $\eqref{eqn:NE6}$ establishing the strong convergence of $\{\mathcal{T}^b_{\epsilon}(w_{\epsilon\delta})\}$ in $L^2(S\times\Omega\times\Gamma)$ and suppose $\mathcal{T}^b_{\epsilon}(w_{\epsilon\delta})$ is strongly convergent to some $\zeta$ in $L^2(S\times\Omega\times\Gamma)$. Then, as an implication of Lemma \ref{thm:PU1}, we have $w_{\epsilon\delta}\overset{2}{\rightharpoonup} \zeta$ in $L^2(S\times\Omega\times\Gamma)$ but $w_{\epsilon\delta}\overset{2}{\rightharpoonup} w_\delta$ in $L^2(S\times\Omega\times\Gamma)$ by Lemma \ref{lemma:APB2}. That is to say that $\mathcal{T}^b_{\epsilon}(w_{\epsilon\delta})\rightarrow w_\delta$ in $L^2(S\times\Omega\times\Gamma)$. Now, as $\psi_\delta$ is continuous so $\psi_\delta(\mathcal{T}^b_{\epsilon}(w_{\epsilon\delta}))\rightarrow \psi_\delta(w_\delta)$ in $L^2(S\times\Omega\times\Gamma)$, i.e. $\psi_\delta(\mathcal{T}^b_{\epsilon}(w_{\epsilon\delta}))\overset{2}{\rightharpoonup}\psi_\delta(w_\delta)$ in $L^2(S\times\Omega\times\Gamma)$.
	\end{proof}
	\textbf{$\epsilon\rightarrow 0$ part: } Testing \eqref{eqn:M1RWF} with $\Phi_1=\phi_1(t,x)+\epsilon\psi_1(t,x,\frac{x}{\epsilon})$ where $\phi_1\in C^\infty_0(S\times\Omega)$ and $\psi_1\in C^\infty_0(S\times\Omega ; C^\infty_{per}(Y))$, we get
	\begin{align*}
		\underbrace{\int_{0}^{T}\int_{\Omega_\epsilon^p}\frac{\partial u_{\epsilon\delta}}{\partial t}[\phi_1(t,x)+\epsilon\psi_1(t,x,\frac{x}{\epsilon})]dxdt}_{I_{\delta1}}+\underbrace{\int_{0}^{T}\int_{\Omega_\epsilon^p} [D_1(x,\frac{x}{\epsilon})\nabla u_{\epsilon\delta}-\textbf{q}_\epsilon u_{\epsilon\delta}] [\nabla \phi_1+\epsilon\nabla_x \psi_1+\nabla_y\psi_1]dxdt}_{I_{\delta2}}\\
		+\underbrace{\epsilon\int_{0}^{T}\int_{\Gamma^*_\epsilon}\frac{\partial w_{\epsilon\delta}}{\partial t}[\phi_1(t,x)+\epsilon\psi_1(t,x,\frac{x}{\epsilon})]d\sigma_xdt}_{I_{\delta3}}=0.
	\end{align*}
	We are now going to pass the limit for each term individually.
	\begin{align*}
		\lim_{\epsilon\rightarrow 0} I_{\delta1}&=-\lim_{\epsilon\rightarrow 0}\int_{0}^{T}\int_{\Omega_\epsilon^p}u_{\epsilon\delta}(t,x)[\frac{\partial \phi_1}{\partial t}+\epsilon\frac{\partial \psi_1}{\partial t}]dxdt\\
		&=-\lim_{\epsilon\rightarrow 0}\int_{0}^{T}\int_{\Omega}\chi(\frac{x}{\epsilon})u_{\epsilon\delta}(t,x)[\frac{\partial \phi_1}{\partial t}+\epsilon\frac{\partial \psi_1}{\partial t}]dxdt=-\int_{0}^{T}\int_{\Omega}\int_{Y} \chi(y)u_\delta(t,x)\frac{\partial \phi_1}{\partial t}dxdydt,\\
		\lim_{\epsilon\rightarrow 0} I_{\delta2}&=\lim_{\epsilon\rightarrow 0}\int_{0}^{T}\int_{\Omega}\chi(\frac{x}{\epsilon})[D_1(x,\frac{x}{\epsilon})-\textbf{q}_\epsilon u_{\epsilon\delta}]\nabla u_{\epsilon\delta} [\nabla \phi_1+\epsilon\nabla \psi_1+\nabla_y\psi_1]dxdt\\
		&=\int_{0}^{T}\int_{\Omega}\int_{Y} \chi(y)[D_1(x,y)(\nabla u_\delta(t,x)+\nabla_y u_{\delta1}(t,x,y))-\textbf{q}u_\delta](\nabla \phi_1+\nabla_y \psi_1)dxdydt,\\
		\lim_{\epsilon\rightarrow 0} I_{\delta3}&=\lim_{\epsilon\rightarrow 0}\epsilon\int_{0}^{T}\int_{\Gamma_\epsilon^*}\frac{\partial w_{\epsilon\delta}}{\partial t}[\phi_1(t,x)+\epsilon\psi_1(t,x,\frac{x}{\epsilon})]d\sigma_xdt=\int_{0}^{T}\int_{\Omega}\int_{\Gamma}\frac{\partial w_\delta}{\partial t}\phi_1 dxd\sigma_ydt.
	\end{align*}
	Putting all together we obtain
	\begin{align}\label{eqn:HE2}
		-\int_{0}^{T}\int_{\Omega}u_\delta(t,x)\frac{\partial \phi_1}{\partial t}dxdt+\frac{1}{|Y^p|}\int_{0}^{T}\int_{\Omega}\int_{Y^p}& [D_1(x, y)(\nabla u_\delta+\nabla_y u_{\delta1})-\textbf{q}u_\delta](\nabla \phi_1+\nabla_y \psi_1)dxdydt\notag\\
		&+\int_{0}^{T}\int_{\Omega}\int_{\Gamma}\frac{1}{|Y^p|}\frac{\partial w_\delta}{\partial t}\phi_1 dxd\sigma_ydt=0.
	\end{align}
	Now, for $\phi_1\equiv 0$, \eqref{eqn:HE2} becomes
	\begin{align*}
		\int_{0}^{T}\int_{\Omega}\int_{Y^p} [D_1(x,y)(\nabla u_\delta(t,x)+\nabla_y u_{\delta1}(t,x,y))-\textbf{q}u_\delta]\nabla_y \psi_1dxdydt=0.
	\end{align*} 
	Depending on the choice of the function $u_{\delta1}$ given by $\eqref{eqn:NE7}$, we obtain $k_j(t,x,y)(j=1,2,...,n)$ is the $Y-$periodic solution of the cell problems \eqref{eqn:CE1}. Next, to get the homogenized equation inserting $\psi_1\equiv 0$ in \eqref{eqn:HE2} we are led to
	\begin{align*}
		\int_{0}^{T}\int_{\Omega}\frac{\partial u_\delta(t,x) }{\partial t}\phi_1dxdt+\int_{0}^{T}\int_{\Omega}\int_{Y^p} \frac{D_1(x,y)}{|Y^p|}\left(\delta_{ij}+\sum_{i,j=1}^{n}\frac{\partial k_j}{\partial y_i}\right)\frac{\partial u_\delta(t,x)}{\partial x_j}\nabla_x \phi_1dxdydt\notag\\
		-\frac{1}{|Y^p|}\int_{0}^{T}\int_{\Omega\times Y^p}[\textbf{q}u_\delta-D_1(x, y)\nabla_y k_0 u_\delta]\nabla_x \phi_1 dxdydt+\int_{0}^{T}\int_{\Omega} P_\delta(t,x)\phi_1 dxdt=0,
	\end{align*}
	which can be written as 
	\begin{align*}
		\int_{0}^{T}\int_{\Omega}\left[\frac{\partial u_\delta(t,x) }{\partial t}-\nabla.(A\nabla u_\delta-(\bar{\textbf{q}}-\tilde{\textbf{q}}_0)u_\delta)+P_\delta(t,x)\right]\phi_1dxdt=0,\text{ for all }\phi_1\in C^\infty_0(S\times\Omega),
	\end{align*}
	where
	\begin{align*}
		A=(a_{ij})_{1\le i,j \le n}=\int_{Y^p}\frac{D_1(x,y)}{|Y^p|}\left(\delta_{ij}+\sum_{i,j=1}^{n}\frac{\partial k_j}{\partial y_i}\right)dy, \quad P_\delta(t,x)=\int_{\Gamma}\frac{1}{|Y^p|}\frac{\partial w_\delta}{\partial t}\;d\sigma_y.
	\end{align*}
	Therefore, the strong form of the macroscopic equation for the mobile species $A$ is given by $\eqref{eqn:M1S2}$. Likewise replacing $\theta$ by $\Phi_2=\phi_2(t,x)+\epsilon\psi_2(t,x,\frac{x}{\epsilon})$, where $\phi_2\in C^\infty_0(S\times\Omega)$ and $\psi_2\in C^\infty_0(S\times\Omega ; C^\infty_{per}(Y))$ in \eqref{eqn:M2RWF} and repeating the same arguments as above we get $\eqref{eqn:M2S2}$. Finally, we choose the test function $\rho(t,x,\frac{x}{\epsilon})\in C^\infty_0(S\times\Omega ; C^\infty_{per}(Y))$ for the ODE \eqref{eqn:M12RWF} and obtain
	\begin{align*}
		\int_{0}^{T}\int_{\Gamma_\epsilon^*}\frac{\partial w_{\epsilon\delta}}{\partial t}\rho(t,x,\frac{x}{\epsilon})d\sigma_xdt=k_d\int_{0}^{T}\int_{\Gamma_\epsilon^*}(R(u_{\epsilon\delta},v_{\epsilon\delta})-\psi_\delta(w_{\epsilon\delta}))\rho(t,x,\frac{x}{\epsilon})d\sigma_xdt.
	\end{align*}
	Passing to the two-scale limit using Lemma \ref{lemma:APB2} and Lemma \ref{lemma:R2S1} yields
	\begin{align*}
		\int_{0}^{T}\int_{\Omega}\int_{\Gamma} \frac{\partial w_\delta}{\partial t}\rho(t,x,y)dxd\sigma_ydt=k_d\int_{0}^{T}\int_{\Omega}\int_{\Gamma}(R(u_\delta,v_\delta)-\psi_\delta(w_\delta))\rho(t,x,y)dxd\sigma_ydt,
	\end{align*}
	whose strong form is given by $\eqref{eqn:M12S2}$.
	\begin{lemma}
		For each $\delta>0$ the solutions $(u_\delta,v_\delta, w_\delta)$ of the problem $\eqref{eqn:M1S2}-\eqref{eqn:M12S2}$ satisfies the following estimate
		\begin{align}\label{eqn:ME3}
			&	\|u_\delta\|_{L^\infty(S; L^2(\Omega))}+\|u_\delta\|_{(\Omega)^T}+\|\nabla u_\delta\|_{(\Omega)^T}+\left\lVert \frac{\partial u_\delta}{\partial t}\right\rVert_{L^2(S;H^{1,2}(\Omega)^*)}+\|v_\delta\|_{L^\infty(S; L^2(\Omega))}+\|v_\delta\|_{(\Omega)^T}\notag\\
			&\quad\quad+\|\nabla v_\delta\|_{(\Omega)^T}+\left\lVert \frac{\partial v_\delta}{\partial t}\right\rVert_{L^2(S;H^{1,2}(\Omega)^*)}+\|w_\delta\|_{(\Omega\times\Gamma)^T}+\left\lVert\frac{\partial w_\delta}{\partial t}\right\rVert_{(\Omega\times\Gamma)^T}+\|P_\delta\|_{(\Omega)^T}\le C,
		\end{align}
		where the constant $C$ is independent of $\delta$.
	\end{lemma}
	\begin{proof}
		We test the ODE $\eqref{eqn:M12S2}$ with $w_\delta$ and using \eqref{eqn:RI} obtain the inequality
		\begin{align*}
			\|w_\delta(t)\|^2_{\Omega\times\Gamma}\le \|w_0\|^2_{\Omega\times\Gamma}+k_d^2(1+\frac{k}{4})^2T|\Omega||\Gamma|+\int_{0}^{t}\|w_\delta(\theta)\|^2_{\Omega\times\Gamma}d\theta.
		\end{align*}
		Gronwall's inequality gives
		\begin{align*}
			&\|w_\delta(t)\|^2_{\Omega\times\Gamma}\le M_w^{*}e^t,\text{ where } M_w^{*}=\|w_0\|^2_{\Omega\times\Gamma}+k_d^2(1+\frac{k}{4})^2T|\Omega||\Gamma|.\\
			\text{So, } &\|w_\delta\|_{(\Omega\times\Gamma)^T}=\left(\int_{0}^{T}	\|w_\delta(s)\|^2_{\Omega\times\Gamma}ds\right)^{\frac{1}{2}}\le C.
		\end{align*}
		We now multiply the ODE \eqref{eqn:M12S2} by $\frac{\partial w_\delta}{\partial t}$ and integrate over $S\times\Omega\times\Gamma$ to get
		\begin{align*}
			\left\lVert \frac{\partial w_\delta}{\partial t}\right\rVert_{(\Omega\times\Gamma)^T}\le k_d(1+\frac{k}{4})\sqrt{T|\Omega||\Gamma|}=C.
		\end{align*} 
		Also,
		\begin{align*}
			|P_\delta(t,x)|^2=\frac{k_d^2}{|Y^p|^2}\left| \int_{\Gamma}(R(u_\delta,v_\delta)-\psi_\delta(w_\delta))d\sigma_y\right|^2\le \frac{k_d^2|\Gamma|^2}{|Y^p|^2}(1+\frac{k}{4})^2.
		\end{align*}
		Hence, $\|P_\delta\|_{(\Omega)^T}\le C$. Choosing $u_\delta$ as a test function in $\eqref{eqn:M1S2}$ and using the ellipticity of the matrix $A$ we deduce that
		\begin{align}\label{eqn:ME1}
			\|u_\delta(t)\|^2_{\Omega}&+2(\alpha_1-\gamma C-\gamma_1)\|\nabla u_\delta\|^2_{(\Omega)^T}\le\|u_0\|^2_{\Omega}+ \|P_\delta\|^2_{(\Omega)^T}\notag\\
			&+\frac{1}{2\gamma}(\|d\|^2_{(\partial\Omega_{in})^T}+\|e\|^2_{(\partial\Omega_{out})^T})+(1+\frac{\|\tilde{\textbf{q}}_0\|^2_{L^\infty(S\times\Omega)}}{2\gamma_1}+2\gamma C)\int_{0}^{t}\|u_\delta(\theta)\|^2_{\Omega}d\theta,
		\end{align}
		since $\int_{0}^{t}\int_{\Omega}\bar{\textbf{q}}.\nabla u_\delta u_\delta dxdt =0$ by \eqref{eqn:v}.
		We choose $\gamma=\frac{\alpha_1}{2C}$ and $\gamma_1=\frac{\alpha_1}{2}$, then by Gronwall's inequality, we obtain $\|u_\delta(t)\|^2_{\Omega}\le M_u^*e^{(1+\alpha_1)t}\implies \|u_\delta\|_{(\Omega)^T}\le C.$
		By $\eqref{eqn:ME1}$, for $\gamma=\frac{\alpha_1}{4C}$ and $\gamma_1=\frac{\alpha_1}{4}$ we also get the estimate $\|\nabla u_\delta\|_{(\Omega)^T}\le C$. Again, $\|u_\delta\|_{L^\infty(S; L^2(\Omega))}=\underset{t\in S}{\text{esssup}}\;\|u_\delta(t)\|_{L^2(\Omega)}\le CT.$
		We calculate the PDE $\eqref{eqn:M1S2}$ as 
		\begin{align*}
			&\left| \langle \frac{\partial u_\delta}{\partial t},\phi\rangle_{\Omega}\right|\le \left\{M_1 \|\nabla u_\delta\|_{\Omega}+\|\tilde{\textbf{q}}_0\|_{L^\infty(\Omega)}\|u_\delta\|_{\Omega}\|\nabla u_\delta\|_{\Omega}+\|P_\delta\|_{\Omega}+C(\|d\|_{\partial\Omega_{in}}+\|e\|_{\partial\Omega_{out}})\right\}\|\phi\|_{H^{1,2}(\Omega)}\\
			\implies &\left\lVert \frac{\partial u_\delta}{\partial t}\right\rVert_{H^{1,2}(\Omega)^{*}}\le C.
		\end{align*}
		Next, we integrate w.r.t. $t$ to obtain
		\begin{align*}
			\left\lVert \frac{\partial u_\delta}{\partial t}\right\rVert_{L^2(S;H^{1,2}(\Omega)^{*})}\le C.
		\end{align*}
		Following the same line of arguments as before we have for type-II mobile species
		\begin{align*}
			\|v_\delta\|_{L^\infty(S; L^2(\Omega))}+\|v_\delta\|_{(\Omega)^T}+\|\nabla v_\delta\|_{(\Omega)^T}+\left\lVert \frac{\partial v_\delta}{\partial t}\right\rVert_{L^2(S;H^{1,2}(\Omega)^{*})}\le C.
		\end{align*}
		Finally, summarizing all the estimates we get \eqref{eqn:ME3}.
	\end{proof}
	
	\begin{lemma}\label{lemma:ELE}
		The solutions $u_\delta, v_\delta$ of the PDEs $\eqref{eqn:M1S2}$ and $\eqref{eqn:M2S2}$ are strongly converges to $u$ and $v$ respectively in $L^2(S;L^2(\Omega))$.
	\end{lemma}
	\begin{proof}
		This is a consequence of Lemma \ref{lemma:1} and the estimate \eqref{eqn:ME3}.
	\end{proof}
	We compute all the a-priori bounds needed to pass $\delta\rightarrow 0$. Following the idea of \cite{van2004crystal}, we define the function $z_\delta\in L^\infty(S\times\Omega\times\Gamma)$ as
	\begin{align*}
		z_\delta(t,x,y)=\psi_\delta (w_\delta(t,x,y)) \text{ for a.e. } (t,x,y)\in S\times\Omega\times\Gamma.
	\end{align*}
	
	Then, from \eqref{eqn:ME3} and Lemma \ref{lemma:ELE}, we have the following convergences: 
	
	\begin{itemize}
		\item $u_\delta \rightharpoonup u \text{ in } L^2(S; H^{1,2}(\Omega)), \quad \frac{\partial u_\delta}{\partial t} \rightharpoonup \frac{\partial u}{\partial t} \text{ in } L^2(S; H^{1,2}(\Omega)^*)$,
		\item $v_\delta \rightharpoonup v \text{ in } L^2(S; H^{1,2}(\Omega)), \quad \frac{\partial v_\delta}{\partial t} \rightharpoonup \frac{\partial v}{\partial t} \text{ in } L^2(S; H^{1,2}(\Omega)^*)$,
		\item $u_\delta\rightarrow u$ in $L^2(S;L^2(\Omega))$, $\quad\;$ $v_\delta\rightarrow v$ in $L^2(S;L^2(\Omega))$,
		\item $w_\delta \rightharpoonup w \text{ in } L^2(S\times\Omega\times\Gamma), \quad \frac{\partial w_\delta}{\partial t} \rightharpoonup \frac{\partial w}{\partial t} \text{ in } L^2(S\times\Omega\times\Gamma)$,
		\item $P_\delta=\int_{\Gamma}\frac{\partial w_\delta}{\partial t}d\sigma_y\rightharpoonup \int_{\Gamma} \frac{\partial w_\delta}{\partial t}d\sigma_y=P$ in $L^2(S\times\Omega)$,
		\item $z_\delta\overset{*}{\rightharpoonup}z$ in $L^\infty(S\times\Omega\times\Gamma)$.
	\end{itemize}
	\begin{lemma}
		The rate term $R(u_\delta,v_\delta)\rightarrow R(u,v)$ in $L^2(S\times\Omega)$.
	\end{lemma}
	\begin{proof}
		Lipschitz continuity of $R$ gives
		\begin{align*}
			\|R(u_\delta,v_\delta)-R(u,v)\|_{L^2(S\times\Omega)}&\le L_R \{\|u_\delta-u\|_{L^2(S\times\Omega)}+\|v_\delta-v\|_{L^2(S\times\Omega)}\}\\
			&\rightarrow 0 \text{  as   }\delta\rightarrow 0 \text{ by Lemma } \ref{lemma:ELE}.
		\end{align*}
	\end{proof}
	\textbf{$\delta\rightarrow 0$ part: }
	We pass the regularization parameter $\delta$ to zero and immediately see that the limits are the weak solutions to the problem $\eqref{eqn:M1S}-\eqref{eqn:M12S}$. Only we need to give special attention to establish $z\in\psi(w)$ part. This is shown in Theorem 2.21 of \cite{van2004crystal}.\\
	\textit{\underline{Uniqueness of the macroscopic model}} : Now, we prove the uniqueness of the problem $(\mathcal{P})$. On the contrary, let $(\textbf{q}_1, u_1,v_1,w_1,z_1)$ and $(\textbf{q}_2, u_2,v_2,w_2,z_2)$ be the solutions of $(\mathcal{P})$. We define $\textbf{Q}=\textbf{q}_1-\textbf{q}_2\ge 0,      U=u_1-u_2\ge 0, V=v_1-v_2\ge0, W=w_1-w_2\ge0$ and $Z=z_1-z_2\ge 0$. Clearly, $\textbf{Q}(0, x, y)=0,  U(0,x)=V(0,x)=W(0,x,y)=Z(0,x,y)=0$ for all $x\in\Omega$ and $y\in\Gamma$. The weak formulations of the resulting equations in terms of the differences are
	\begin{align}
		&\langle \frac{\partial \textbf{Q}}{\partial t}, \psi\rangle_{(\Omega\times Y^p)^T}+\mu \langle \nabla_y \textbf{Q} : \nabla_y \psi\rangle_{(\Omega\times Y^p)^T}=0 \label{eqn:quw}\\
		&\langle \frac{\partial U}{\partial t}, \phi\rangle_{(\Omega)^T}+\langle A\nabla U-\bar{\textbf{q}}U, \nabla \phi\rangle_{(\Omega)^T}+\frac{1}{|Y^p|}\int_{0}^{T}\int_{\Omega\times\Gamma}\frac{\partial W}{\partial t}\phi dxd\sigma_ydt=0 \label{eqn:M1UWF},\\
		&\langle \frac{\partial V}{\partial t}, \theta\rangle_{(\Omega)^T}+\langle B\nabla V-\bar{\textbf{q}}V, \nabla \theta\rangle_{(\Omega)^T}+\frac{1}{|Y^p|}\int_{0}^{T}\int_{\Omega\times\Gamma}\frac{\partial W}{\partial t}\theta dxd\sigma_ydt=0\label{eqn:M2UWF},\\
		&\langle \frac{\partial W}{\partial t}, \eta\rangle_{(\Omega\times\Gamma)^T}=k_d\langle R(u_1,v_1)-R(u_2,v_2)-Z, \eta\rangle_{(\Omega\times\Gamma)^T} \label{eqn:M12UWF},
	\end{align}
	for all $\psi\in (L^2(S\times\Omega; H^{1,2}_{div}(Y)))^n$ and $(\phi,\theta,\eta)\in L^2(S; H^{1,2}(\Omega))\times L^2(S; H^{1,2}(\Omega))\times L^2(S\times\Omega\times\Gamma)).$
	We replace $\psi$ by $\chi_{(0,t)}\textbf{Q}$ in \eqref{eqn:quw} and deduce that
	\begin{align*}
		\|\textbf{Q}(t)\|^2_{\Omega\times Y^p}+2\mu \|\nabla_y \textbf{Q}\|^2_{(\Omega\times Y^p)^t}=0\implies \textbf{Q}(t)=0 \text{ for a.e. } t\in S.
	\end{align*}
	Testing the equation $\eqref{eqn:M12UWF}$ with $\eta=\chi_{(0,t)}W(t)$ and taking into account that $R$ is Lipschitz we have
	\begin{align*}
		\|W(t)\|^2_{\Omega\times\Gamma}&= 2k_dL_R\int_{0}^{t}\int_{\Omega\times\Gamma}(U+V)Wdxd\sigma_ydt-\underbrace{2k_d\int_{0}^{t}\int_{\Omega\times\Gamma}ZW dxd\sigma_ydt}_{\ge0}\\
		&\le k_d^2L_R^2\int_{0}^{t}\|(U+V)(s)\|^2_{\Omega\times\Gamma}ds+\int_{0}^{t}\|W(s)\|^2_{\Omega\times\Gamma}ds.
	\end{align*}
	As a consequence of Gronwall's inequality, we get the estimate
	\begin{align}\label{eqn:E1}
		\|W(t)\|^2_{\Omega\times\Gamma}\le C(T)\int_{0}^{t}\|(U+V)(s)\|^2_{\Omega\times\Gamma}ds.
	\end{align}
	We add the equations \eqref{eqn:M1UWF} and \eqref{eqn:M2UWF} and take $\theta=\chi_{(0,t)}\phi$ to obtain
	\begin{align*}
		\langle \frac{\partial (U+V)}{\partial t}, \phi\rangle_{(\Omega)^t}+\langle (A\nabla U+B\nabla V)-\bar{\textbf{q}}(U+V), \nabla \phi\rangle_{(\Omega)^t}+\frac{2}{|Y^p|}\int_{0}^{t}\int_{\Omega\times\Gamma}\frac{\partial W}{\partial t}\phi dxd\sigma_ydt=0.
	\end{align*} 
	Coercivity of the matrices $A$ and $B$ in combination with  $\phi=(U+V)$ gives that
	\begin{align*}
		\|(U+V)(t)\|^2_{\Omega}+2\alpha_1 \|\nabla (U+V)\|^2_{(\Omega)^t}\le \frac{4}{|Y^p|}\left|\int_{0}^{t}\int_{\Omega\times\Gamma}\frac{\partial W}{\partial t}(U+V) dxd\sigma_ydt\right|,
	\end{align*}
	since \eqref{eqn:qs1} and \eqref{eqn:v} implies that $\langle \bar{\textbf{q}}(U+V), \nabla (U+V) \rangle =0$.
	The r.h.s. can be simplified to 
	\begin{align*}
		\left|\int_{0}^{t}\int_{\Omega\times\Gamma}\frac{\partial W}{\partial t}(U+V) dxd\sigma_ydt\right| &\le k_d\int_{0}^{t}\int_{\Omega\times\Gamma}
		|R(u_1,v_1)-R(u_2,v_2)-\underbrace{Z}_{\ge 0}||U+V|dxd\sigma_ydt\\
		&\le k_d L_R\int_{0}^{t}\int_{\Omega\times\Gamma} |U+V|^2dxd\sigma_ydt,\text{  since   } R \text{ is Lipschitz }\\
		&=k_dL_R|\Gamma|\int_{0}^{t}\|U+V\|^2_{L^2(\Omega)},\text{ as } U,V \text{ are independent of }y.
	\end{align*}
	Hence, we can write
	\begin{align*}
		\|(U+V)(t)\|^2_{L^2(\Omega)}\le \frac{4k_dL_R|\Gamma|}{|Y^p|}\int_{0}^{t}\|(U+V)(s)\|^2_{L^2(\Omega)}ds.
	\end{align*}
	Gronwall's inequality and the positivity of $U$ and $V$ ensures that
	\begin{align*}
		U(t)=0 \text{  and  } V(t)=0 \text{ for a.e. } t\in S.
	\end{align*}
	Consequently, by $\eqref{eqn:E1}$ $W(t)=0$ for a.e. $t\in S$. This concludes the proof of uniqueness.
	\section{Conclusion}
	We discussed a reaction-diffusion-advection system coupled with a crystal dissolution and precipitation model and the unsteady Stokes equation in a porous medium. We analyzed the system with one mineral and two mobile species having space-dependent different diffusion coefficients. We tackle the multivalued dissolution rate term by introducing a regularization parameter $\delta$. We have employed Rothe's method to address the question of the existence of a unique global-in-time weak solution. We applied a different version of extension lemma to pass the homogenization limit $\epsilon\rightarrow0$. We also have shown that both the repeated limits exist and are equal. However, our results should encourage further work in this area. The issue of global time existence of a unique weak solution of a system of multi-species diffusion-reaction equation with different diffusion coefficients is still a challenging open problem. Much remains to be investigated in this context and our future works will address those questions.

	\section{Data availability}
	
	Data sharing not applicable to this article as no datasets were generated or analysed during the current study.
	\newpage
	
	\bibliographystyle{acm}
	\bibliography{cite1}
	\newpage
	\section*{Appendix}
	\begin{lemma}[Theorem 2.1 of \cite{meirmanov2011compactness}]\label{lemma:1}
		Let $(c_\epsilon)$ be a bounded sequence in $L^\infty(S;L^2(\Omega))\cap L^2(S;H^{1,2}(\Omega))$ and weakly convergent in $L^2(S;L^2(\Omega))\cap L^2(S;H^{1,2}(\Omega))$ to a function $c$. Suppose further that the sequence $\left(\chi_\epsilon\frac{\partial c_\epsilon}{\partial t}\right)$ is bounded in $L^2(S ;H^{1,2}(\Omega)^*)$. Then the sequence $(c_\epsilon)$ is strongly convergent to the function $c$ in $L^2(S; L^2(\Omega))$.
	\end{lemma}
	\begin{lemma}[cf. Proposition 1.3. in \cite{showalter1996monotone}]
		Let $B$ be a Banach space and $B_0$, $B_1$ be reflexive spaces with $B_0\subset B\subset B_1$. Suppose further that $B_0\hookrightarrow\hookrightarrow B\hookrightarrow B_1$. For $1<r, s<\infty$ and $0<T<\infty$ define $X:=\{u\in L^r(S; B_0) : \partial_t u\in L^s(S; B_1)\}.$ Then $X\hookrightarrow\hookrightarrow L^r(S; B).$
	\end{lemma}
	\subsection{Two Scale Convergence}
	\begin{defin}(cf. \cite{allaire1992homogenization, lukkassen2002two})
		Let $\{u_\epsilon\}_{\epsilon>0}$ be a sequence of functions in $L^2(S\times\Omega) (\Omega\subset\mathbb{R}^n$ open) where $\epsilon$ being a sequence of strictly positive numbers that tends to zero. $\{u_\epsilon\}$ is said to be two-scale converges to a unique function $u(t,x,y)$ if and only if for any $\psi\in C^\infty_0(S\times\Omega; C^\infty_{per}(Y))$ we have
		\begin{align}\label{eqn:2S}
			\lim_{\epsilon\rightarrow 0}\int_{0}^{T}\int_{\Omega}u_\epsilon(t,x)\psi(t,x,\frac{x}{\epsilon})dxdt=\int_{0}^{T}\int_{\Omega}\int_{Y} u(t,x,y)\psi(t,x,y)dxdydt.
		\end{align}
	\end{defin}
	We denote \eqref{eqn:2S} by $u_\epsilon \overset{2}{\rightharpoonup} u$. Now, we quote the following theorems whose proofs can be found in \cite{allaire1992homogenization, neuss1992homogenization, lukkassen2002two}.
	\begin{lemma}\label{thm:TS1}
		Let $u_\epsilon$ strongly convergent to $u\in L^2(S\times\Omega)$ then $u_\epsilon \overset{2}{\rightharpoonup} u_0$ where $u_0(t,x,y)=u(t,x)$.
	\end{lemma}
	\begin{lemma}\label{thm:TS2}
		(i) For each bounded sequence $\{u_\epsilon\}\in L^2(S\times\Omega)$ there exists a subsequence which two scale converges to $u(t,x,y)\in L^2(S\times\Omega\times Y)$.\\
		(ii) Let $\{u_\epsilon\}$ be a bounded sequence in $L^2(S; H^{1,2}(\Omega))$, which converges weakly to a limit function $u(t,x)\in L^2(S; H^{1,2}(\Omega))$. Then there exists $u_1(t,x,y)\in L^2(S\times\Omega ; H^{1,2}_{per}(Y)/\mathbb{R})$ such that upto a subsequence  $u_\epsilon(t,x) \overset{2}{\rightharpoonup} u(t,x)$ and $\nabla u_\epsilon \overset{2}{\rightharpoonup} \nabla_x u+\nabla_y u_1$.\\
		(iii) Let $\{u_\epsilon\}$ and $\{\epsilon\nabla u_\epsilon\}$ be two bounded sequences in $L^2(S\times\Omega)$ and $L^2(S\times\Omega)^n$. Then there exists a function $u\in L^2(S\times\Omega; H^{1,2}_{per}(Y))$ such that $u_\epsilon\overset{2}{\rightharpoonup} u$ and $\epsilon\nabla_x u_\epsilon \overset{2}{\rightharpoonup}\nabla_y u$, respectively.
	\end{lemma}
	\begin{defin}(Two-scale convergence for $\epsilon$-periodic hypersurfaces \cite{allairetwo}) A sequence of functions $\{u_\epsilon\}$ in $L^2(S\times\Gamma^*_\epsilon)$ is said to be two-scale converges to a limit $u\in L^2(S\times\Omega\times\Gamma)$ if and only if for any $\psi\in C(S\times\bar{\Omega} ; C_{per}(Y))$ we have
		\begin{align}\label{eqn:2SS}
			\lim_{\epsilon\rightarrow 0}\epsilon\int_{0}^{T}\int_{\Gamma^*_\epsilon} u_\epsilon(t,x)\psi(t,x,\frac{x}{\epsilon})d\sigma_xdt=  \int_{0}^{T}\int_{\Omega}\int_{\Gamma} u(t,x,y)\psi(t,x,y)dxd\sigma_ydt.
		\end{align}
	\end{defin}
	\begin{lemma}\label{thm:TS3}
		Let $\{u_\epsilon\}$ be a sequence in $L^2(S\times\Gamma^*_\epsilon)$ such that $ \epsilon\int_{0}^{T}\int_{\Gamma^*_\epsilon}|u_\epsilon(t,x)|^2d\sigma_xdt\le C,$
		where C is a constant independent of $\epsilon$. Then there exists a subsequence $u_\epsilon$(still denoted by same symbol) and a two scale limit $u\in L^2(S\times\Omega\times\Gamma)$ such that
		\begin{align*}
			\lim_{\epsilon\rightarrow 0} \epsilon\int_{0}^{T}\int_{\Gamma^*_\epsilon} u_\epsilon(t,x)\psi(t,x,\frac{x}{\epsilon})d\sigma_xdt=\int_{0}^{T}\int_{\Omega}\int_{\Gamma} u(t,x,y)\psi(t,x,y)dxd\sigma_ydt,
		\end{align*}
		for any $\psi(t,x,y)\in C(S\times\bar{\Omega} ; C_{per}(Y))$.
	\end{lemma}

	\subsection{Periodic Unfolding Operator}
	\begin{defin}\label{dfn:U}(Periodic Unfolding Operator) (cf. \cite{cioranescu2002periodic, cioranescu2008periodic}) Let $u_\epsilon\in L^r(S\times\Omega)$, where $1\le r\le\infty$. We define the boundary unfolding operator $\mathcal{T}_\varepsilon : L^r(S\times\Omega)\rightarrow L^r(S\times\Omega\times Y)$ such that
		\begin{align}\label{eqn:U}
			\mathcal{T}_\varepsilon(\phi)(t,x,y)=	
			\begin{cases}
				\phi(t,\varepsilon\left[\frac{x}{\varepsilon}\right]_Y+\varepsilon y)\text{ for all } (t,x,y)\in S\times\Omega\times Y\\
				0 \text{ otherwise, }
			\end{cases},
		\end{align}
		where $\left[\frac{x}{\varepsilon}\right]_Y$ denotes the unique integer combination $\sum_{j=1}^{n}k_je_j$ of the periods such that $x-\left[\frac{x}{\varepsilon}\right]_Y$ belongs to $Y$.
	\end{defin}
	\begin{lemma}\label{lemma:uo}
		Let $1<p<\infty$. Then the operator $\mathcal{T}_\varepsilon$, defined in the Definition \ref{dfn:U} satisfy the following properties:
		\begin{itemize}
			\item[$(i)$] Let $u_\epsilon, v_\epsilon\in L^p(S\times\Omega)$, then $\mathcal{T}_\varepsilon(u_\epsilon v_\epsilon)=\mathcal{T}_\varepsilon(u_\epsilon)\mathcal{T}_\varepsilon(v_\epsilon)$.
			\item[$(ii)$] Let $u_\epsilon\in L^p(S\times\Omega)$, then $\|\mathcal{T}_\varepsilon(u_\epsilon)\|_{L^2(S\times\Omega\times Y)}=\|u_\epsilon\|_{L^2(S\times\Omega)}$.
			\item[$(iii)$] Let $\{u_\epsilon\}$  be a bounded sequence in $L^p(S\times\Omega)$. Then the following statements hold:\\
			$(a)$ If $u_\epsilon\rightarrow u$ in $L^p(S\times\Omega)$, then $ \mathcal{T}_\varepsilon(u_\epsilon)\rightarrow u$ in  $L^p(S\times\Omega\times Y)$.\\
			$(b)$ If $u_\epsilon \overset{2}{\rightharpoonup} u$ in $L^p(S\times\Omega\times Y)$, then $\mathcal{T}_\varepsilon(u_\epsilon) \rightharpoonup u$ in $L^p(S\times\Omega\times Y)$.
		\end{itemize}
	\end{lemma}
	\begin{proof}
		See in \cite{cioranescu2002periodic, cioranescu2008periodic, cioranescu2006periodic} for proof details.
	\end{proof}
	\subsection{Boundary Unfolding Operator}
	\begin{defin}
		(Time dependent boundary unfolding operator) For $\epsilon>0$, the boundary unfolding of a function $\phi\in L^2(S\times\Gamma^*_\epsilon)$ is defined by
		\begin{align}\label{eqn:BU}
			\mathcal{T}^b_\varepsilon(\phi)(t,x,y)=\phi(t,\varepsilon\left[\frac{x}{\varepsilon}\right]_Y+\varepsilon y)\text{ for all } (t,x,y)\in S\times\Omega\times\Gamma,
		\end{align}
		where $\left[\frac{x}{\varepsilon}\right]_Y$ denotes the unique integer combination $\sum_{j=1}^{n}k_je_j$ of the periods such that $x-\left[\frac{x}{\varepsilon}\right]_Y$ belongs to $Y$. The most notable point is that the oscillation due to the perforations are shifted into the second variable $y$ which belongs to a fixed domain $\Gamma$.
	\end{defin}
	\begin{lemma}[Lemma 4.6 in \cite{marciniak2008derivation}]\label{thm:PU1}
		If $u_\epsilon$ converges two-scale to $u$ and $\mathcal{T}^b_\epsilon u_\epsilon$ converges weakly to $u_0\in L^2(S\times\Omega ; L^2_{per}(\Gamma))$ then $u=u_0$ a.e. in $S\times\Omega\times\Gamma$.
	\end{lemma}
	\begin{lemma}[Lemma 17 of \cite{fatima2014sulfate}]\label{thm:PU2}
		If $\phi\in L^2(S\times\Gamma^*_\epsilon)$ then the following identity holds
		\begin{align*}
			\frac{1}{|Y|}\|\mathcal{T}^b_\epsilon\phi\|^2_{L^2(S\times\Omega\times\Gamma)}=\epsilon\|\phi\|^2_{L^2(S\times\Gamma^*_\epsilon)}.
		\end{align*}
	\end{lemma}
	\begin{lemma}[Proposition 5.2 of \cite{cioranescu2006periodic}, \cite{cioranescu2012periodic}]\label{thm:PU3}
		(i) $\mathcal{T}^b_\epsilon$ is a linear operator.\\
		(ii) Let $\{\phi_\epsilon\}$ be a sequence in $L^2(S\times\Omega)$ such that $\phi_\epsilon\rightarrow\phi$ strongly in $L^2(S\times\Omega)$ then $\mathcal{T}^b_\epsilon(\phi_\epsilon)\rightarrow \phi$ strongly in $L^2(S\times\Omega\times\Gamma)$. 
	\end{lemma}
	\subsection{Trace Theorems, Restriction Theorem and Extension Theorem}
	\begin{lemma}[Theorem 1 of Section 5.5 in \cite{evans1998partial}]\label{thm:T1}
		Let $1\le p<\infty$ and $\Omega\subset \mathbb{R}^n$ be a bounded domain with sufficiently smooth boundary $\partial\Omega$. Then there exists a bounded linear operator $T : H^{1,2}(\Omega)\rightarrow L^2(\partial \Omega)$ such that
		\begin{align*}
			&(i) Tu := u|_{\partial \Omega} \text{ if } u\in H^{1,2}(\Omega)\cap C(\bar{\Omega}), (ii) \|Tu\|_{L^2(\partial \Omega)}\le C\|u\|_{H^{1,2}(\Omega)},
		\end{align*}
		for each  $u\in H^{1,2}(\Omega)$, where $C$ depends on $p$ and $\Omega$ but independent of $u$.
	\end{lemma}
	The sobolev space $H^\beta(\Omega)$ as a completion of $C^\infty(\Omega)$ is a Hilbert space equipped with a norm
	\begin{align*}
		\|\phi\|_{H^\beta(\Omega)}=\|\phi\|_{H^{[\beta]}(\Omega)}+\left(\int_{\Omega}\int_{\Omega}\frac{|\phi(x)-\phi(y)|^2}{|x-y|^{n+2(\beta-[\beta])}}dxdy\right)^{\frac{1}{2}}
	\end{align*}
	and the embedding $H^\beta(\Omega)\hookrightarrow L^2(\Omega)$ is continuous (cf. Theorem 7.57 of \cite{kufner1977function}).\\
	To estimate the boundary integral we frequently use the following trace inequality for $\epsilon$-dependent hypersurfaces $\Gamma_\epsilon^*$: For $\phi_\epsilon\in H^{1,2}(\Omega_\epsilon^p)$ there exists a constant $C$, which is independent of $\epsilon$ such that
	\begin{align}\label{eqn:TI1}
		\epsilon \int_{\Gamma^*_\epsilon}|\phi_\epsilon|^2d\sigma_x\le C\left[\int_{\Omega_\epsilon^p}|\phi_\epsilon|^2dx+\epsilon^2\int_{\Omega_\epsilon^p}|\nabla\phi_\epsilon|^2dx\right].
	\end{align}
	The proof of \eqref{eqn:TI1} is given in Lemma 3 of \cite{hornung1991diffusion} and in Lemma 2.7.2 of \cite{neuss1992homogenization}. Following the proof of the trace theorem of Theorem 6.3 of \cite{auchmuty2014sharp}, we will get a different form of the trace inequality:  For every $\phi_\epsilon\in H^{1,2}(\Omega_\epsilon^p)$ there exist positive constants $C^*$ and $\bar{C}$, independent of $\epsilon$ such that
	\begin{align}\label{eqn:TI2}
		\|\phi_\epsilon\|^2_{\Gamma^*_\epsilon}\le C^*\|\phi_\epsilon\|^2_{\Omega_\epsilon^p}+2\bar{C}\|\phi_\epsilon\|_{\Omega_\epsilon^p}\|\nabla \phi_\epsilon\|_{\Omega_\epsilon^p}.
	\end{align}
	For a function $\phi_\epsilon\in H^\beta(\Omega_\epsilon^p)$ with $\beta\in(\frac{1}{2},1)$, the inequality \eqref{eqn:TI1} refines into
	\begin{align}\label{eqn:TI3}
		\epsilon \int_{\Gamma^*_\epsilon}|\phi_\epsilon|^2d\sigma_x\le C_0\left[\int_{\Omega_\epsilon^p}|\phi_\epsilon|^2dx+\epsilon^{2\beta}\int_{\Omega_\epsilon^p}\int_{\Omega_\epsilon^p}\frac{|\phi_\epsilon(x)-\phi_\epsilon(y)|}{|x-y|^{n+2\beta}}dxdy\right],
	\end{align}
	where $C_0$ is a constant independent of $\epsilon$. See Lemma 4.2 of \cite{marciniak2008derivation} for the proof of the inequality \eqref{eqn:TI3}.
	\begin{lemma}[Restriction Theorem cf. \cite{allaire1992homogenization, neuss1992homogenization}]
		There exists a linear restriction operator $R_\epsilon : H^{1,2}_0(\Omega)^n \rightarrow H^{1,2}_0(\Omega^p_\epsilon)^n$ such that $R_\epsilon u(x)=u(x)|_{\Omega_\epsilon^p}$ for $u\in H^{1,2}_0(\Omega)^n$ and $\nabla. R_\epsilon u=0$ if $\nabla.u=0$. Furthermore, the restriction satisfies the following bound
		\begin{align}
			\|R_\epsilon u\|_{L^2(\Omega_\epsilon^p)}+\epsilon\|\nabla R_\epsilon u\|_{L^2(\Omega_\epsilon^p)}\le C\left(\|u\|_{L^2(\Omega)}+\epsilon \|\nabla u\|_{L^2(\Omega)}\right),
		\end{align}
		with an $\epsilon$ independent constant $C$.
		
	\end{lemma}
	\begin{lemma}[Extension Operator]\label{lemma:EL1}
		\begin{enumerate}
			\item For $v_{\epsilon\delta}, u_{\epsilon\delta}\in H^{1,2}(Y^p)$ there exists an extension $\tilde{v}_{\epsilon\delta}, \tilde{u}_{\epsilon\delta}$ to $Y$ such that 
			\begin{enumerate}
				\item $\|\tilde{v}_{\epsilon\delta}\|_Y\le C \|v_{\epsilon\delta}\|_{Y^p}$ and $\|\nabla\tilde{v}_{\epsilon\delta}\|_Y\le C \|\nabla v_{\epsilon\delta}\|_{Y^p}$,
				\item $\|\tilde{u}_{\epsilon\delta}\|_Y\le C \|u_{\epsilon\delta}\|_{Y^p}$ and $\|\nabla\tilde{u}_{\epsilon\delta}\|_Y\le C \|\nabla u_{\epsilon\delta}\|_{Y^p}$.
			\end{enumerate}    
			\item For $v_{\epsilon\delta}, u_{\epsilon\delta}\in H^{1,2}(\Omega_\epsilon^p)$ there exists an extension $\tilde{v}_{\epsilon\delta}, \tilde{u}_{\epsilon\delta}$ to $\Omega$ such that
			\begin{enumerate}
				\item $\|\tilde{v}_{\epsilon\delta}\|_{H^{1,2}(\Omega)}\le C \|v_{\epsilon\delta}\|_{H^{1,2}(\Omega^p_\epsilon)},$     
				\item $\|\tilde{u}_{\epsilon\delta}\|_{H^{1,2}(\Omega)}\le C \|u_{\epsilon\delta}\|_{H^{1,2}(\Omega^p_\epsilon)}.$
			\end{enumerate}
		\end{enumerate}
	\end{lemma}
	\begin{rem}
		The proof of the above lemma can be found in \cite{hornung1991diffusion,neuss1992homogenization}, where a reflection operator is constructed on the interface of the solid matrix $Y^s$ and then the rest is done via a simple scaling argument. Although, the choice of the reflection operator is not unique, we may also construct it as
		\begin{gather*}
			\bar{v}_{\epsilon\delta} \colon S\times\Gamma\times(-\delta, \delta) \rightarrow S\times U \text{ (nbd of $\Gamma=\partial Y^s$) } \text{ by } \\
			\bar{v}_{\epsilon\delta}(t,z,y_n)=\begin{cases}
				v_{\epsilon\delta}(t,z,y_n) \; \text{ if } y_n\geq0,\\
				3v_{\epsilon\delta}(t,z,-\frac{y_n}{3})-2v_{\epsilon\delta}(t,z,-y_n) \; \text{ if } y_n\leq0.
			\end{cases}
			\text{  where  } z=(y_1,y_2,\cdots,y_{n-1})
		\end{gather*} 
	\end{rem}
	\subsection{Proof of the Lemma \ref{lemma: 1}}
	The equation \eqref{eqn:R2} for $i=1$ and $\theta=\frac{v_1-v_0}{h}$ takes the form
	\begin{align*}
		\left\lVert\frac{v_1-v_0}{h}\right\rVert^2_{\Omega_\epsilon^p}+b(v_1,\frac{v_1-v_0}{h})=\langle f(v_0), \frac{v_1-v_0}{h}\rangle_{\Gamma^*_\epsilon}-\langle g_0, \frac{v_1-v_0}{h}\rangle_{\partial \Omega_{in}}-\langle h_0, \frac{v_1-v_0}{h}\rangle_{\partial \Omega_{out}},
	\end{align*}
	which can be estimated as
	\begin{align*}
		\left\lVert\frac{v_1-v_0}{h}\right\rVert^2_{\Omega_\epsilon^p}+\frac{\alpha}{h}\|\nabla (v_1-v_0)\|^2_{\Omega_\epsilon^p}&\leq \langle f(v_0), \frac{v_1-v_0}{h}\rangle_{\Gamma^*_\epsilon}-\langle g_0, \frac{v_1-v_0}{h}\rangle_{\partial \Omega_{in}}\\
		&-\langle h_0, \frac{v_1-v_0}{h}\rangle_{\partial \Omega_{out}}-b(v_0,\frac{v_1-v_0}{h}).
	\end{align*}
	The terms on the r.h.s. can be calculated as follows:
	\begin{align*}
		&\langle f(v_0), \frac{v_1-v_0}{h}\rangle_{\Gamma^*_\epsilon}\leq C_6+C\gamma_1\left\lVert\frac{v_1-v_0}{h}\right\rVert^2_{\Omega_\epsilon^p}+\frac{C\gamma_1}{h^2}\|\nabla(v_1-v_0)\|^2_{\Omega_\epsilon^p},\\
		&\langle g_0, \frac{v_1-v_0}{h}\rangle_{\partial \Omega_{in}}+\langle h_0, \frac{v_1-v_0}{h}\rangle_{\partial \Omega_{out}}\le\frac{1}{4\gamma_2}\left(\|g_0\|^2_{\partial\Omega_{in}}+\|h_0\|^2_{\partial\Omega_{out}}\right)+\gamma_2\left\lVert\frac{v_1-v_0}{h}\right\rVert^2_{\partial\Omega},\\
		&b(v_0,\frac{v_1-v_0}{h})\le\frac{M^2}{4\gamma_3}\|\nabla v_0\|^2_{\Omega_\epsilon^p}+\frac{\gamma_3}{h^2}\|\nabla(v_1-v_0)\|^2_{\Omega_\epsilon^p}+\frac{1}{4\gamma_4}\|\textbf{q}_0v_0\|_{\Omega_\epsilon^p}+\frac{\gamma_4}{h^2}\|\nabla(v_1-v_0)\|^2_{\Omega_\epsilon^p},
	\end{align*}
	where $\gamma_1, \gamma_2, \gamma_3, \gamma_4$ are the \textit{Young's} inequality constants, $C$ is the trace constant and $C_6=\frac{k_d^2}{4\gamma_1}(1+\frac{k}{4})^2\frac{|\Omega||\Gamma|}{|Y|}$. Putting it all together, we arrive at
	\begin{align*}
		(1-C\gamma_1-C\gamma_2)&\left\lVert\frac{v_1-v_0}{h}\right\rVert^2_{\Omega_\epsilon^p}+\frac{1}{h^2}\left(\alpha h-C\gamma_1-C\gamma_2-\gamma_3-\gamma_4\right)\|\nabla (v_1-v_0)\|^2_{\Omega_\epsilon^p}\\
		&\le C_6+\frac{1}{4\gamma_2}\left(\|g_0\|^2_{\partial\Omega_{in}}+\|h_0\|^2_{\partial\Omega_{out}}\right)+\frac{M^2}{4\gamma_3}\|\nabla v_0\|^2_{\Omega_\epsilon^p}+\frac{1}{4\gamma_4}\|\textbf{q}_0v_0\|_{\Omega_\epsilon^p}=C_7.
	\end{align*}
	We now use \eqref{eqn:qu} and choose $h$ small enough such that $\alpha h<1$. Then for $\gamma_1=\frac{\alpha h}{4C}=\gamma_2$ and $\gamma_3=\frac{\alpha h}{8}=\gamma_4$ we obtain
	\begin{align*}
		\left\lVert\frac{v_1-v_0}{h}\right\rVert^2_{\Omega_\epsilon^p}+\frac{1}{2h^2}\|\nabla (v_1-v_0)\|^2_{\Omega_\epsilon^p}\le C.
	\end{align*}
	For $i\geq2$, we subtract \eqref{eqn:R2} for $i=j$ from \eqref{eqn:R2} for $i=j-1$. After that, we use \eqref{eqn:v} and test the difference with $\theta=\frac{v_j-v_{j-1}}{h}$ to get
	\begin{align}
		\left\lVert\frac{v_j-v_{j-1}}{h}\right\rVert^2_{\Omega_\epsilon^p}+\frac{\alpha}{h}\|\nabla (v_j-v_{j-1})\|^2_{\Omega_\epsilon^p}
		&\leq\langle f(v_{j-1})-f(v_{j-2}),\frac{v_j-v_{j-1}}{h}\rangle_{\Gamma^*_\epsilon}-\langle g_{j-1}-g_{j-2},\frac{v_j-v_{j-1}}{h}\rangle_{\partial\Omega_{in}}\notag\\
		&-\langle h_{j-1}-h_{j-2},\frac{v_j-v_{j-1}}{h}\rangle_{\partial\Omega_{out}}+
		\langle \frac{v_{j-1}-v_{j-2}}{h}, \frac{v_j-v_{j-1}}{h}\rangle_{\Omega_\epsilon^p}. \label{eqn:R1}
	\end{align}
	Therefore, we can simplify \eqref{eqn:R1} as
	\begin{align*}
		(1-C\gamma_1-C\gamma_2-\gamma_3)\left\lVert\frac{v_j-v_{j-1}}{h}\right\rVert^2_{\Omega_\epsilon^p}&+\frac{1}{h^2}(\alpha h-C\gamma_1-C\gamma_2)\|\nabla (v_j-v_{j-1})\|^2_{\Omega_\epsilon^p}\le C_{8}+\frac{1}{4\gamma_3}\left\lVert\frac{v_{j-1}-v_{j-2}}{h}\right\rVert^2_{\Omega_\epsilon^p}.
	\end{align*}
	Again, for $\alpha h<1$ we derive
	\begin{align*}
		\left\lVert\frac{v_j-v_{j-1}}{h}\right\rVert^2_{\Omega_\epsilon^p}+\frac{2}{h^2}\|\nabla (v_j-v_{j-1})\|^2_{\Omega_\epsilon^p}\le C_9 + C_{10} \left\lVert\frac{v_{j-1}-v_{j-2}}{h}\right\rVert^2_{\Omega_\epsilon^p}.
	\end{align*}
	If we take $\sigma_j=\left\lVert\frac{v_j-v_{j-1}}{h}\right\rVert^2_{\Omega_\epsilon^p}+\frac{2}{h^2}\|\nabla (v_j-v_{j-1})\|^2_{\Omega_\epsilon^p}$ then, we have $\sigma_j\leq C_9+C_{10}\sigma_{j-1}\Rightarrow \sigma_j\leq C$ by Gronwall's inequality.
\end{document}